\documentclass[10pt]{amsart}

\usepackage{amsthm, amsfonts, amssymb, color}
 \usepackage{mathrsfs}
\usepackage{amsmath}
 \usepackage{amstext, amsxtra}
  \usepackage{txfonts}
 \usepackage[colorlinks, linkcolor=black, citecolor=blue, pagebackref, hypertexnames=false]{hyperref}

 \allowdisplaybreaks
\newtheorem{thm}{\hskip\parindent {Theorem}}[section]
\newtheorem{lem}{\hskip\parindent {Lemma}}[section]
\newtheorem{defi}{\hskip\parindent {Definition}}[section]

\newtheorem{rem}{\hspace\parindent{Remark}}[section]
\newtheorem{pro}{\hskip\parindent{Proposition}}[section]

\renewcommand\Re{\operatorname{Re}}
\renewcommand\Im{\operatorname{Im}}

\title[Soliton-potential interaction]{
Soliton-potential interactions for nonlinear Schr\"odinger equation in $\mathbb{R}^3$ }
\author{Qingquan Deng, \ Avy Soffer \ and Xiaohua Yao}

\address {Qingquan Deng, Department of Mathematics and  Hubei Province Key Laboratory of Mathematical Physics,
 Central China Normal University, Wuhan, 430079, P.R. China}
\email{dengq@mail.ccnu.edu.cn}
\address{ Avy Soffer, Department of Mathematics, Rutgers University,
Piscataway, 08854-8019, USA}
\email{soffer@math.rutgers.edu}
\address{Xiaohua Yao, Department of Mathematics and  Hubei Province Key Laboratory of Mathematical Physics,
 Central China Normal University, Wuhan, 430079, P.R. China}
\email{yaoxiaohua@mail.ccnu.edu.cn }
\date{\today}
\subjclass[2000]{ 35Q35; 37K40}
\keywords{Endpoint Strichartz estimates; Soliton-potential interaction; Asymptotic stability.}

\begin{document}

\maketitle
\begin{abstract}
In this work we mainly consider the dynamics and scattering of a narrow soliton of NLS equation with a potential in $\mathbb{R}^3$, where the asymptotic state of the system can be far from the initial state in parameter space. Specifically, if we let a narrow soliton state with initial velocity $\upsilon_{0}$ to interact with an extra potential $V(x)$,
then the velocity $\upsilon_{+}$ of outgoing solitary wave in infinite time will in general be very different from $\upsilon_{0}$. In contrast to our present work, previous works proved that the soliton is asymptotically stable under the assumption that $\upsilon_{+}$ stays close to $\upsilon_{0}$ in a certain manner.
\end{abstract}

\tableofcontents

\section{Introduction}
\setcounter{equation}{0}
Consider the following nonlinear Schr\"{o}dinger (NLS) equation  in $\mathbb{R}^{3}$,
\begin{eqnarray}\label{Sch eq}
&&i\partial_{t}\psi=-\frac{1}{2}\Delta\psi+V\psi-F_{\epsilon}(|\psi|^{2})\psi\nonumber\\
&&\psi(\cdot,0)=\psi_{0}.
\end{eqnarray}
Here $\epsilon$ is a small positive number which will be fixed later and $V$ is a positive smooth bump function supported in unit ball $B(0,1)$ in $\mathbb{R}^{3}$. The nonlinear term is of form
\begin{eqnarray}\label{nonlinear}
F_{\epsilon}(|s|^{2})=\frac{f(s)}{\theta\epsilon^{-\frac{2}{p}r}+f(s)}|s|^{p}
\end{eqnarray}
where $1<p<\frac{4}{3}$, $f$ is a homogenous function of degree $r>0$  such that $\frac{7}{3}<r+p\leq4$ and $\theta$ is a fixed small positive number. The initial data $\psi_{0}$ is a narrow soliton of form
\begin{eqnarray}\label{initial}
e^{i\bar{\upsilon}_{0}\cdot x+i\gamma_{0}}\epsilon^{-\frac{2}{p}}\phi(\frac{1}{\epsilon}(x-\bar{a}_{0}),\epsilon^{-2}\mu_{0})
\end{eqnarray}
with given parameters $\{\bar{a}_{0},\bar{\upsilon}_{0},\gamma_{0},\mu_{0}\}$ of order $\mathcal{O}(1)$ and  $\phi=\phi(\cdot,\mu)$, the ground state satisfying
\begin{eqnarray}\label{ground}
-\frac{1}{2}\Delta\phi-F(|\phi|^{2})\phi=-\mu\phi,
\end{eqnarray}
we use $F(|s|^{2})$ to denote $F_{1}(|s|^{2})$ in (\ref{nonlinear}). We refer to Rodnianski, Schlag and Soffer\cite{RSS2} for the existence of such solution to (\ref{ground}).

In this work we mainly focus on the dynamics and scattering of a  narrow soliton of equation (\ref{Sch eq}) in a setting which is not asymptotically stable. That is, we allow the asymptotic state of the system to be far from the initial state in parameter space.
Specifically, if we let soliton state of form (\ref{initial}) with initial velocity $\bar{\upsilon}_{0}$ to interact with an extra potential $V(x)$,
then the velocity $\upsilon_{+}$ of outgoing solitary wave will in general be very different. Previous works ( see e.g. \cite{CucMa1}, \cite{GNP}, \cite{Per1}, \cite{RSS2} ) required that $\upsilon_{+}$ stay close to $\bar{\upsilon}_{0}$ in a certain manner
and then the soliton is asymptotically stable, which in contrast to our work, we allow $\upsilon_{+}$ to be very different from $\bar{\upsilon}_{0}$.

To introduce our main result, we first
make change of variables
\begin{eqnarray}\label{change}
u(x,t)=\epsilon^{\frac{2}{p}}\psi(\epsilon x,\epsilon^{2}t),
\end{eqnarray}
then equivalently,  the equation (\ref{Sch eq}) describing a narrow soliton interaction with a normal potential,  becomes the following NLS equation
\begin{eqnarray}\label{Sch eq-1}
&&i\partial_{t}u=-\frac{1}{2}\Delta u+\epsilon^{2}V(\epsilon\cdot)u-F(|u|^{2})u\nonumber\\
&&u(\cdot,0)=u_{0}=e^{i\bar{\upsilon}_{0}\epsilon\cdot x+i\gamma_{0}}\phi(x-\frac{\bar{a}_{0}}{\epsilon},\mu_{0}),
\end{eqnarray}
which describes actually a normal soliton interaction with a flat potential.
Set $\sigma=\{a,\upsilon,\gamma,\mu\}$ and $\sigma_{0}=\{a_{0},\upsilon_{0},\gamma_{0},\mu_{0}\}$ where $a_{0}=\frac{\bar{a}_{0}}{\epsilon}$ and
$\upsilon_{0}=\epsilon\bar{\upsilon}_{0}$.
Let $V_{\epsilon}=\epsilon^{2}V(\epsilon\cdot)$, we introduce the time-independent one-soliton linearized Hamiltonian
\begin{eqnarray}\label{eq-213}
\mathscr{H}(\sigma)
&=&\mathscr{H}_{0}+\mathcal{V}_{1\epsilon}+\mathcal{V}_{2}(\sigma)
\end{eqnarray}
where
\begin{eqnarray}\label{eq-213'}
 \mathcal {V}_{1\epsilon}=\left(
  \begin{array}{cc}
    V_{\epsilon}& 0 \\
  0 &  -V_{\epsilon}\\
  \end{array}
  \right)\nonumber
\end{eqnarray}
\begin{eqnarray}\label{eq-214'}
\mathcal{V}_{2}(\sigma)=\left(
  \begin{array}{cc}
 \mu-F(|\phi(\sigma)|^{2})-F'(|\phi(\sigma)|^{2})\phi(\sigma)^{2}& -F'(|\phi(\sigma)|^{2})\phi(\sigma)^{2} \\
    F'(|\phi(\sigma)|^{2})\phi(\sigma)^{2} & -\mu+F(|\phi(\sigma)|^{2})+F'(|\phi(\sigma)|^{2})\phi(\sigma)^{2} \\
  \end{array}
  \right)\nonumber
\end{eqnarray}
 and $\phi(\sigma)=\phi(x,\mu)$ satisfying (\ref{ground}).
Denote by \begin{eqnarray}\label{eq-213}
\mathscr{H}_{1}
&=&\mathscr{H}_{0}+\mathcal{V}_{1\epsilon}
\end{eqnarray}
and
\begin{eqnarray}\label{eq-214}
\mathscr{H}_{2}(\sigma)
&=&\mathscr{H}_{0}+\mathcal{V}_{2}(\sigma),
\end{eqnarray}
We introduce some spectral assumptions on $\mathscr{H}_{2}(\sigma)$.
\vskip0.3cm
\textbf{Spectral assumptions:} For $|\sigma-\sigma_{0}|<c$ with some constant $c>0$, one has

(i) 0 is the only point of the discrete spectrum of $\mathscr{H}_{2}(\sigma)$
and the dimension of the corresponding root space is 8.

(ii) There are no embedded eigenvalues in ${\rm spec}_{ess}(\mathscr{H}_{2}(\sigma))=(-\infty,-\mu]\cup[\mu,+\infty)$  and the points $\pm\mu$ are not resonances.
\vskip0.3cm
\begin{rem}{\rm $\pm\mu$ are said to be a resonances of $\mathscr{H}_{2}(\sigma)$ if there exists solutions $\psi_{\pm}$ to equation $(\mathscr{H}_{2}(\sigma)\pm\mu)\psi_{\pm}=0$ such that $\langle x\rangle^{-s}\psi_{\pm}\in L^{2}$ for any $s>1/2$.}
\end{rem}
\vskip0.5cm

\begin{thm}\label{thm-main}
Let $\epsilon\ll 1$ and $F$ as in (\ref{ground}) with a small fixed parameter $\theta>0$. Assume that the linearized operator $\mathscr{H}_{2}(\sigma)$ defined by (\ref{eq-214}) satisfy the spectral assumptions for all $\sigma\in \mathbb{R}^{8}$ satisfying $|\sigma-\sigma_{0}|<c$ for some $c>0$.
 Then the solution to NLS equation (\ref{Sch eq-1}) exists globally and
there exist $\sigma_{+}\in \mathbb{R}^{8}$ and $u_{+}\in H^{1}$ such that
\begin{eqnarray}\label{ex1-3}
\|u-w(x,t;\sigma_{+})-e^{it\Delta}u_{+}\|_{H^{1}}\rightarrow 0,\ \ \ t\rightarrow +\infty.
\end{eqnarray}
\end{thm}

\vskip0.5cm
The proof of Theorem \ref{thm-main} is actually a direct corollary of Theorems \ref{thm-1} and \ref{thm-1-1} given in Section \ref{sec-2}. In order to understand some key-points of our result, we would would like to give more explanations by the following two steps:

${(\rm i)}$ We first solve the soliton dynamics up to large but finite time $T_{0}$ (see Theorem \ref{thm-1} of Subsection \ref{sec2-2}), so that after time $T_{0}$, the new outgoing solitary wave is far from the support of potential $V_{\epsilon}$ and its velocity is pointing away from $V_{\epsilon}$. Since one can use the classical dynamics to describe the leading order behavior of  such soliton-potential interaction (see Fr\"{o}hlich, Gustafson, Jonsson and Sigal\cite{FJGS}), we follow the same idea of \cite{FJGS} to conclude that after an appropriate time $T_{0}$, the above condition on the outgoing solitary wave applies.

${(\rm ii)}$ It follows from the finite time results that the solution is the state of the solitary wave with some parameter $\sigma_{T_{0}}$ plus a perturbation $r_{T_{0}}$, as well as the solitary wave moves away from potential. Next, we begin with a new system with initial conditions on $(\sigma_{T_{0}}, r_{T_{0}})$,  and establish the long time behavior of the new system, i.e. Theorem \ref{thm-1-1} of Subsection \ref{sec-2-3}, which is an optimal version of previous works on asymptotic stability (see e.g. \cite{CucMa1}, \cite{GNP}, \cite{Per1}, \cite{RSS2} and references therein).

To implement this strategy, we encounter and overcome new technical difficulties. To do step {(\rm ii)}, we have to know that besides the solitary wave moving away from potential (with nonzero speed), the radiation part $r_{T_{0}}$ of the solution is small. For this we assume that in the original NLS equation (\ref{Sch eq}), the soliton (\ref{initial}) is narrow which by scaling is equivalent to the potential is small and flat in new NLS equation (\ref{Sch eq-1}). By applying the method of \cite{FJGS}, one has after time $T_{0}$ of order $\mathcal{O}(\epsilon^{-2})$, the $H^{1}$ norm of the radiation will be of order $\mathcal{O}(\epsilon^{2_{-}})$. To proceed, we linearize around soliton gives a time-dependent matrix charge transfer Hamiltonian depending on parameter vector $\sigma(t)$ and $\epsilon$. Moreover, the initial data is only small in $H^{1}$, but in general is large in weighted spaces such as $L^{2}(\langle x\rangle^{s}dx)$ and $L^{1}$. Then we would investigate the optimal Strichartz estimates in $\epsilon$  for the time and $\epsilon$ dependent charge transfer Hamiltonian to take advantage of the smallness of radiation in $H^{1}$ norm. Furthermore, since only Strichartz estimates is applied, we may have to use $L^{1}$ norm for the derivative of parameter vector $\dot{\sigma}(t)$.
This is in contrast with the proof in \cite{RSS2} for example, where the authors had the $L^{p}$ decay estimates in hand and hence the localization at least in $L^{p}$ for radiation is needed (as well as smoothness in higher Sobolev norms) and the decay bounds for $\dot{\sigma}(t)$ is also used. It should be noted that the Strichartz estimates (homogeneous and inhomogeneous) for the following time and $\epsilon$ dependent matrix charge transfer Hamiltonian
\begin{eqnarray}
\mathscr{H}(t,\widetilde{\sigma}(t))=\mathscr{H}_{0}+\mathcal {V}_{1\epsilon}+\mathcal {V}_{2}(t,\widetilde{\sigma}(t))\nonumber
\end{eqnarray}
would be the most important part of this work. We prove these estimates are uniformly for $\epsilon\ll1$.

\vskip0.5cm
\begin{thm}\label{thm-2}
Let $Z(t,x)$ solve the equation
\begin{eqnarray}\label{eq-301}
&&i\partial_{t}Z=\mathscr{H}(t,\widetilde{\sigma}(t))Z+F\nonumber\\
&&Z(\cdot,0)=Z_{0},
\end{eqnarray}
where the matrix charge transfer Hamiltonian $\mathscr{H}(t,\widetilde{\sigma}(t))$ satisfies separation and spectral assumptions. Assume that the bootstrap assumption (\ref{eq-221}) holds for $\epsilon\ll1$ and $Z$ satisfies
\begin{eqnarray}\label{eq-302}
\|P_{b}(t)Z(t)\|_{L^{2}_{t}L^{6}_{x}}\lesssim B,
\end{eqnarray}
with some constant $B$. Then for all admissible pairs $(p,q)$ and $(\tilde{p},\tilde{q})$
\begin{eqnarray}\label{eq-303}
\big\|Z\big\|_{L^{p}_{t}L^{q}_{x}}\lesssim \|Z_{0}\|_{L^{2}}+\|F\|_{L^{\tilde{p}'}_{t}L^{\tilde{q}'}_{x}}+B.
\end{eqnarray}
Moreover, the constants in both estimates (\ref{eq-302}) and (\ref{eq-303}) are independent of $\epsilon$.
\end{thm}

\vskip0.3cm
One can see Sections \ref{sec-4-1} and \ref{sec-5} for proof details. For $\epsilon=1$, the charge transfer models has been extensively studied. Scattering theory and the asymptotic completeness was proved by Graf \cite{Gr}, W\"{u}ller \cite{Wu}, Yajima \cite{Ya1} and Zielinski \cite{Zi}. The next significant step is made by Rodinianski, Schlag and Soffer \cite{RSS1} and Cai \cite{Cai}, who proved the point-wise decay estimates. Recently,
 the Strichartz estimates  has been proved in Chen \cite{Chen}, Deng, Soffer and Yao \cite{DSY} and partially in Cuccagna and Maeda \cite{CucMa1}. The main idea used previously is to deduce the Strichartz estimates from local decay estimates. In \cite{DSY}, the authors used the following logic:
$$\rm {Decay\ estimates\Rightarrow Kato{-}Jensen  \Rightarrow Local\ decay\Rightarrow\ Strichartz\ estimates},$$
whereas in \cite{Chen} and \cite{CucMa1}, they used channel decomposition and proved local decay directly. If we apply the same argument as the one in \cite{DSY}, it is necessary to trace down the $\epsilon$-dependence in each step, which is way to complicated. On the other hand, the procedures in \cite{Chen} and \cite{CucMa1} would lead to $\epsilon$-dependent in a bad way, in fact, the bound in  (\ref{eq-303}) blows up as $\epsilon$ goes to zero. In our work, we still reduce Theorem \ref{thm-2} to proving local decay and then by using $L^{\frac{6}{5},1}\rightarrow L^{6,\infty}$ estimates for the solutions of linear and nonlinear Schr\"{o}dinger equation with one potential (see \cite{Bec3}) to prove for solution $Z(x,t)$ to equation (\ref{eq-301}) in Theorem \ref{thm-2},
\begin{eqnarray}
\big\|\mathcal {V}^{1/2}_{1\epsilon}Z(t)\big\|^{2}_{L^{2}_{t}L^{2}_{x}}+\big\|\mathcal {V}^{1/2}_{2}(t)Z(t)\big\|_{L^{2}_{t}L^{2}_{x}}\lesssim \|Z_{0}\|_{L^{2}_{x}}+\|F\|_{L^{\tilde{p}'}_{t}L^{\tilde{q}'}_{x}}+B,\nonumber
\end{eqnarray}
where $(p,q)$ and $(\tilde{p},\tilde{q})$ be admissible pairs and the constant above is $\epsilon$-independent. Moreover, $V_{\epsilon}$ is obtained by scaling, it is small in $L^{p}$ for $p<\frac{3}{2}$, large in $L^{p}$ for $p>\frac{3}{2}$ when $\epsilon\ll1$ and  invariant in $L^{\frac{3}{2}}$. Thus one cannot just deal with it as a small perturbation to get the desire estimates as in \cite{Bec3}. In fact, $L^{\frac{3}{2}}$ is to some extent considered as the critical space for $L^{p}$ decay estimates even for scalar and one potential Schr\"{o}dinger flow (see for example \cite{Be-G} and \cite{GB}).
\vskip0.3cm
 Finally, let us review some of the history of soliton-potential interaction of NLS equation. Fr\"{o}hlich, Gustafson, Jonsson and Sigal in \cite{FJGS} considered finite time results  for soliton interacting with a flat potential (equivalently, narrow soliton interacting with normal potential), they proved for large but finite time, the soliton moves along almost the classical trajectory and only radiates small energy. Their results has been improved later in one dimension by Homler and Zworski \cite{HM}. Also, similar results for soliton interacting with a flat time-dependent potential are obtained by  Salem \cite{AS}. As far as the long time behavior, Perelman \cite{Per5} considered a slow varying soliton   hits a potential in one dimension and show that such soliton would split into two parts when time $t\rightarrow\pm \infty$, which is totally different from our problem. Cuccagna and Maeda \cite{CucMa1} proved that the ground states are asymptotically stable if it interact with non-trapping potential weakly. There still exist some other significant works and we will not list them all, one can see for example \cite{S-S}, \cite{Bam1}, \cite{BJ}, \cite{DH}, \cite{FJGS-1}, \cite{GHW}, \cite{GNP}, \cite{HM1}, \cite{HMZ1}, \cite{HMZ2}, \cite{SZ} and references therein. Our work is also related to the analysis of multi-soliton problem, see \cite{AFS}, \cite{CL}, \cite{MM1}, \cite{MMT}, \cite{Per2}, \cite{Per1}, \cite{Per4}, \cite{RSS2}.


\section{The analysis of soliton-potential interactions }\label{sec-2}
\setcounter{equation}{0}
\subsection{Some spectral results}\label{sec2-1}
In this subsection, we will introduce results related to the spectrum of Hamiltonians $\mathscr{H}_{1}$ and $\mathscr{H}_{2}(\sigma)$.
 Notice that $V_{\epsilon}$ is positive and compactly supported, there is no eigenvalue for $\mathscr{H}_{1}$, which leaves us to consider the spectral properties of $\mathscr{H}_{2}(\sigma)$.  The continuous spectrum of $\mathscr{H}_{2}(\sigma)$ would be $(-\infty,\mu]\cup[\mu,+\infty)$. Additionally, ${\rm spec}(\mathscr{H}_{2}(\sigma))\subset \mathbb{R}$ and may have finite and finite dimensional point spectrum. Moreover, it has been proved in \cite{RSS1} that if $\phi(x,\mu)$ is the solution of (\ref{ground}), the convexity condition
  \begin{eqnarray}
\big\langle \partial_{\mu}\phi(\cdot,\mu),\phi(\cdot,\mu)\big\rangle>0\nonumber
\end{eqnarray}
holds for $|\mu-\mu_{0}|<c$.

 Zero is always a eigenvalue and admits a generalized eigenspace. It has 4 eigenvectors
\begin{eqnarray}\label{ex3-1}
\eta_{1}(x;\sigma)=\left(
\begin{array}{c}
\phi(\sigma)\\
-\phi(\sigma)
\end{array}
\right),
\ \ \eta_{j}(x;\sigma)=-i\left(
\begin{array}{c}
\partial_{j-2}\phi(\sigma)\\
\partial_{j-2}\phi(\sigma)
\end{array}
\right),\ \ 3\leq j\leq5,
 \end{eqnarray}
 and 4 generalized eigenvectors
 \begin{eqnarray}\label{ex3-2}
\eta_{2}(x;\sigma)=-i\left(
\begin{array}{c}
\partial_{\mu}\phi(\sigma)\\
\partial_{\mu}\phi(\sigma)
\end{array}
\right),
\ \ \eta_{j}(x;\sigma)=\left(
\begin{array}{c}
x_{j-5}\phi(\sigma)\\
-x_{j-5}\phi(\sigma)
\end{array}
\right),\ \ 6\leq j\leq8.
 \end{eqnarray}
 One can easily see that
 \begin{eqnarray}\label{ex3-4}
 \mathscr{H}_{2}(\sigma)\eta_{2}=i\eta_{1}\ \ {\rm and} \ \ \mathscr{H}_{2}(\sigma)\eta_{j+3}=-i\eta_{j}, \ \ 3\leq j\leq5.
 \end{eqnarray}

\begin{pro}\label{pro-21}Let $F$ be defined as in (\ref{nonlinear}) with $\epsilon=1$ and $\mathscr{H}_{2}(\sigma)$ be defined by (\ref{eq-214}). Assume that the above spectral assumption holds. Then



(I) Let $L(\sigma)={\rm ker}(\mathscr{H}_{2}(\sigma)^{2})$ and $L^{\ast}(\sigma)={\rm ker}(\mathscr{H}^{\ast}_{2}(\sigma)^{2})$, one has
\begin{eqnarray}
L^{2}(\mathbb{R}^{3})\times L^{2}(\mathbb{R}^{3})=L(\sigma)+L^{\ast}(\sigma)^{\perp}.\nonumber
\end{eqnarray}
The space $L^{\ast}(\sigma)$ has dimension 8 and can be expanded by $\xi_{j}(x,\sigma)$ $j=1,2,\ldots,8,$ where
\begin{eqnarray}\label{ex3-3}
\xi_{j}(x;\sigma)=\left(
\begin{array}{c}
\phi(\sigma)\\
\phi(\sigma)
\end{array}
\right),
\ \ \xi_{j}(x;\sigma)=\left(
\begin{array}{c}
i\partial_{j-2}\phi(\sigma)\\
-i\partial_{j-2}\phi(\sigma)
\end{array}
\right),\ \ 3\leq j\leq5,
\end{eqnarray}
and
 \begin{eqnarray}\label{ex3-2}
\xi_{2}(x;\sigma)=\left(
\begin{array}{c}
i\partial_{\mu}\phi(\sigma)\\
-i\partial_{\mu}\phi(\sigma)
\end{array}
\right),
\ \ \xi_{j}(x;\sigma)=\left(
\begin{array}{c}
x_{j-5}\phi(\sigma)\\
x_{j-5}\phi(\sigma)
\end{array}
\right),\ \ 6\leq j\leq8.
 \end{eqnarray}
with
\begin{eqnarray}
&&\mathscr{H}^{\ast}_{2}(\sigma)\xi_{1}(\cdot;\sigma)=0,\nonumber\\
&&\mathscr{H}^{\ast}_{2}(\sigma)\xi_{2}(\cdot;\sigma)=-i\xi_{1}(\cdot;\sigma),\nonumber\\
&&\mathscr{H}^{\ast}_{2}(\sigma)\xi_{j}(\cdot;\sigma)=0, \ \ 3\leq j\leq5,\nonumber\\
&&\mathscr{H}^{\ast}_{2}(\sigma)\xi_{j}(\cdot;\sigma)=i\xi_{j-3}(\cdot;\sigma),\ \   6\leq j\leq8.\nonumber
\end{eqnarray}
Moreover, there is a natural isomorphism
\begin{eqnarray}\label{ex3-6}
\mathcal{J}=\left(
  \begin{array}{cc}
    0 & 1 \\
    -1 & 0 \\
  \end{array}
  \right)
\end{eqnarray}
between $L(\sigma)$ and $L^{\ast}(\sigma)$ and
\begin{eqnarray}\label{ex3-7}
\eta_{j}=\mathcal{J}\xi_{j},\ \ \ 1\leq j \leq8.
\end{eqnarray}

(II) Let $P_{c}(\sigma)$ denote the projection onto $L^{\ast}(\sigma)^{\perp}$ and set $P_{b}(\sigma)=I-P_{c}(\sigma)$, let $\eta_{j}$ and $\xi_{j}$ are defined by (\ref{ex3-1})-(\ref{ex3-2}) and (\ref{ex3-3}), respectively. Then $$\mathscr{H}_{2}(\sigma)P_{c}(\sigma)=P_{c}(\sigma)\mathscr{H}_{2}(\sigma)$$
and
\begin{eqnarray}\label{eq-216}
P_{b}(\sigma)f=\frac{1}{n_{1}}\big(\eta_{1}\langle f,\xi_{2}\rangle+\eta_{2}\langle f,\xi_{1}\rangle\big)+\sum_{\ell=3}^{5}\frac{1}{n_{\ell}}\big(\eta_{\ell}\langle f,\xi_{\ell+3}\rangle+\eta_{\ell+3}\langle f,\xi_{\ell}\rangle\big)
\end{eqnarray}
where  $n_{1}=\langle\eta_{1},\xi_{2}\rangle$ and
$n_{\ell}=\langle\eta_{\ell}, \xi_{\ell+3}\rangle$ $(3\leq\ell\leq5)$.

(III) The linear stability property for  $\mathscr{H}_{2}(\sigma)$ holds. That is,
\begin{eqnarray}
\sup_{t\in \mathbb {R}}\big\|e^{it\mathscr{H}_{2}(\sigma)}P_{c}(\sigma)\big\|<\infty.\nonumber
\end{eqnarray}
\end{pro}

\begin{proof}One can see the  proof for $(I)$ and $(III)$ in \cite[section 12]{RSS2} and references therein. As for the statement $(II)$, we refer to \cite{Per1} for its proof.
\end{proof}

\subsection {Interactions on finite time $(0, T_\epsilon]$}\label{sec2-2}
We first consider equation (\ref{Sch eq-1}) for finite time, that the soliton-potential interaction happens.
Notice that $\phi$ in initial data is the groundstate satisfying (\ref{ground}) and $\epsilon$ is small enough, applying similar argument as in \cite{FJGS}, one could find
 solution to equation (\ref{Sch eq-1}) which will stay close to a solitary wave of form
\begin{eqnarray}\label{eq-1}
\phi(x,t;\sigma)=e^{i\upsilon\cdot x+i\gamma}\phi(x-a,\mu),
\end{eqnarray}
where $\sigma:=\{a,\upsilon,\gamma,\mu\}$, $a=\upsilon t+a_{0}$ and $\gamma=\mu t-\frac{|\upsilon|^{2}}{2}t+\gamma_{0}$ with $\gamma_{0}\in [0,2\pi)$, $a_{0}\in \mathbb{R}$.
Specifically, we have the following results which will be proved in the next section.
\begin{thm}\label{thm-1}
Assume $F$ is given by (\ref{ground}) and $\epsilon\ll 1$. Let $I_{0}$ be any closed bounded interval in $(0,\infty)$. Then there is a constant $C>0$, independent of $\epsilon$ but possibly depend on $I_{0}$ such that for $2\leq\delta\leq3$ and times $0<t<C\min\big\{\epsilon^{-\delta},\ \epsilon^{-8+2\delta}\big\}$, the solution to equation (\ref{Sch eq-1}) with initial data for some parameter
\begin{eqnarray}
\sigma_{0}=\{a_{0},\upsilon_{0},\gamma_{0},\mu_{0}\}\in \mathbb{R}^{n}\times\mathbb{R}^{n}\times[0,2\pi)\times I_{0},\ \ a_{0}=\frac{\bar{a}_{0}}{\epsilon},\ \upsilon_{0}=\bar{\upsilon}_{0}\epsilon\nonumber
\end{eqnarray}
is of form
\begin{eqnarray}\label{sol-1}
u(x,t)=e^{i\upsilon\cdot x+i\gamma}\big(\phi(x-a,\mu)+r(x-a,t)\big)
\end{eqnarray}
where
\begin{eqnarray}\label{r-1}
\big\|r\big\|_{H^{1}}=\mathcal {O}(\epsilon^{4-\delta})
\end{eqnarray}
and the parameters $\upsilon$, $a$, $\gamma$ and $\mu$ satisfy the following differential equations
\begin{eqnarray}\label{mol-1}
&&\dot{\upsilon}=-\epsilon^{3}(\nabla V)(\epsilon a)+\mathcal {O}(\epsilon^{8-2\delta})\nonumber\\
&&\dot{a}=\upsilon+\mathcal {O}(\epsilon^{8-2\delta})\nonumber\\
&&\dot{\gamma}=\mu-\frac{|\upsilon|^{2}}{2}-\epsilon^{2}V(\epsilon a)+\mathcal {O}(\epsilon^{8-2\delta})\nonumber\\
&&\dot{\mu}=\mathcal {O}(\epsilon^{8-2\delta}).
\end{eqnarray}
\end{thm}

\begin{rem}{\rm  Since our potential $V_{\epsilon}$ is flat of size $\mathcal{O}(\epsilon^{-1})$ and also small in $L^{\infty}$ of size $\mathcal{O}(\epsilon^{2})$, the existence time interval of solution  and the estimates for the remainder terms in $(\ref{sol-1})$ and $(\ref{mol-1})$ are slightly better than the ones in \cite{FJGS}, where the authors only assume potential is flat and not necessarily to be small.

}
\end{rem}

\subsection{Post-interactions after $T_\epsilon$}\label{sec-2-3}

The modulations equations (\ref{mol-1}) for parameters $\sigma$ in Theorem \ref{thm-1}
shows that the moving solitary wave hits a small and flat potential, it moves almost along classic trajectory. Since the potential is a smooth bump function (one can even make it radial), we would expect that the solitary wave will move out of the impact of the potential. In fact, by using the modulation equation (\ref{mol-1}), it could be realized by the time $\mathcal{O}(\epsilon^{-2})$ if one choose quantity of initial position $\bar{a}_{0}$ and velocity $\bar{\upsilon}_{0}$ are of order $1$.

Thus in the following we assume $|\bar{a}_{0}|=|\bar{\upsilon}_{0}|=\mathcal{O}(1)$ and choose $\delta=\delta_{0}$ with some $2<\delta_{0}<\frac{5}{2}$ and $\epsilon\ll1$ in Theorem \ref{thm-1}.
Now we consider equation (\ref{Sch eq-1}) from time $C\epsilon^{-\delta_{0}}$ to $+\infty$. Let us begin with the equation (\ref{Sch eq-1}) at time $T_{0}=C\epsilon^{2-\delta_{0}}\epsilon^{-2}=T\epsilon^{-2}$ for large constant $T$,
\begin{eqnarray}\label{Sch eq-3}
&&i\partial_{t}u=-\frac{1}{2}\Delta u+\epsilon^{2}V(\epsilon\cdot)\psi-F(|u|^{2})u\nonumber\\
&&u(\cdot,T_{0})=u_{T_{0}}=e^{i\upsilon_{T_{0}}\cdot x+i\gamma_{T_{0}}}\big(\phi(x-a_{T_{0}},\mu_{T_{0}})+r(x-a_{T_{0}},T_{0})\big).
\end{eqnarray}
with
\begin{eqnarray}\label{eq-21}
&&\big|\upsilon_{T_{0}}\big|=\big|\upsilon(T_{0})\big|=c_{\bar{\upsilon}_{0},V}\epsilon^{3-\delta_{0}}+\mathcal {O}(\epsilon^{8-3\delta_{0}})\geq C_{1}\epsilon\nonumber\\
&&\big|a_{T_{0}}\big|=\big|a(T_{0})\big|=c_{\bar{a}_{0},V}\epsilon^{3-2\delta_{0}}+\mathcal {O}(\epsilon^{8-4\delta_{0}})\geq C_{2}\epsilon^{-1}\nonumber\\
&&\gamma_{T_{0}}=\gamma(T_{0})=\gamma_{0}+\mathcal {O}(\epsilon^{8-2\delta_{0}})\nonumber\\
&&\mu_{T_{0}}=\mu(T_{0})=\mu_{0}+\mathcal {O}(\epsilon^{8-2\delta_{0}}),
\end{eqnarray}
where $C_{1}$ and $C_{2}$ are large constants.
We start (\ref{Sch eq-3}) at $t=0$ and rewrite it as
 \begin{eqnarray}\label{Sch eq-4}
&&i\partial_{t}u=-\frac{1}{2}\Delta u+V_{\epsilon}u-F(|u|^{2})u\nonumber\\
&&u(\cdot,0)=u_{0}=e^{i\upsilon_{T_{0}}\cdot x+i\gamma_{T_{0}}}\big(\phi(x-a_{T_{0}},\mu_{T_{0}})+r(x-a_{T_{0}},T_{0})\big),
\end{eqnarray}
where $V_{\epsilon}(x)=\epsilon^{2}V(\epsilon x)$,
the parameter $\sigma_{T_{0}}=\{a_{T_{0}},\upsilon_{T_{0}},\gamma_{T_{0}},\mu_{T_{0}}\}$ satisfies (\ref{eq-21}) and $r$ satisfies (\ref{r-1}) with $\delta=\delta_{0}$.

Notice that the support of $V_{\epsilon}$ is of $\epsilon^{-1}$ with $\epsilon\ll1$, by (\ref{mol-1}), (\ref{eq-21}) and the observation that the solitary wave  moves almost along classic trajectory, we know that at the solitary wave moves away from the potential and at $t=0$ in equation (\ref{Sch eq-4}) they almost separate from each other. Thus it is reasonable to assume
\begin{eqnarray}\label{eq-211}
\big|a_{T_{0}}+\upsilon_{T_{0}}t\big|\geq c_{1}\epsilon^{-1}+c_{0}\epsilon t,
\end{eqnarray}
with some large positive constants $c_{0}$ and $c_{1}$. This also means in equation (\ref{Sch eq-4}), the soliton already sits out of the impact of potential at time $t=0$
and the distance between the centers of moving soliton and potential become far away from each other as time goes.

To deal with equation (\ref{Sch eq-4}), we first linearize it around soliton.
Let $u(x,t)$ be the solution near moving soliton and make ansatz
\begin{eqnarray} \label{ansatz}
u(x,t)=e^{i\theta(x,t;\sigma(t))}\phi(x-y(t;\sigma(t)),\mu(t))+R(x,t):=w(x,t;\sigma(t))+R(x,t).
\end{eqnarray}
Here
\begin{eqnarray}\label{ex-21}
&&\theta(x,t;\sigma(t))=\upsilon(t)\cdot x-\int_{0}^{t}\big(\dot{\upsilon}(s)\cdot y(s;\sigma(s))+\frac{|\upsilon(s)|^{2}}{2}-\mu(s)\big)ds+\gamma(t)\nonumber\\
&&y(t;\sigma(t))=\int_{0}^{t}\upsilon(s)ds+a(s).
\end{eqnarray}
We will write
\begin{eqnarray}\label{ex-24}
w(\sigma(t))=w(x,t;\sigma(t))=e^{i\theta(\sigma(t))}\phi(\sigma(t))
\end{eqnarray}
where
\begin{eqnarray}\label{ex-22}
\theta(\sigma(t))=\theta(x,t;\sigma(t))
\end{eqnarray}
 and
\begin{eqnarray}\label{ex-23}
\phi(\sigma(t))=\phi(x-y(t;\sigma(t)),\mu(t)),
\end{eqnarray}
It is easy to see
\begin{eqnarray}\label{eq-22}
i\partial_{t}\big(w(\sigma(t))+R(x,t)\big)&=&-\big(\dot{\upsilon}(t)\cdot (x-y(t;\sigma(t)))-\frac{1}{2}|\upsilon(t)|^{2}+\mu(t)+\dot{\gamma}(t)\big)w(\sigma(t))\nonumber\\
&&\ \ \ \ \ -\ ie^{i\theta(\sigma(t))}\big(\upsilon(t)+\dot{a}(t)\big)\nabla \phi(\sigma(t))\nonumber\\
&&\ \ \ \ \ \ \ \ \ \ +\ ie^{i\theta(\sigma(t))}\dot{\mu}(t)\partial_{\mu}\phi(\sigma(t))+i\partial_{t}R\nonumber
\end{eqnarray}
and
\begin{eqnarray}\label{eq-23}
\Delta\big(w(\sigma(t))+R(x,t)\big)=|\upsilon(t)|^{2}w(\sigma(t))-e^{i\theta(\sigma(t))}\Delta\phi(\sigma(t))
-2ie^{i\theta(\sigma(t))}\upsilon(t)\nabla\phi(\sigma(t))-\Delta R,\nonumber
\end{eqnarray}
which imply the equation for $R(x,t)$,
\begin{eqnarray}\label{eq-24}
i\partial_{t}R&=&-\frac{\Delta}{2} R+V_{\epsilon}R-F(|w(\sigma(t))|^{2})R-F'(|w(\sigma(t))|^{2})|w(\sigma(t))|^{2}R-F'(|w(\sigma(t))|^{2})w(\sigma(t))^{2}\overline{R}\nonumber\\
&&\ \ \ \ \ \ +\ \big(\dot{\upsilon}(t)\cdot (x-y(t;\sigma(t)))w(\sigma(t))+\dot{\gamma}(t) w(\sigma(t))\nonumber\\
&&\ \ \ \ \ \ \ \ \ \ \ \ +\ i\dot{a}(t)e^{i\theta(\sigma(t))}\nabla \phi(\sigma(t))-ie^{i\theta(\sigma(t))}\dot{\mu}(t)\partial_{\mu}\phi(\sigma(t))\big)\nonumber\\
&&\ \ \ \ \ \ \ \ \ \ \ \ \ \ +\ V_{\epsilon}w(\sigma(t))+\mathcal {O}(|R|^{q})+\mathcal {O}(|R|^{2}|w(\sigma(t))|^{q-2}).
\end{eqnarray}
Rewriting the equation (\ref{eq-24}) as a system for $Z=(R,\overline{R})^{T}$,
\begin{eqnarray}\label{eq-25}
i\partial_{t}Z=\mathscr{H}(t,\sigma(t))Z+\mathscr{N}(\sigma(t))+\mathscr{V}_{\epsilon}(\sigma(t))+\mathcal {O}(|Z|^{q})+\mathcal {O}(|Z|^{2}|w(\sigma(t))|^{q-2}).
\end{eqnarray}
Here $\frac{7}{3}<q\leq 5$ and the matrix charge transfer model
\begin{eqnarray}\label{eq-26}
\mathscr{H}(t,\sigma(t))=\mathscr{H}_{0}+\mathcal {V}_{1\epsilon}+\mathcal {V}_{2}(t,\sigma(t))
\end{eqnarray}
with matrixes
\begin{eqnarray}\label{eq-26'}
\mathscr{H}_{0}=\left(
  \begin{array}{cc}
    -\frac{1}{2}\Delta& 0 \\
  0 &  \frac{1}{2}\Delta\\
  \end{array}
  \right),\ \ \ \
  \mathcal {V}_{1\epsilon}=\left(
  \begin{array}{cc}
    V_{\epsilon}& 0 \\
  0 &  -V_{\epsilon}\\
  \end{array}
  \right)
\end{eqnarray}
and
\begin{eqnarray}\label{eq-26''}
&&\mathcal {V}_{2}(t,\sigma(t))\nonumber\\
&=&\left(
  \begin{array}{cc}
   -F(|w(\sigma(t))|^{2})-F'(|w(\sigma(t))|^{2})|w(\sigma(t))|^{2}& -F'(|w(\sigma(t))|^{2})w(\sigma(t))^{2} \\
    F'(|w(\sigma(t))|^{2})\overline{w(\sigma(t))}^{2} &  F(|w(\sigma(t))|^{2})+F'(|w(\sigma(t))|^{2})|w(\sigma(t))|^{2} \\
  \end{array}
  \right).\nonumber
\end{eqnarray}
We would take both $\mathscr{N}(\sigma(t))$ and $\mathscr{V}_{\epsilon}(\sigma(t))$ as nonlinear terms, which are interpreted as
\begin{eqnarray}\label{eq-27}
\mathscr{N}(\sigma(t))=\left(
\begin{array}{c}
f\\
-\overline{f}
\end{array}
\right)
\end{eqnarray}
with
\begin{eqnarray}\label{eq-28}
f=\dot{\upsilon}(t)\cdot (x-y(t;\sigma(t))) w(\sigma(t))+\dot{\gamma}(t) w(\sigma(t))+i\dot{a}(t)e^{i\theta(\sigma(t))}\nabla \phi(\sigma(t))
-ie^{i\theta(\sigma(t))}\dot{\mu}(t)\partial_{\mu}\phi(\sigma(t))\nonumber
\end{eqnarray}
and
\begin{eqnarray}\label{eq-29}
\mathscr{V}_{\epsilon}(\sigma(t))=\left(
\begin{array}{c}
V_{\epsilon}w(\sigma(t))\\
-V_{\epsilon}\overline{w(\sigma(t))}
\end{array}
\right).
\end{eqnarray}

\begin{defi}\label{defi-2}
Let $\sigma(t)$ be an admissible path and $\theta(x,t;\sigma(t))$ and $y(t;\sigma(t))$ be defined as in (\ref{ex-21}).
Then we define
\begin{eqnarray}
\xi_{1}(x,t;\sigma(t))=\left(
\begin{array}{c}
e^{i\theta(x,t;\sigma(t))}\phi(x-y(t;\sigma(t)),\mu(t))\\
e^{-i\theta(x,t;\sigma(t))}\phi(x-y(t;\sigma(t)),\mu(t))
\end{array}
\right),
\nonumber
\end{eqnarray}
\begin{eqnarray}
\xi_{2}(x,t;\sigma(t))=\left(
\begin{array}{c}
ie^{i\theta(x,t;\sigma(t))}\partial_{\mu}\phi(x-y(t;\sigma(t)),\mu(t))\\
-ie^{-i\theta(x,t;\sigma(t))}\partial_{\mu}\phi(x-y(t;\sigma(t)),\mu(t))
\end{array}
\right),\nonumber
\end{eqnarray}
\begin{eqnarray}
\xi_{j}(x,t;\sigma(t))=\left(
\begin{array}{c}
ie^{i\theta(x,t;\sigma(t))}\partial_{x_{j-2}}\phi(x-y(t;\sigma(t)),\mu(t))\\
-ie^{-i\theta(x,t;\sigma(t))}\partial_{x_{j-2}}\phi(x-y(t;\sigma(t)),\mu(t))
\end{array}
\right),\ \ \ 3\leq j\leq 5\nonumber
\end{eqnarray}
and
\begin{eqnarray}
\xi_{j}(x,t;\sigma(t))=\left(
\begin{array}{c}
\big(x_{j-5}-y_{j-5}(t;\sigma(t)\big)e^{i\theta(x,t;\sigma(t))}\phi(x-y(t;\sigma(t)),\mu(t))\\
\big(x_{j-5}-y_{j-5}(t;\sigma(t)\big)e^{-i\theta(x,t;\sigma(t))}\phi(x-y(t;\sigma(t)),\mu(t))
\end{array}
\right),\ \ \ 6\leq j\leq 8\nonumber
\end{eqnarray}
\end{defi}


\begin{pro}\label{pro-22}
Let $Z$ satisfy the system (\ref{eq-26}) and $\xi_{j}$ be defined as in Definition \ref{defi-2}.
Suppose for all $t\geq0$,
\begin{eqnarray}\label{ex-210}
\big\langle Z(t),\xi_{j}(x,t;\sigma(t))\big\rangle=0
\end{eqnarray}
where $\xi_{j}$ is defined as in Definition \ref{defi-2}. Then we have the following system for parameter vector $\widetilde{\sigma}(t)$,
\begin{eqnarray}
&&-2i\dot{\mu}(t)\big\langle \partial_{\mu}\phi(\sigma(t)),\phi(\sigma(t))\big\rangle+\big\langle Z(t),\dot{\sigma}(t)\Phi_{1}(\sigma(t))\big\rangle\nonumber\\
&&\ \ \ =\ \big\langle\mathscr{V}_{\epsilon}(\sigma(t)),\xi_{1}(\cdot,t;\sigma(t))\big\rangle +
\big\langle\mathcal {O}(|Z|^{q})+\mathcal {O}(|Z|^{2}|w(\sigma(t))|^{q-2}),\xi_{1}(\cdot,t;\sigma(t))\big\rangle,\nonumber
\end{eqnarray}
\begin{eqnarray}
&&2i\dot{\gamma}(t)\big\langle \partial_{\mu}\phi(\sigma(t)),\phi(\sigma(t))\big\rangle+\big\langle Z(t),\dot{\sigma}(t)\Phi_{2}(\sigma(t))\big\rangle\nonumber\\
&&\ \ \ =\ \big\langle\mathscr{V}_{\epsilon}(\sigma(t)),\xi_{2}(\cdot,t;\sigma(t))\big\rangle +
\big\langle\mathcal {O}(|Z|^{q})+\mathcal {O}(|Z|^{2}|w(\sigma(t))|^{q-2}),\xi_{2}(\cdot,t;\sigma(t))\big\rangle,\nonumber
\end{eqnarray}
\begin{eqnarray}
&&\dot{\upsilon}_{j-3}(t)\big\|\phi(\sigma(t))\big\|^{2}_{L^{2}}+\big\langle Z(t),\dot{\sigma}(t)\Phi_{j}(\sigma(t))\big\rangle\nonumber\\
&&\ \ \ =\ \big\langle\mathscr{V}_{\epsilon}(\sigma(t)),\xi_{j}(\cdot,t;\sigma(t))\big\rangle +
\big\langle\mathcal {O}(|Z|^{q})+\mathcal {O}(|Z|^{2}|w(\sigma(t))|^{q-2}),\xi_{j}(\cdot,t;\sigma(t))\big\rangle,\ 3\leq j \leq 5,\nonumber
\end{eqnarray}
\begin{eqnarray}\label{ex-25}
&&\dot{a}_{j-5}(t)\big\|\phi(\sigma(t))\big\|^{2}_{L^{2}}+\big\langle Z(t),\dot{\sigma}(t)\Phi_{j}(\sigma(t))\big\rangle\nonumber\\
&&\ \ \ =\ \big\langle\mathscr{V}_{\epsilon}(\sigma(t)),\xi_{j}(\cdot,t;\sigma(t))\big\rangle +
\big\langle\mathcal {O}(|Z|^{q})+\mathcal {O}(|Z|^{2}|w(\sigma(t))|^{q-2}),\xi_{j}(\cdot,t;\sigma(t))\big\rangle,\ 6\leq j \leq 8,\nonumber\\
\end{eqnarray}
where $\frac{7}{3}<q\leq 5$, $w(\sigma(t))$ and $\phi(\sigma(t))$ are defined by (\ref{ex-24}) and (\ref{ex-23}) respectively and for all $j$
\begin{eqnarray}\label{eq-220}
\big|\Phi_{j}(\sigma(t))\big|\leq C\big(|\phi(\sigma(t))|+|D\phi(\sigma(t))|+|D^{2}\phi(\sigma(t))|+|D\partial_{\mu}\phi(\sigma(t))|\big).
\end{eqnarray}
\end{pro}

For post-interaction region, we have the following statement:
\begin{thm}\label{thm-1-1}
Let $F$ as in (\ref{ground}) with a small fixed parameter $\theta>0$. 
Assume that the linearized operator $\mathscr{H}_{2}(\sigma)$ defined by (\ref{eq-214}) satisfies the spectral condition for all $\sigma\in \mathbb{R}^{8}$ with $|\sigma-\sigma(0)|<c$ and $\epsilon\ll 1$. Then solution to equation (\ref{Sch eq-4}) is of form
\begin{eqnarray} \label{ex1-1-1}
u(x,t)=w(x,t;\sigma(t))+R(x,t)
\end{eqnarray}
with $w(x,t;\sigma(t))$ defined as in (\ref{ansatz}),
\begin{eqnarray}\label{ex1-2-1}
\|R\|_{L^{2}_{t}W^{1,6}_{x}\cap L^{\infty}_{t}H^{1}_{x}}\lesssim \epsilon^{\alpha}\ \ \ {\rm and}\ \ \ \|\dot{\sigma}\|_{L_{t}^{1}}\lesssim \epsilon^{2\alpha},
\end{eqnarray}
for some $0<\alpha<4-\delta_{0},\ 2<\delta_{0}<\frac{5}{2}$. Moreover, there exist $\sigma_{+}$ and $u_{+}\in H^{1}$ such that
\begin{eqnarray}\label{ex1-3-1}
\|u-w(x,t;\sigma_{+})-e^{it\Delta}u_{+}\|_{H^{1}}\rightarrow 0,\ \ \ t\rightarrow +\infty,
\end{eqnarray}
where $w(x,t;\sigma_{+})=e^{i\theta_{+}(x,t)}\phi(x-\int_{0}^{t}\upsilon(s)ds-a_{+},\mu_{+})$ with
$$\theta_{+}(x,t)=\upsilon_{+}\cdot x-\int_{0}^{t}\big(\dot{\upsilon}(s)\cdot y(s;\sigma(s))+\frac{|\upsilon(s)|^{2}}{2}-\mu(s)\big)ds+\gamma_{+}.$$
\end{thm}

\section{The proof of Theorem \ref{thm-1}}
\setcounter{equation}{0}

In this section, to make our paper self-contained, we will sketch the proof of Theorem \ref{thm-1} by using the method of  \cite{FJGS}. We first note that by the spectral assumptions, it is easy to verify all the assumptions in \cite{FJGS}. And then the proof of Theorem \ref{thm-1} will be divided into several subsections.

\subsection{Hamiltonian and solitary manifold}

Consider nonlinear Schr\"{o}dinger equation  (\ref{Sch eq-1})
\begin{eqnarray}
i\partial_{t}u=-\frac{1}{2}\Delta u+\epsilon^{2}V(\epsilon\cdot)u-F(|u|^{2})u\nonumber
\end{eqnarray}
and define its associated Hamiltonian functional on $H^{1}(\mathbb{R}^{3},\mathbb{C})$
\begin{eqnarray}\label{eng-func}
\mathcal {W}(u):=\frac{1}{4}\int_{\mathbb{R}^{3}}|\nabla u|^{2}dx+\frac{1}{2}\int_{\mathbb{R}^{3}}V_{\epsilon}|u|^{2}dx-\mathcal {F}(u)
\end{eqnarray}
where $V_{\epsilon}(x)=\epsilon^{2}V(\epsilon x)$ and $\mathcal{F}'(u)=F(|u|^{2})u$. Here let us review some basic facts of $H^{1}(\mathbb{R}^{3},\mathbb{C})$.
It is equipped with form
\begin{eqnarray}\label{5-1}
\omega(u,v):=\Im \int_{\mathbb{R}^{3}}u\bar{v}dx,
\end{eqnarray}
which is considered as a real space
\begin{eqnarray}\label{5-2}
H^{1}(\mathbb{R}^{3},\mathbb{R}^{2})=H^{1}(\mathbb{R}^{3},\mathbb{R})\oplus H^{1}(\mathbb{R}^{3},\mathbb{R}),\ \ u\mapsto(\Re u, \Im u).
\end{eqnarray}
It also has real inner product
\begin{eqnarray}\label{5-3}
\langle u,v\rangle:=\Re \int_{\mathbb{R}^{3}}u\bar{v}dx,
\end{eqnarray}
so that $\omega(u,v)=\langle u,\mathcal {J}^{-1}v\rangle$, where
\begin{eqnarray}\label{5-4}
\mathcal{J}=\left(
  \begin{array}{cc}
    0 & 1 \\
    -1 & 0 \\
  \end{array}
  \right)
\end{eqnarray}
is an complex structure  on $H^{1}(\mathbb{R}^{3},\mathbb{R}^{2})$ corresponding to the operator $i^{-1}$ on $H^{1}(\mathbb{R}^{3},\mathbb{C})$ and we also use $\mathcal{J}=i^{-1}$. Thus the equation (\ref{Sch eq-1}) can be written as
\begin{eqnarray}\label{5-5}
\partial_{t}u=\mathcal{J}\mathcal {W}'(u).
\end{eqnarray}
The Hamiltonian $\mathcal {W}(u)$ enjoys the conservation of energy and mass, that is, $\mathcal {W}(u)=const$ and $\mathcal {N}(u)=const$ with
\begin{eqnarray}\label{5-6}
\mathcal{N}=\int_{\mathbb{R}^{3}}|u|^{2}dx.
\end{eqnarray}
Notice that the groundstate  $\phi_{\mu}=\phi(\cdot,\mu)$ defined by
\begin{eqnarray}
-\frac{1}{2}\Delta\phi_{\mu}-F(|\phi_{\mu}|^{2})\phi_{\mu}=-\mu\phi_{\mu},\nonumber
\end{eqnarray}
for some $\mu\in I\subset \mathbb{R}$ is the critical point of the functional
\begin{eqnarray}\label{func}
\mathcal {E}(u):=\frac{1}{4}\int_{\mathbb{R}^{3}}|\nabla u|^{2}+\frac{\mu}{2}|u|^{2}dx-\mathcal {F}(u).
\end{eqnarray}
The Hessian of $\mathcal {E}(u)$ at $\phi_{\mu}$ is the operator
\begin{eqnarray}\label{5-7}
\mathcal{L}_{\mu}=-\frac{1}{2}\Delta+\mu-(F(|\phi_{\mu}|^{2})\phi_{\mu})'.
\end{eqnarray}
and
\begin{eqnarray}\label{5-8}
\mathcal{L}_{\mu}=\left(
  \begin{array}{cc}
    L_{+} & 0 \\
    0 & L_{-}\\
  \end{array}
  \right)
\end{eqnarray}
in complex and real expression, respectively. Here
\begin{eqnarray}\label{5-9}
L_{-}=-\frac{1}{2}\Delta+\mu-F(|\phi_{\mu}|^{2})
\end{eqnarray}
and
\begin{eqnarray}\label{5-10}
L_{+}=-\frac{1}{2}\Delta+\mu-2F'(|\phi_{\mu}|^{2})\phi_{\mu}^{2}-F(|\phi_{\mu}|^{2}).
\end{eqnarray}

Now we introduce the manifold of solitary waves. Let $\sigma:=(a, \upsilon, \mu, \gamma)$ and
\begin{eqnarray}\label{5-10}
\phi_{\sigma}:=\phi_{a\upsilon \mu\gamma}=\mathcal{S}_{a\upsilon\gamma}\phi_{\mu}=\mathcal{T}_{a}\mathcal{T}_{\upsilon}\mathcal{T}_{\gamma}\phi_{\mu}(x)
=e^{i\upsilon(x-a)+i\gamma}\phi_{\mu}(x-a)
\end{eqnarray}
where
\begin{eqnarray}\label{5-11}
\mathcal{T}_{a}u(x,t)=u(x-a,t),\ \mathcal{T}_{\upsilon}u(x,t)=e^{i\upsilon\cdot x}u(x,t)\ \text{and}\  \mathcal{T}_{\gamma}u(x,t)=e^{i\gamma}u(x,t).
\end{eqnarray}
The manifold of solitary waves is defined as
\begin{eqnarray}\label{5-12}
M_{s}:=\big\{\phi_{a\upsilon \mu\gamma}:\ (a, \upsilon, \mu, \gamma)\in \mathbb{R}^{3}\times\mathbb{R}^{3}\times[0,2\pi]\times I\big\},
\end{eqnarray}
and then the tangent space to this manifold at $\phi_{\mu}\in M_{s}$ is given by
\begin{eqnarray}\label{5-13}
T_{\phi_{\mu}}M_{s}:={\rm span}\big(z_{a}, z_{\upsilon}, z_{\gamma}, z_{\mu}\big),
\end{eqnarray}
where
\begin{eqnarray}\label{5-14}
z_{a}=-\nabla\phi_{\mu},\ z_{\upsilon}=-\mathcal{J}x\phi_{\mu},\ z_{\gamma}=-\mathcal{J}\phi_{\mu}\ {\rm and}\ z_{\mu}=\partial_{\mu}\phi_{\mu}.
\end{eqnarray}
Here we have to note that $\mathcal{J}=i^{-1}$ if one takes the complex expression of $\mathcal{L}_{\mu}$ and $\mathcal{J}$ is of form (\ref{5-4}) and $\phi_{\mu}=(\phi_{\mu},0)$ if
one uses the real representation of $\mathcal{L}_{\mu}$, we will use the complex representation of $\mathcal{L}_{\mu}$  in the rest of Section \ref{sec-5}.
Moreover, it follows
\begin{eqnarray}\label{5-15}
\mathcal{L}_{\mu}z_{a}=0,\ \mathcal{L}_{\mu}z_{\gamma}=0,\ \mathcal{L}_{\mu}z_{\upsilon}=iz_{t},\ \mathcal{L}_{\mu}z_{\mu}=iz_{\gamma}
\end{eqnarray}
and
\begin{eqnarray}\label{5-16}
\mathcal{L}^{2}_{\mu}z_{\upsilon}=0,\  \mathcal{L}^{2}_{\mu}z_{\mu}=0.
\end{eqnarray}

\subsection{Skew-orthogonal decomposition}
In this subsection, we will decomposition the solution to equation (\ref{Sch eq-1}) along manifold $M_{s}$ into a solitary wave and a fluctuation which
is skew-orthogonal to the soliton manifold and derive the equations for the fluctuation and parameters $\sigma=(a, \upsilon, \mu, \gamma)$. To this end, let us
define the $\delta-$ neighborhood
\begin{eqnarray}\label{5-17}
U_{\delta}=\big\{u\in H^{1}:\ \inf_{\sigma\in \Sigma_{0}}\|u-\phi_{\sigma}\|_{H^{1}}\leq \delta\big\},
\end{eqnarray}
where
\begin{eqnarray}\label{5-18}
\Sigma_{0}:=\mathbb{R}^{3}\times\mathbb{R}^{3}\times[0,2\pi]\times (I\setminus \partial I).\nonumber
\end{eqnarray}
Then we have the following so called skew-orthogonal decomposition for all $u\in U_{\delta}$ with small enough $\delta$ and refer the reader to \cite[Proposition 5.1]{FJGS} for the proof.
\begin{lem}\label{le-51}
Let $u\in U_{\delta}$ for sufficiently small $\delta>0$. There exists a unique $\sigma=\sigma(u)\in C(U_{\delta},\Sigma)$ such that
\begin{eqnarray}\label{5-19}
\omega(u-\phi_{\sigma},z)=\langle u-\phi_{\sigma},\mathcal {J}^{-1}z\rangle=0,\ \ \forall z\in T_{\phi_{\sigma}}M_{s}.
\end{eqnarray}
\end{lem}

Now given a solution $u$ to equation (\ref{Sch eq-1}) such that $u\in U_{\delta}$, it follows from Lemma \ref{le-51}
\begin{eqnarray}\label{eq-5-20}
u-\phi_{\sigma}\bot \mathcal{J}T_{\phi_{\sigma}}M_{s}.
\end{eqnarray}
Set
\begin{eqnarray}\label{eq-5-21}
u-\phi_{\sigma}=\mathcal{S}_{a\upsilon\gamma}r\ \ {\rm and}\ \ \varphi:=\mathcal{S}_{a\upsilon\gamma}^{-1}u.
\end{eqnarray}
We introduce operators
\begin{eqnarray}\label{eq-5-22}
(\mathcal{K}_{1},\mathcal{K}_{2},\mathcal{K}_{3})=\nabla,\ (\mathcal{K}_{4},\mathcal{K}_{5},\mathcal{K}_{6})=-\mathcal{J}(x_{1},x_{2},x_{3}),\ \mathcal{K}_{7}=-\mathcal{J},\ \mathcal{K}_{8}=\partial_{\mu}
\end{eqnarray}
with coefficients
\begin{eqnarray}\label{eq-5-23}
&&(\alpha_{1},\alpha_{2},\alpha_{3})=\dot{a}-\upsilon,\ \ \ (\alpha_{1},\alpha_{2},\alpha_{3})=-\dot{\upsilon}-\nabla V_{\epsilon}(a),\nonumber\\
 &&\alpha_{7}=\mu-\frac{1}{2}\upsilon^{2}+\dot{a}\cdot \upsilon-V_{\epsilon}(a)-\dot{\gamma},\ \ \ \alpha_{8}=-\dot{\mu}.
\end{eqnarray}

\begin{lem}\label{le-52}
The fluctuation $r$ defined as in (\ref{eq-5-21}) satisfies the equation
\begin{eqnarray}\label{5-24}
\dot{r}=(\mathcal{J}\mathcal{L}_{\mu}+\sum_{j=1}^{7}\alpha_{j}\mathcal{K}_{j}+\mathcal{R}_{V})r+N_{\mu}(r)+\mathcal{J}\sum_{j=1}^{8}\alpha_{j}\mathcal{K}_{j}\phi_{\sigma}+\mathcal{R}_{V}\phi_{\sigma}
\end{eqnarray}
 and the parameter $\sigma$ satisfy the equations
\begin{eqnarray}\label{5-25}
&&\dot{a}_{k}=\upsilon_{k}+(m(\mu))^{-1}\big(\langle x_{k}\phi_{\mu},\mathcal {J}N_{\mu}(r)+\mathcal {J}\mathcal{R}_{V}r\rangle+\sum_{j=1}^{8}\langle\alpha_{j}\mathcal{K}_{j}x_{k}\phi_{\mu},r\rangle\big)\nonumber\\
&&\dot{\upsilon}_{k}=-\partial_{k}V_{\epsilon}(a)+(m(\mu))^{-1}\langle \partial_{k}\phi_{\mu},N_{\mu}(r)+\mathcal{R}_{V}r\rangle-\sum_{j=1}^{8}\langle\alpha_{j}\mathcal{K}_{j}\partial_{k}\phi_{\mu},\mathcal{J}r\rangle+\langle \partial_{k}\phi_{\mu},\mathcal{R}_{V}\phi_{\mu}\rangle\nonumber\\
&&\dot{\gamma}=\mu-\frac{1}{2}\upsilon^{2}+\dot{a}\cdot \upsilon-V_{\epsilon}(a)-(m'(\mu))^{-1}\langle \partial_{\mu}\phi_{\mu},N_{\mu}(r)+\mathcal{R}_{V}r\rangle-\sum_{j=1}^{8}\langle\alpha_{j}\mathcal{K}_{j}\partial_{\mu}\phi_{\mu},\mathcal{J}r\rangle
\nonumber\\
&&\ \ \ \ \ \ \ \ \ +\ \langle \partial_{\mu}\phi_{\mu},\mathcal{R}_{V}\phi_{\mu}\rangle\nonumber\\
&&\dot{\mu}=(m'(\mu))^{-1}\langle \phi_{\mu},\mathcal {J}N_{\mu}(r)+\mathcal {J}\mathcal{R}_{V}r\rangle-\sum_{j=1}^{8}\langle\alpha_{j}\mathcal{K}_{j}\phi_{\mu},r\rangle,\nonumber\\
\end{eqnarray}
where
\begin{eqnarray}\label{5-26}
\mathcal{R}_{V}:=V_{\epsilon}(x+a)-V_{\epsilon}(a)-\nabla V_{\epsilon}(a)\cdot x,\ \  \ \ m(\mu)=\int_{\mathbb{R}^{3}}|\phi_{\mu}|^{2}dx
\end{eqnarray}
and
\begin{eqnarray}\label{5-27}
N_{\mu}(r)=F(|r+\phi_{\mu}|^{2})(r+\phi_{\mu})-F(|\phi_{\mu}|^{2})(\phi_{\mu})-(F(|\phi_{\mu}|^{2})(\phi_{\mu}))'r.
\end{eqnarray}
Moreover, let $|\alpha|=\max_{j=1,\ldots,8}|\alpha_{j}|$ with $\alpha_{j}$ defined by (\ref{eq-5-23}), we have
\begin{eqnarray}\label{5-28}
|\alpha|=\mathcal {O}(|\alpha|\|r\|_{H^{1}}+\epsilon^{4}+\|r\|^{2}_{H^{1}}).
\end{eqnarray}
\end{lem}

\begin{proof}
Actually, (\ref{5-24}), (\ref{5-25})  and (\ref{5-28})  are obtained by the same procedures as the ones in \cite{FJGS}. Precisely, one only need to substitute $V$ by $V_{\epsilon}$ and correspondingly use the estimate
\begin{eqnarray}\label{5-29}
\mathcal{R}_{V}:=V_{\epsilon}(x+a)-V_{\epsilon}(a)-\nabla V_{\epsilon}(a)\cdot x=\mathcal{O}(\epsilon^{4}x^{2}).\nonumber
\end{eqnarray}
in the whole proof.
\end{proof}

\subsection{The completion of proof }

We will use an approximate of Lyaponuv functional to obtain an explicit estimates for $r$ and $\sigma$ and then finish the proof of  Theorem \ref{thm-1}. The whole process is based on the analysis in \cite{FJGS}, except we have to keep track of extra $\epsilon$ which comes from the potential $V_{\epsilon}$.
Let recall the decomposition for $u\in U_{\delta}$,
\begin{eqnarray}\label{5-30}
\varphi:=\mathcal{S}_{a\upsilon\gamma}^{-1}u=\phi_{\mu}+r,
\end{eqnarray}
and then we prove that the Lyapunov functional $\mathcal{E}(\varphi)-\mathcal{E}(\phi_{\mu})$ is approximately conserved.

\begin{lem}\label{le-53}
Let $u\in U_{\delta}$ be the solution to equation (\ref{Sch eq-1}) and $\varphi$, $r$ and $\phi_{\mu}$ be defined as in (\ref{5-30}). Then
\begin{eqnarray}\label{5-33}
\partial_{t}(\mathcal{E}(\varphi)-\mathcal{E}(\phi_{\mu}))=\mathcal{O}(|\alpha|\|r\|^{2}_{H^{1}}+\epsilon^{4}\|r\|_{H^{1}}+\epsilon^{3}\|r\|^{2}_{H^{1}}).
\end{eqnarray}
\end{lem}

\begin{proof}
Let $u$ be the solution to equation (\ref{Sch eq-1}), we first notice that
\begin{eqnarray}\label{5-34}
\partial_{t}\int_{\mathbb{R}^{3}}V_{\epsilon}|u|^{2}dx=\langle(\nabla V_{\epsilon})iu,\nabla u\rangle\nonumber
\end{eqnarray}
and the Ehrenfest's theorem
\begin{eqnarray}\label{5-35}
\partial_{t}\langle u,-i\nabla u\rangle=-\langle(\nabla V_{\epsilon})u, u\rangle,\nonumber
\end{eqnarray}
which follows from the nonlinear Schrodinger equation (\ref{Sch eq-1}). Then by using the same trick as the one in the proof of \cite[Lemma 3]{FJGS}, we have
\begin{eqnarray}\label{5-36}
\partial_{t}\mathcal{E}(\varphi)=\frac{1}{2}\dot{\mu}\|\varphi\|^{2}_{L^{2}}-\frac{1}{2}\langle(\dot{\upsilon}+\nabla V_{\epsilon}(\cdot+a))\varphi, \nabla\varphi\rangle.
\end{eqnarray}
On the other hand, since $\phi_{\mu}$ is the critical point of functional $\mathcal{E}(\varphi)$ and thus
\begin{eqnarray}\label{5-37}
\partial_{t}\mathcal{E}(\phi_{\mu})=\frac{1}{2}\dot{\mu}\|\phi_{\mu}\|^{2}_{L^{2}}.
\end{eqnarray}
Applying (\ref{5-36}) and (\ref{5-37}), we obtain
\begin{eqnarray}\label{5-38}
\partial_{t}(\mathcal{E}(\varphi)-\mathcal{E}(\phi_{\mu}))&=&\frac{1}{2}\dot{\mu}\big(\|\varphi\|^{2}_{L^{2}}-\|\phi_{\mu}\|^{2}_{L^{2}}\big)-\frac{1}{2}\langle(\dot{\upsilon}+\nabla V_{\epsilon}(\cdot+a))\varphi, \nabla\varphi\rangle\nonumber\\
&:=&J_{1}-J_{2}.
\end{eqnarray}
It follows from the skew-orthogonal decomposition (\ref{5-30}) that
\begin{eqnarray}\label{5-39}
J_{1}=\frac{1}{2}\dot{\mu}\|r\|^{2}_{L^{2}}=\mathcal{O}(|\alpha|\|r\|^{2}_{L^{2}}).
\end{eqnarray}
As for $J_{2}$,
\begin{eqnarray}\label{5-40}
2J_{2}&=&\langle(\dot{\upsilon}+\nabla V_{\epsilon}(\cdot+a))r, \nabla r\rangle+\langle(\nabla V_{\epsilon}(\cdot+a))i\phi_{\mu}, \nabla r\rangle+\langle(\nabla V_{\epsilon}(\cdot+a))ir, \nabla \phi_{\mu}\rangle\nonumber\\
&=&\langle(\dot{\upsilon}+\nabla V_{\epsilon}(a))r, \nabla r\rangle+\langle(\nabla V_{\epsilon}(\cdot+a)-\nabla V_{\epsilon}(a))r, \nabla r\rangle\nonumber\\
&&\ \ \ \ \ \ +\ \langle(\nabla V_{\epsilon}(\cdot+a)-\nabla V_{\epsilon}(a))i\phi_{\mu}, \nabla r\rangle+\langle(\nabla V_{\epsilon}(\cdot+a)-\nabla V_{\epsilon}(a))ir, \nabla \phi_{\mu}\rangle,\nonumber
\end{eqnarray}
where in the first equality above we use $\langle i\nabla\phi_{\mu},r\rangle=\langle iz_{t},r\rangle=0$ and $\langle iqr,\nabla r\rangle=0$ for any real-valued function $q\in L^{\infty}$, and in the second equality we apply
\begin{eqnarray}\label{5-41}
\nabla V_{\epsilon}(a)\langle ir, i\nabla\phi_{\mu}\rangle=\nabla V_{\epsilon}(a)\langle i\phi_{\mu},\nabla r\rangle=0.\nonumber
\end{eqnarray}
Notice that
\begin{eqnarray}\label{5-42}
\nabla V_{\epsilon}=\mathcal{O}(\epsilon^{3}) \ \ {\rm and}\ \ \nabla V_{\epsilon}(x+a)-\nabla V_{\epsilon}(a)=\mathcal{O}(\epsilon^{4}|x|),\nonumber
\end{eqnarray}
it is easy to see
\begin{eqnarray}\label{5-43}
J_{2}=\mathcal{O}(|\alpha|\|r\|^{2}_{H^{1}}+\epsilon^{4}\|r\|_{H^{1}}+\epsilon^{3}\|r\|^{2}_{H^{1}}).
\end{eqnarray}
Combining (\ref{5-39}) and (\ref{5-43}), we finish the proof.
\end{proof}

Next we introduce the lower bound for $\mathcal{E}(\varphi)-\mathcal{E}(\phi_{\mu})$ due to \cite{FJGS}.
\begin{lem}\label{le-54}
Let $u\in U_{\delta}$ be the solution to equation (\ref{Sch eq-1}) and $\varphi$, $r$ and $\phi_{\mu}$ be defined as in (\ref{5-30}). Then there exist positive constants $\rho$ and $c$ independent of $\epsilon$ such that for $\|r\|_{H^{1}}\leq 1$,\
\begin{eqnarray}\label{5-44}
\big|\mathcal{E}(\phi_{\mu}+r)-\mathcal{E}(\phi_{\mu})\big|\geq \rho\|r\|^{2}_{H^{1}}-c\|r\|^{3}_{H^{1}}.
\end{eqnarray}
\end{lem}

Now we present our main result.
\begin{pro}\label{pro-55}
Let $u\in U_{\delta}$ be the solution to equation (\ref{Sch eq-1}) and $\varphi$, $r$ and $\phi_{\mu}$ be defined as in (\ref{5-30}). Assume that $\epsilon$ is sufficiently small. Then there exist positive constants $c,C$ independent of $\epsilon$ such that for $2\leq\delta\leq3$ and $t\leq C\min\big\{\epsilon^{-\delta},\ \epsilon^{-8+2\delta}\big\}$,
\begin{eqnarray}\label{5-45}
&&\|r\|_{H^{1}}\leq c\epsilon^{4-\delta}\nonumber\\
&&\|\alpha\|_{L^{\infty}}\leq c\epsilon^{8-2\delta}.
\end{eqnarray}
\end{pro}

\begin{proof}
Notice that for $\|r\|_{H^{1}}\leq 1$, it follows from Lemmas \ref{le-53} and \ref{le-54}
\begin{eqnarray}\label{5-46}
\rho\|r(t)\|^{2}_{H^{1}}\leq c_{1}t\big(\epsilon^{4}\|r(t)\|_{H^{1}}+(\epsilon^{3}+|\alpha|)\|r(t)\|^{2}_{H^{1}}\big)+c_{2}\|r(t)\|^{3}_{H^{1}},
\end{eqnarray}
where $\rho$ defined as in Lemma \ref{le-54} and $c_{1},\ c_{2}$ are positive constants. We rewrite (\ref{5-46})
\begin{eqnarray}\label{5-47}
C\|r(t)\|^{2}_{H^{1}}\leq t\big(\epsilon^{4}\|r(t)\|_{H^{1}}+(\epsilon^{3}+|\alpha|)\|r(t)\|^{2}_{H^{1}}\big)+\|r(t)\|^{3}_{H^{1}},
\end{eqnarray}
for some constant $C>0$.
For $t\leq \frac{C}{2}(\epsilon^{\delta}+|\alpha|)^{-1}:=\tau$ with  $2\leq\delta\leq3$, it follows from (\ref{5-47}) that
\begin{eqnarray}\label{5-48}
\frac{C}{4}t^{8-2\delta}-\frac{C}{4}\|r(t)\|^{2}_{H^{1}}+\|r(t)\|^{3}_{H^{1}}\geq0\nonumber.
\end{eqnarray}
Take $X=\sup_{t\in [0,\tau]}\|r(t)\|^{2}_{H^{1}}$ and then
we have $\frac{C}{4}t^{8-2\delta}-\frac{C}{4}X^{2}+X^{3}\geq0$, which leads to
\begin{eqnarray}\label{5-49}
X\leq c\epsilon^{4-\delta},
\end{eqnarray}
with some constant $c>0$ for sufficiently small $\epsilon$. Now putting (\ref{5-49}) back in (\ref{5-28}), we have
\begin{eqnarray}\label{5-49}
\sup_{t\in [0,\tau]}|\alpha(t)|\leq c\epsilon^{8-2\delta},\nonumber
\end{eqnarray}
which finishes the proof.
\end{proof}

{\bf \emph{The proof of Theorem \ref{thm-1}}}:  Choose $\epsilon$ such that $c\epsilon^{4-\delta}<\frac{1}{2}h$ where the constant $c$ is the one in (\ref{5-45}) and $h$ is given in Lemma \ref{le-51}. Then there exists a maximal $T_{0}>0$ such that the solution $u$ of equation (\ref{Sch eq-1}) belongs to $U_{h}$ for all $t\leq T_{0}$. It follows from Lemma \ref{le-52} and Proposition \ref{pro-55} that Theorem \ref{thm-1} is true for $t\leq\min\{T_{0}, c\epsilon^{8-2\delta}\}$, which combined the inequality $c\epsilon^{4-\delta}<\frac{1}{2}h$ imply Theorem \ref{thm-1} holds for $t\leq c\epsilon^{8-2\delta}$.

\section{The proof of Theorem \ref{thm-1-1}}
\setcounter{equation}{0}

This section is devoted to the proof of Theorem \ref{thm-1-1}, the asymptotic behavior of the solution for post-interaction region. First of all, we have obtained the following matrix Schr\"{o}dinger equation (see also equation (\ref{eq-25}))
\begin{eqnarray}\label{eq-z-1}
&&i\partial_{t}Z=\mathscr{H}(t,\sigma(t))Z+\mathscr{N}(\sigma(t))+\mathscr{V}_{\epsilon}(\sigma(t))+\mathcal {O}(|Z|^{q})+\mathcal {O}(|Z|^{2}|w(\sigma(t))|^{q-2}),\nonumber\\
&&Z(0)=(R(0),\overline{R}(0))^{T},
\end{eqnarray}
with $\frac{7}{3}<q\leq 5$, $\mathscr{N}(\sigma(t))$ and $\mathscr{V}_{\epsilon}(\sigma(t))$ defined by (\ref{eq-27}) and (\ref{eq-29}) respectively, and
\begin{eqnarray}
\big\|Z(0)\big\|_{H^{1}}=\mathcal {O}(\epsilon^{4-\delta_{0}}),\nonumber
\end{eqnarray}
as well as the ODE system for $\sigma$ (see Proposition \ref{pro-22}).

{\bf Bootstrap assumptions:} There exist a sufficiently large constant $C_{0}$ such that for soma $0<\alpha<4-\delta_{0}$
\begin{eqnarray}
\label{eq-222}&&\big\|Z\big\|_{L^{2}_{t}W^{1,6}_{x}\cap L^{\infty}_{t}H^{1}_{x}}\leq C_{0}^{-1}\epsilon^{\alpha}.\\
\label{eq-221}&&\|\dot{\sigma}\|_{L^{1}_{t}}\leq \epsilon^{2\alpha},
\end{eqnarray}
where $\delta_{0}\in (2,\frac{5}{2})$ defined in Section 2.3.

\subsection{End-point Strichartz estimates}\label{sec-4-1}
In this subsection, we always assume that the bootstrap assumption (\ref{eq-221}) holds for $\epsilon\ll1$ which will be verified in the next subsection. let us introduce a new charge transfer Hamiltonian,
\begin{eqnarray}\label{eq-228}
\mathscr{H}(t,\widetilde{\sigma}(t))=\mathscr{H}_{0}+\mathcal {V}_{1\epsilon}+\mathcal {V}_{2}(t,\widetilde{\sigma}(t))
\end{eqnarray}
where  $\mathscr{H}_{0}$ and $\mathcal {V}_{1\epsilon}$ are defined as in (\ref{eq-26}) and
\begin{eqnarray}\label{eq-225}
&&\mathcal {V}_{2}(t,\widetilde{\sigma}(t))\nonumber\\
&=&\left(
  \begin{array}{cc}
   -F(|w(\widetilde{\sigma}(t))|^{2})-F'(|w(\widetilde{\sigma}(t))|^{2})|w(\widetilde{\sigma}(t))|^{2}& -F'(|w(\widetilde{\sigma}(t))|^{2})w(\widetilde{\sigma}(t))^{2} \\
    F'(|w(\widetilde{\sigma}(t))|^{2})\overline{w(\widetilde{\sigma}(t))}^{2} &  F(|w(\widetilde{\sigma}(t))|^{2})+F'(|w(\widetilde{\sigma}(t))|^{2})|w(\widetilde{\sigma}(t))|^{2} \\
  \end{array}
  \right),\nonumber
\end{eqnarray}
with
\begin{eqnarray}\label{eq-227}
w(\widetilde{\sigma}(t))=e^{i\widetilde{\theta}(t,x,\widetilde{\sigma}(t))}\phi(x-\widetilde{y}(t,\widetilde{\sigma}(t))),
\end{eqnarray}
\begin{eqnarray}\label{ex-226}
&&\widetilde{\theta}(x,t;\widetilde{\sigma}(t))=\upsilon_{T_{0}}\cdot x+\gamma_{T_{0}}+\Big(\mu_{T_{0}}-\frac{|\upsilon_{T_{0}}|^{2}}{2}\Big)t-\int_{0}^{t}\Big(\frac{|\upsilon(s)-\upsilon_{T_{0}}|^{2}}{2}-(\mu(s)-\mu_{T_{0}})\Big)ds,\nonumber\\
&&\widetilde{y}(t;\widetilde{\sigma}(t))=\int_{0}^{t}\big(\upsilon(s)-\upsilon_{T_{0}}\big)ds+a_{T_{0}}+\upsilon_{T_{0}}t.
\end{eqnarray}
Let $T_{0}(t)$ and $T(t)$ be operators defined by the following formulas:
\begin{eqnarray}\label{eq-42-1}
T_{0}(t)=B_{\beta_{0}(t),y_{0}(t),\upsilon_{0}},\ T(t)=B_{\beta(t),y(t),0},
\ (B_{\beta,y,\upsilon}f)(x)=e^{i\beta\varrho+i\upsilon\cdot x\varrho}f(x-y)
\end{eqnarray}
where
\begin{eqnarray}\label{eq-42}
&&\varrho=\left(
  \begin{array}{cc}
    1 & 0 \\
    0 & -1 \\
  \end{array}
  \right),\nonumber\\
 && \beta_{0}(t)=\gamma_{T_{0}}+\Big(\mu_{T_{0}}-\frac{|\upsilon_{T_{0}}|^{2}}{2}\Big)t,\ \ y_{0}(t)=a_{T_{0}}+\upsilon_{T_{0}}t,\nonumber\\
 &&\beta(t)=\int_{0}^{t}\big(-\frac{|\upsilon(s)-\upsilon_{T_{0}}|^{2}}{2}+\mu(s)-\mu_{T_{0}}\big)ds,\ \ y(t)=\int_{0}^{t}\big(\upsilon(s)-\upsilon_{T_{0}}\big)ds.
\end{eqnarray}
Thus defined, we have
 $$B_{\beta,y,\upsilon}f=B_{\beta,0,\upsilon}B_{0,y,0}f$$
 and
 $$B^{\ast}_{\beta,y,\upsilon}f=B_{-\beta,-y,0}B_{0,0,-\upsilon}f.$$
Denote
\begin{eqnarray}\label{eq-45}
P_{b}(t)=T_{0}(t)T(t)P_{b}(\sigma_{T_{0}})T^{\ast}(t)T_{0}^{\ast}(t)
\end{eqnarray}
and
\begin{eqnarray}\label{eq-45-1}
P_{c}(t)=T_{0}(t)T(t)P_{c}(\sigma_{T_{0}})T^{\ast}(t)T_{0}^{\ast}(t)
\end{eqnarray}
where  $P_{c}(\sigma_{T_{0}})$
defined as in Proposition \ref{pro-21} is the projection onto the subspace of the continuous spectrum of $\mathscr{H}_{2}(\sigma_{T_{0}})$ (see (\ref{eq-214}) for definition)
and $P_{b}(\sigma_{T_{0}})=I-P_{c}(\sigma_{T_{0}})$  defined by (\ref{eq-216}). Moreover, let
\begin{eqnarray}
\widetilde{\xi}_{j}(x,t;\widetilde{\sigma}(t))=T_{0}(t)T(t)\xi_{j}(x,\sigma_{T_{0}})\nonumber
\end{eqnarray}
and
\begin{eqnarray}
\widetilde{\eta}_{j}(x,t;\widetilde{\sigma}(t))=T_{0}(t)T(t)\eta_{j}(x,\sigma_{T_{0}}),\nonumber
\end{eqnarray}
where $\xi_{j}(x,\sigma_{T_{0}})$ and $\eta_{j}(x,\sigma_{T_{0}})$ are defined in section \ref{sec2-1}. It follows from (\ref{eq-45}) and (\ref{eq-216}) that
\begin{eqnarray}\label{eq-216-1}
P_{b}(t)f&=&\frac{1}{n_{1}}\big(\widetilde{\eta}_{1}(x,t;\widetilde{\sigma}(t))\langle f,\widetilde{\xi}_{2}(x,t;\widetilde{\sigma}(t))\rangle+\widetilde{\eta}_{2}(x,t;\widetilde{\sigma}(t))\langle f,\widetilde{\xi}_{1}(x,t;\widetilde{\sigma}(t))\rangle\big)\nonumber\\
&&+\ \ \ \sum_{\ell=3}^{5}\frac{1}{n_{\ell}}\big(\widetilde{\eta}_{\ell}(x,t;\widetilde{\sigma}(t))\langle f,\widetilde{\xi}_{\ell+3}(x,t;\widetilde{\sigma}(t))\rangle+\widetilde{\eta}_{\ell+3}(x,t;\widetilde{\sigma}(t))\langle f,\widetilde{\xi}_{\ell}(x,t;\widetilde{\sigma}(t))\rangle\big).
\end{eqnarray}

Now we introduce the end-point Strichartz estimates for the new Hamiltonian (\ref{eq-228}).  The admissible pair $(p,q)$ for Strichartz estimates satisfies
\begin{eqnarray}\label{adm}
\frac{2}{p}=\frac{3}{2}-\frac{3}{q},\ \ \ \ \ 2\leq p\leq\infty.
\end{eqnarray}
In particular, the endpoint admissible pair $(p,q)=(2,6)$ is crucial in our paper.

Since we study the soliton-potential problem in $H^{1}$, it is necessary to consider the  end-point Strichartz estimates in the following form which will be obtained by applying Theorem \ref{thm-2}.
\begin{pro}\label{pro-31}
Let $Z(t,x)$ be the solution of equation (\ref{eq-301}) and satisfy (\ref{eq-302}). Then for all admissible pairs $(p,q)$ and $(\tilde{p},\tilde{q})$
\begin{eqnarray}\label{eq-304}
\big\|Z\big\|_{L^{p}_{t}W^{1,q}_{x}}\lesssim \|Z_{0}\|_{H^{1}}+\|F\|_{L^{\tilde{p}'}_{t}W^{1,\tilde{q}'}_{x}}+B,
\end{eqnarray}
where $B$ is the constant in (\ref{eq-302}).
\end{pro}

\begin{proof}
Our main idea is apply the Strichartz estimates Theorem \ref{thm-2} for $\partial_{k}Z$ ($k=1,2,3$).
To this end, differentiating the equation (\ref{eq-301}) we obtain the following equation for $\partial_{k}Z$,
\begin{eqnarray}\label{eq-307}
i\partial_{t}\partial_{k}Z=\mathscr{H}(t,\widetilde{\sigma}(t))\partial_{k}Z+\big(\partial_{k}\mathcal {V}_{1\epsilon}+\partial_{k}\mathcal {V}_{2}(t,\widetilde{\sigma}(t))\big)Z+\partial_{k}F.
\end{eqnarray}
On the other hand, it follows from  (\ref{eq-216-1}) that
\begin{eqnarray}\label{eq-305}
P_{b}(t)\partial_{k}Z(t)&=&-\frac{1}{n_{1}}\Big(\widetilde{\eta}_{1}(x,t;\widetilde{\sigma}(t))\langle Z(t),\partial_{k}\widetilde{\xi}_{2}(x,t;\widetilde{\sigma}(t))\rangle+\widetilde{\eta}_{2}(x,t;\widetilde{\sigma}(t))\langle Z(t),\partial_{k}\widetilde{\xi}_{1}(x,t;\widetilde{\sigma}(t))\rangle\Big)\nonumber\\
&&-\ \ \ \sum_{\ell=3}^{5}\frac{1}{n_{\ell}}\Big(\widetilde{\eta}_{\ell}(x,t;\widetilde{\sigma}(t))\langle Z(t),\partial_{k}\widetilde{\xi}_{\ell+3}(x,t;\widetilde{\sigma}(t))\rangle\nonumber\\
&&\ \ \ \ \ \ \ \ \ \ \ \ \ \ \ \ \ \ \ \ \ \ +\widetilde{\eta}_{\ell+3}(x,t;\widetilde{\sigma}(t))\langle Z(t),\partial_{k}\widetilde{\xi}_{\ell}(x,t;\widetilde{\sigma}(t))\rangle\Big).\nonumber
\end{eqnarray}
which combined the endpoint Strichartz estimates for $Z$ (see Theorem \ref{thm-2}) lead to
\begin{eqnarray}\label{eq-306}
\big\|P_{b}(t)\partial_{k}Z(t)\big\|_{L^{2}_{t}L^{6}_{x}}\lesssim \big\|Z\big\|_{L^{2}_{t}L^{6}_{x}}\lesssim \|Z_{0}\|_{L^{2}}+\|F\|_{L^{\tilde{p}'}_{t}L^{\tilde{q}'}_{x}}+B:=\widetilde{B}.
\end{eqnarray}
That is, $\partial_{k}Z$ ($k=1,2,3$) satisfies (\ref{eq-302}) for some constant $\widetilde{B}$.
Finally, Using the endpoint Strichartz estimates (\ref{eq-303}) again we have
\begin{eqnarray}\label{eq-308}
\big\|\partial_{k}Z\big\|_{L^{p}_{t}L^{q}_{x}}&\lesssim& \big\|\partial_{k}Z_{0}\big\|_{L^{2}}+\big\|\big(\partial_{k}\mathcal {V}_{1\epsilon}+\partial_{k}\mathcal {V}_{2}(t,\widetilde{\sigma}(t))\big)Z\big\|_{L^{2}_{t}L^{6/5}_{x}}+\big\|\partial_{k}F\big\|_{L^{\tilde{p}'}_{t}L^{\tilde{q}'}_{x}}+\widetilde{B}\nonumber\\
&\lesssim& \big\|\partial_{k}Z_{0}\big\|_{L^{2}}+\big\|Z\big\|_{L^{2}_{t}L^{6}_{x}}+\big\|\partial_{k}F\big\|_{L^{\tilde{p}'}_{t}L^{\tilde{q}'}_{x}}+\widetilde{B}\nonumber\\
&\lesssim&\|Z_{0}\|_{H^{1}}+\|F\|_{L^{\tilde{p}'}_{t}W^{1,\tilde{q}'}_{x}}+B.
\end{eqnarray}
Thus we finish the proof.
\end{proof}



\subsection{The completion of proof }
We first close the estimate for $\sigma$ by the following proposition.
\begin{pro}\label{pro-23}
Assume the separation(\ref{eq-211}) and spectral assumptions hold. Let $\widetilde{\sigma}, Z$ be any choice of functions that satisfy the bootstrap assumptions for sufficiently small $\epsilon>0$. Then $$\|\dot{\sigma}\|_{L^{1}_{t}}\leq \frac{1}{2}\epsilon^{2\alpha}.$$
\end{pro}

\begin{proof}
By the convexity condition (see (II) of Proposition \ref{pro-21}) and the bootstrap assumption (\ref{eq-222}), the left hand side of the system (\ref{ex-25}) is of form $B(t)\dot{\widetilde{\sigma}}$ with $B(t)$ an invertible matrix of order $\mathcal{O}(1)$. Then we have for $\frac{7}{3}<q\leq 5$,
\begin{eqnarray}\label{eq-223}
\|\dot{\sigma}\|_{L^{1}_{t}}&\leq& \sum_{j}\Big(\big\|\big\langle\mathscr{V}_{\epsilon}(\sigma(t)),\xi_{j}(\cdot,t;\sigma(t))\big\rangle\big\|_{L^{1}_{t}} +
\big\|\big\langle\mathcal {O}(|Z|^{q})+\mathcal {O}(|Z|^{2}|w(\sigma(t))|^{q-2}),\xi_{j}(\cdot,t;\sigma(t))\big\rangle\big\|_{L^{1}_{t}}\Big)\nonumber
\end{eqnarray}
Notice that for each $\frac{7}{3}<q\leq 5$, we can always choose $n>\frac{3}{2}$ such that $2\leq n(q-2)\leq 6$ and then
\begin{eqnarray}
&&\big\|\big\langle\mathcal {O}(|Z|^{q})+\mathcal {O}(|Z|^{2}|w(\sigma(t))|^{q-2}),\xi_{j}(\cdot,t;\sigma(t))\big\rangle\big\|_{L^{1}_{t}}\nonumber\\
&\leq&C\big\||Z|^{2}\big\|_{L^{1}_{t}L^{3}_{x}}+C\big\||Z|^{2}\big\|_{L^{1}_{t}L^{3}_{x}}\big\||Z|^{q-2}\big\|_{L^{\infty}_{t}L^{n}_{x}}\nonumber\\
&\leq&C\big\|Z\big\|^{2}_{L^{2}_{t}L^{6}_{x}}+C\big\|Z\big\|^{2}_{L^{2}_{t}L^{6}_{x}}\big\|Z\big\|^{q-2}_{L^{\infty}_{t}L^{n(q-2)}_{x}}\nonumber\\
&\leq&CC_{0}^{-2}\epsilon^{2\alpha}\leq \frac{1}{2}\epsilon^{2\alpha},\nonumber
\end{eqnarray}
where we use the fact that $C_{0}$ is a sufficiently large constant in the last inequality. On the other hand, by bootstrap assumption, we have
$\sigma(t)\sim\sigma_{T_{0}}$ for all $t$, which combined the exponentially decay of $\phi$, supp $V\subset B(0,1)$ and the separation inequality (\ref{eq-211}) lead to
\begin{eqnarray}\label{eq-223'}
\big\|\big\langle\mathscr{V}_{\epsilon}(\sigma(t)),\xi_{j}(\cdot,t;\sigma(t))\big\rangle\big\|_{L^{1}_{t}}
&\leq&C \Big\|\int_{\mathbb{R}^{n}}e^{-\mu(0)|x-(a(0)+\upsilon(0)t)|}V_{\epsilon}(x)dx\Big\|_{L^{1}_{t}}\nonumber\\
&\leq&C\Big\|\int_{\mathbb{R}^{n}}e^{-\mu(0)(|(a(0)+\upsilon(0)t)|-|x|)}V_{\epsilon}(x)dx\Big\|_{L^{1}_{t}}\nonumber\\
&\leq& C\epsilon^{k}
\end{eqnarray}
for any $k>0$. Hence we finish the proof.
\end{proof}

It remains to verify the bootstrap assumption (\ref{eq-222}) for the perturbation $Z$. We first rewrite the equation (\ref{eq-z-1}) for $Z$ as follows,
\begin{eqnarray}\label{eq-224}
&&i\partial_{t}Z=\mathscr{H}(t,\widetilde{\sigma}(t))Z+F(t),\nonumber\\
&&Z(0)=(R(0),\overline{R}(0))^{T},
\end{eqnarray}
where
$$F(t)=\big(\mathscr{H}(t,\sigma(t))-\mathscr{H}(t,\widetilde{\sigma}(t))\big)Z+\mathscr{N}(\sigma(t))+\mathscr{V}_{\epsilon}(\sigma(t))+\mathcal {O}(|Z|^{q})+\mathcal {O}(|Z|^{2}|w(\sigma(t))|^{q-2})$$
with $2<q\leq 5$ and
$\mathscr{N}(\sigma(t))$ and $\mathscr{V}_{\epsilon}(\sigma(t))$ defined by (\ref{eq-27}) and (\ref{eq-28}), respectively.
And then by using the already proved bootstrap estimates for $\sigma$  and the Strichartz estimates for matrix Schr\"{o}dinger equation (\ref{eq-224}) (see Theorem \ref{thm-2} and Proposition \ref{pro-31}), we could finally finish the proof.
To do this, let us begin with the following lemma which verifies the assumption (\ref{eq-302}) in Theorem \ref{thm-2}.

\begin{lem}\label{le-43}
Let $Z$ satisfies the bootstrap assumption (\ref{eq-222}) and the orthogonality condition
\begin{eqnarray}\label{orth}
\big\langle Z(t),\xi_{j}(x,t;\sigma(t))\big\rangle=0
\end{eqnarray}
with respect to an admissible path $\sigma(t)$ obeying the bootstrap estimates (\ref{eq-221}). Then we have
\begin{eqnarray}\label{eq-229}
\|P_{b}(t)Z(t)\|_{L^{2}_{t}L^{6}_{x}}\lesssim \epsilon^{3\alpha}.
\end{eqnarray}
\end{lem}

\begin{proof}
Notice that for all $t>0$
\begin{eqnarray}\label{eq-230}
&&\big|\theta(x,t;\sigma(t))-\widetilde{\theta}(x,t;\widetilde{\sigma}(t))\big|\nonumber\\
&\leq&\big|\gamma(t)-\gamma(0)\big|+\big|\big(\upsilon(t)-\upsilon(0)\big)\big(x-y(t;\sigma(t))\big)\big|+\int^{t}_{0}\big|\big(\upsilon(s)-\upsilon(0)\big)\dot{a}(s)\big|ds\nonumber\\
&\leq&\|\dot{\sigma}\|_{L^{1}}\big(1+\|\dot{\sigma}\|_{L^{1}}+\big|x-y(t,\sigma(t))\big|\big)
\end{eqnarray}
and
\begin{eqnarray}\label{eq-231}
\big|y(t;\sigma(t))-\widetilde{y}(t;\widetilde{\sigma}(t))\big|\leq \|\dot{\sigma}\|_{L^{1}}.
\end{eqnarray}
On the other hand, it follows from (\ref{eq-216-1}) and the orthogonality condition (\ref{orth}) that
\begin{eqnarray}\label{eq-232}
P_{b}(t)Z(t)
&=&\frac{1}{n_{1}}\Big(\widetilde{\eta}_{1}(x,t;\widetilde{\sigma}(t))\langle Z(t),\widetilde{\xi}_{2}(x,t;\widetilde{\sigma}(t))\rangle+\widetilde{\eta}_{2}(x,t;\widetilde{\sigma}(t))\langle Z(t),\widetilde{\xi}_{1}(x,t;\widetilde{\sigma}(t))\rangle\Big)\nonumber\\
&&\ \ \ \ \ \ \ \ \ +\ \sum_{\ell=3}^{5}\frac{1}{n_{\ell}}\Big(\widetilde{\eta}_{\ell}(x,t;\widetilde{\sigma}(t))\langle Z(t),\widetilde{\xi}_{\ell+3}(x,t;\widetilde{\sigma}(t))\rangle\nonumber\\
&&\ \ \ \ \ \ \ \ \ \ \ \ \ \ \ \ \ \ \ \ \ \ +\widetilde{\eta}_{\ell+3}(x,t;\widetilde{\sigma}(t))\langle Z(t),\widetilde{\xi}_{\ell}(x,t;\widetilde{\sigma}(t))\rangle\Big)\nonumber\\
&=&\frac{1}{n_{1}}\Big(\widetilde{\eta}_{1}(x,t;\widetilde{\sigma}(t))\langle Z(t),\widetilde{\xi}_{2}(x,t;\widetilde{\sigma}(t))-\xi_{2}(x,t;\sigma(t))\rangle\nonumber\\
&&\ \ \ \ \ \ \ \ +\ \widetilde{\eta}_{2}(x,t;\widetilde{\sigma}(t))\langle Z(t),\widetilde{\xi}_{1}(x,t;\widetilde{\sigma}(t))-\xi_{1}(x,t;\sigma(t))\rangle\Big)\nonumber\\
&&\ \ \ \ \ \ \  \ \ \ \ + \sum_{\ell=3}^{5}\frac{1}{n_{\ell}}\Big(\widetilde{\eta}_{\ell}(x,t;\widetilde{\sigma}(t))\langle Z(t),\widetilde{\xi}_{\ell+3}(x,t;\widetilde{\sigma}(t))-\xi_{\ell+3}(x,t;\sigma(t))\rangle\nonumber\\
&&\ \ \ \ \ \ \ \ \ \ \ \ \ \ \ \ \ \ \ \ \ \ +\ \widetilde{\eta}_{\ell+3}(x,t;\widetilde{\sigma}(t))\langle Z(t),\widetilde{\xi}_{\ell}(x,t;\widetilde{\sigma}(t))-\xi_{\ell}(x,t;\sigma(t))\rangle\Big).
\end{eqnarray}
Moreover, by using (\ref{eq-230})-(\ref{eq-231}), we have
\begin{eqnarray}\label{eq-233}
\big|\widetilde{\xi}_{1}(x,t;\widetilde{\sigma}(t))-\xi_{1}(x,t;\sigma(t))\big|
&\lesssim&\big|e^{i\widetilde{\theta}(t,x,\widetilde{\sigma}(t))}\phi(x-\widetilde{y}(t,\widetilde{\sigma}(t)))-e^{i\theta(t,x,\sigma(t))}\phi(x-y(t,\sigma(t)))\big|\nonumber\\
&\lesssim&\big|e^{i\widetilde{\theta}(t,x,\widetilde{\sigma}(t))}\phi(x-\widetilde{y}(t,\widetilde{\sigma}(t)))-e^{i\widetilde{\theta}(t,x,\widetilde{\sigma}(t))}\phi(x-y(t,\sigma(t)))\big|\nonumber\\
&&\ +\ \big|e^{i\widetilde{\theta}(t,x,\widetilde{\sigma}(t))}\phi(x-y(t,\sigma(t)))
-e^{i\theta(t,x,\sigma(t))}\phi(x-y(t,\sigma(t)))\big|\nonumber\\
&\lesssim&\|\dot{\sigma}\|_{L^{1}}\big(2+\|\dot{\sigma}\|_{L^{1}}+\big|x-y(t,\sigma(t))\big|\big)\phi(x-y(t,\sigma(t))),
\end{eqnarray}
similar estimates also hold for $2\leq j\leq 8$.
Thus by (\ref{eq-232})-(\ref{eq-233}), it follows
\begin{eqnarray}\label{eq-234}
\big\|P_{b}(t)Z(t)\big\|_{L^{2}_{t}L^{6}_{x}}\lesssim \|Z\|_{L^{2}_{t}L^{6}_{x}}\|\dot{\sigma}\|_{L^{1}}\lesssim \epsilon^{3\alpha}.
\end{eqnarray}
\end{proof}

\begin{pro}\label{pro-24}
Let $Z$ be a solution of the equation (\ref{eq-224}) satisfying the bootstrap assumption (\ref{eq-222}). We also assume (due to Proposition \ref{pro-23}) that the admissible path $\sigma(t)$ obey the estimate (\ref{eq-221}). Then we have the following estimates
\begin{eqnarray}\label{eq-235}
\big\|Z\big\|_{L^{2}_{t}W^{1,6}_{x}\cap L^{\infty}_{t}H^{1}_{x}}\leq C_{0}^{-1}\frac{\epsilon^{\alpha}}{2}.
\end{eqnarray}
\end{pro}

\begin{proof}Notice that when $\frac{7}{3}<q\leq 5$,
\begin{eqnarray}
\big\||Z|^{q}\big\|_{L^{2}_{t}W^{1,\frac{6}{5}}_{x}}\leq \|Z\|_{L^{2}_{t}W^{1,6}_{x}}\|Z\|^{\frac{5-q}{2}}_{L^{\infty}_{t}L^{2}_{x}}\|Z\|^{\frac{3q-7}{2}}_{L^{\infty}_{t}L^{6}_{x}}\nonumber
\end{eqnarray}
and similarly to (\ref{eq-233})
\begin{eqnarray}
\big\|\big(\mathscr{H}(t,\sigma(t))-\mathscr{H}(t,\widetilde{\sigma}(t))\big)Z\big\|_{L^{2}_{t}W^{1,\frac{6}{5}}_{x}}&\leq& \big\|\big(\mathcal {V}_{2}(t,\sigma(t))-\mathcal {V}_{2}(t,\widetilde{\sigma}(t))\big)Z\big\|_{L^{2}_{t}W^{1,\frac{6}{5}}_{x}}\nonumber\\
&\lesssim& \big\|\dot{\sigma}\big\|_{L^{1}}\big\|Z\big\|_{L^{2}_{t}W^{1,6}_{x}}\leq C_{0}^{-1}\epsilon^{3\alpha}.\nonumber
\end{eqnarray}
Moreover, by using Proposition \ref{pro-23} and similar procedures as in (\ref{eq-223'}), we obtain
$$\big\|\mathscr{N}(\sigma(t))\big\|_{L^{1}_{t}W^{1,2}_{x}}\lesssim \epsilon^{2\alpha}$$
and for any $k>0$
\begin{eqnarray}
\big\|\mathscr{V}_{\epsilon}(\sigma(t))\big\|_{L^{2}_{t}W^{1,\frac{6}{5}}_{x}}\lesssim \epsilon^{k}.\nonumber
\end{eqnarray}
Then it follows from the endpoint Strichartz estimates (\ref{pro-31}) and Lemma \ref{le-43} that
\begin{eqnarray}\label{eq-236}
\big\|Z\big\|_{L^{2}_{t}W^{1,6}_{x}\cap L^{\infty}_{t}H^{1}_{x}}&\lesssim& \big\|Z_{0}\big\|_{H^{1}}+\big\|\big(\mathscr{H}(t,\sigma(t))-\mathscr{H}(t,\widetilde{\sigma}(t))\big)Z\big\|_{L^{2}_{t}W^{1,\frac{6}{5}}_{x}}
+\big\|\mathscr{N}(\sigma(t))\big\|_{L^{1}_{t}W^{1,2}_{x}}\nonumber\\
&&\  +\ \big\|\mathscr{V}_{\epsilon}(\sigma(t))\big\|_{L^{2}_{t}W^{1,\frac{6}{5}}_{x}}+\big\||Z|^{q}\big\|_{L^{2}_{t}W^{1,\frac{6}{5}}_{x}}
+\big\||Z|^{2}|w(\sigma(t))|^{q-2}\big\|_{L^{2}_{t}W^{1,\frac{6}{5}}_{x}}+B\nonumber\\
&\lesssim&  \epsilon^{4-\delta_{0}}+C_{0}^{-1}\epsilon^{3\alpha}+\epsilon^{2\alpha}+\epsilon^{k}+C_{0}^{-q}\epsilon^{q\alpha}+C_{0}^{-2}\epsilon^{2\alpha}+\epsilon^{3\alpha}\nonumber\\
&\leq&C_{0}^{-1}\frac{\epsilon^{\alpha}}{2},\nonumber
\end{eqnarray}
where $B$ is defined as in (\ref{eq-302}).
\end{proof}

 Finally, we prove the scattering. To this end, we rewrite the solution to equation (\ref{Sch eq-4}) as
 \begin{eqnarray}\label{eq-ex-1}
 u(x,t)=w(x,t;\sigma_{+})+\big(w(x,t;\sigma(t))-w(x,t;\sigma_{+})\big)+R(x,t)
\end{eqnarray}
where $$w(x,t;\sigma_{+})=e^{i\theta_{+}(x,t)}\phi(x-\int_{0}^{t}\upsilon(s)ds-a_{+},\mu_{+})$$
with
$$\theta_{+}(x,t)=\upsilon_{+}\cdot x-\int_{0}^{t}\big(\dot{\upsilon}(s)\cdot y(s;\sigma(s))+\frac{|\upsilon(s)|^{2}}{2}-\mu(s)\big)ds+\gamma_{+}.$$
Take $\sigma_{+}=\lim_{t\rightarrow+\infty}\sigma$ and notice that
\begin{eqnarray}\label{eq-ex-2}
&&|w(x,t;\sigma(t))-w(x,t;\sigma_{+})|\nonumber\\
&\leq& \big|e^{i\theta(x,t;\sigma(t))}-e^{i\theta_{+}(x,t)}\big|\phi(x-\int_{0}^{t}\upsilon(s)ds-a(t),\mu(t))\nonumber\\
&&\ +\ \big|\phi(x-\int_{0}^{t}\upsilon(s)ds-a(t),\mu(t))-\phi(x-\int_{0}^{t}\upsilon(s)ds-a_{+},\mu(t))\big|\nonumber\\
&&\ \ \ \ \ \ +\ \big|\phi(x-\int_{0}^{t}\upsilon(s)ds-a_{+},\mu(t))-\phi(x-\int_{0}^{t}\upsilon(s)ds-a_{+},\mu_{+})\big|,\nonumber
\end{eqnarray}
it is easy to see $$\|w(x,t;\sigma(t))-w(x,t;\sigma_{+}\|_{H^{1}}\rightarrow0,\ \ t\rightarrow+\infty$$
in $H^{1}$. On the other hand, $R(x,t)$ verifies an nonlinear Schr\"{o}dinger equation
\begin{eqnarray}\label{eq-ex-3}
&&i\partial_{t}R=-\frac{1}{2}\Delta R+V(x,t)R+F\nonumber\\
&&R(x,0)=e^{i\upsilon_{T_{0}}\cdot x+i\gamma_{T_{0}}}r(x-a_{T_{0}},T_{0}):=R_{0},
\end{eqnarray}
with spatially exponentially localized potential $V$ and $F$ the scalar version of the one in (\ref{eq-224}). It has been shown that $R$ satisfies
$$\big\|R\big\|_{L^{2}_{t}W^{1,6}_{x}\cap L^{\infty}_{t}H^{1}_{x}}\lesssim \epsilon^{\alpha}.$$
Then it follows from a standard small data scattering theorem, there exists $u_{+}\in H^{1}$ constructed by
\begin{eqnarray}\label{scattering}
u_{+}=R_{0}-\lim_{t\rightarrow +\infty}\int^{t}_{0}e^{-\frac{i}{2}\Delta}(V(x,\cdot)R(x,\cdot)+F(x,\cdot))ds
\end{eqnarray}
such that
$$\|u-w(x,t;\sigma_{+})-e^{it\Delta}u_{+}\|_{H^{1}}\rightarrow 0,\ \ \ t\rightarrow +\infty.
$$
Thus we finish the proof of Theorem  \ref{thm-1-1}.

\section{The proof of Theorem \ref{thm-2}}\label{sec-5}
\setcounter{equation}{0}
\subsection{Some key lemmas}

We introduce one soliton adiabatic propagators $\mathscr{U}_{2}(t,s)$:
\begin{eqnarray}\label{eq-41}
&&i\mathscr{U}_{2t}(t,s)=\mathscr{H}_{2}(t)\mathscr{U}_{2}(t,s),\ \ \mathscr{U}_{2}(t,t)=I,
\ \ \mathscr{H}_{2}(t)=\mathscr{H}_{0}+\mathcal {V}_{2}(t,\widetilde{\sigma}(t))+E(t)\nonumber\\
&&E(t)=iT_{0}(t)\big[\overline{P}_{c}'(t),\overline{P}_{c}(t)\big]T^{\ast}_{0}(t),\ \ \overline{P}_{c}(t)=T(t)P_{c}(\sigma_{T_{0}})T^{\ast}(t).
\end{eqnarray}
Here $P_{c}(\sigma_{T_{0}})$ defined as in Proposition (\ref{pro-21}) is the projection onto the subspace of the continuous spectrum of $\mathscr{H}_{2}(\sigma_{T_{0}})$ (see (\ref{eq-214}) for definition),
 $\overline{P}_{c}'(t)$ is the derivative of the projector  $\overline{P}_{c}(t)$ with respect to $t$ and the operators $T_{0}(t)$, $T(t)$ are given by (\ref{eq-42-1}).
It is known that (see [Per])
\begin{eqnarray}\label{eq-43}
\mathscr{U}_{2}(t,s)P_{c}(s)=P_{c}(t)\mathscr{U}_{2}(t,s),\ \ \ P_{c}(t)=T_{0}(t)\overline{P}_{c}(t)T_{0}^{\ast}(t)
\end{eqnarray}
and
\begin{eqnarray}\label{eq-44}
\mathcal {V}_{2}(t,\widetilde{\sigma}(t))=T_{0}(t)T(t)\mathcal {V}_{2}(\sigma_{T_{0}})T^{\ast}(t)T_{0}^{\ast}(t).
\end{eqnarray}
Moreover, one has
\begin{eqnarray}\label{eq-49}
\mathscr{U}_{2}(t,s)=T_{0}(t)U_{2}(t,s)T_{0}^{\ast}(s)
\end{eqnarray}
where $U_{2}(t,s)$ is the propagator associated to the equation
\begin{eqnarray}\label{eq-410'}
&&iU_{2t}(t,s)=H_{2}(t)U_{2}(t,s),\ \ \ U_{2}(t,t)=I,
\end{eqnarray}
\begin{eqnarray}\label{eq-410}
 H_{2}(t)&=&\mathscr{H}_{0}+T(t)\mathcal {V}_{2}(\sigma_{T_{0}})T^{\ast}(t)+i\big[\overline{P}_{c}'(t),\overline{P}_{c}(t)\big]\nonumber\\
 &=&T(t)\big(\mathscr{H}_{0}+\mathcal {V}_{2}(\sigma_{T_{0}})\big)T^{\ast}(t)+i\big[\overline{P}_{c}'(t),\overline{P}_{c}(t)\big]\nonumber\\
 &:=&T(t)H_{2}(\sigma_{T_{0}})T^{\ast}(t)+i\big[\overline{P}_{c}'(t),\overline{P}_{c}(t)\big].
\end{eqnarray}
The adiabatic theorem (see \cite{ASY}, \cite{Kato1}, \cite{Per1} for exmple) says that
\begin{eqnarray}\label{eq-410'''}
\overline{P}_{c}(t)U_{2}(t,s)=U_{2}(t,s)\overline{P}_{c}(s).
\end{eqnarray}
\begin{lem}\label{le-41}
Let $\phi(t)$ be the solution of Schr\"{o}dinger equation
\begin{eqnarray}\label{eq-413''}
i\phi_{t}(t)=\mathscr{H}_{2}(t)\phi(t)+G(t)
\end{eqnarray}
with initial data $\phi(0)=f$
and $\epsilon\ll1$.
Then have
\begin{eqnarray}\label{eq-413}
\big\|P_{c}(t)\mathscr{U}_{2}(t)f\big\|_{L^{\infty}_{t}L^{2}_{x}\cap L^{2}_{t}L^{6}_{x}}\lesssim \|f\|_{L^{2}}
\end{eqnarray}
and
\begin{eqnarray}\label{eq-413'}
\Big\|\int_{0}^{t}\mathscr{U}_{2}(t,s)P_{c}(s)G(s)ds\Big\|_{L^{2}_{t}L^{6}_{x}}\lesssim \|G\|_{L^{2}_{t}L^{6/5}_{x}}.
\end{eqnarray}
Moreover,
\begin{eqnarray}\label{eq-414'}
\int_{\mathbb{R}}\big\|\mathscr{U}_{2}(t,s)P_{c}(s)f\big\|_{L^{6,\infty}}dt\lesssim \|f\|_{L^{6/5,1}}
\end{eqnarray}
 uniformly for all $s\in\mathbb{R}$.
\end{lem}

\begin{proof}Let us first consider (\ref{eq-413}).
Notice that by using (\ref{eq-49}), it suffices to prove that (\ref{eq-413})  hold for $\overline{P}_{c}(t)U_{2}(t)$ with $U_{2}(t,s)$ defined as in (\ref{eq-410'}) and $U_{2}(t)=U_{2}(t,0)$.
To this end, consider the linear equation (\ref{eq-410'}) and  
and denote by $\psi(t)=\overline{P}_{c}(t)U_{2}(t)f\in {\rm Ran}\overline{P}_{c}(t)$, it follows from
\cite[Theorem 1.7]{Bec3} that
\begin{eqnarray}\label{eq-414-6}
\big\|\psi(t)\big\|_{L^{\infty}_{t}L^{2}_{x}\cap L^{2}_{t}L^{6}_{x}}
&\lesssim&\|f\|_{L^{2}}+\big\|\big[\overline{P}_{c}'(t),\overline{P}_{c}(t)\big]\psi(t)\big\|_{L^{2}_{t}L^{6/5}_{x}}+\|\overline{P}_{b}(t)\psi(t)\|_{L^{2}_{t}L^{6}_{x}}
\end{eqnarray}
The last terms in (\ref{eq-414-6}) disappears naturally.
On the other hand,
\begin{eqnarray}\label{eq-414-1}
\big[\overline{P}_{c}'(t),\overline{P}_{c}(t)\big]=\big[\overline{P}_{b}'(t),\overline{P}_{b}(t)\big],
\end{eqnarray}
with $\overline{P}_{b}(t)=T(t)P_{b}(\sigma_{T_{0}})T^{\ast}(t)$ and
\begin{eqnarray}\label{eq-414-2}
T'(t)=T(t)\big(i\beta'(t)-y'(t)\nabla\big).\nonumber
\end{eqnarray}
As a consequence, we have
\begin{eqnarray}\label{eq-414-4}
\big[\overline{P}_{b}'(t),\overline{P}_{b}(t)\big]=T(t)\big[[i\beta'(t)-y'(t)\nabla,P_{b}(\sigma_{T_{0}})],P_{b}(\sigma_{T_{0}})\big]T^{\ast}(t)\nonumber
\end{eqnarray}
and there exist some localized functions $\Phi_{j}$ and $\Psi_{k}$ centered at $b(t)$ such that
\begin{eqnarray}\label{eq-414-5}
\big[\overline{P}_{b}'(t),\overline{P}_{b}(t)\big]f=\big(\beta'(t)-y'(t)\big)\sum_{j,k}\Phi_{j}\big\langle\Psi_{k},f\big\rangle,
\end{eqnarray}
where
\begin{eqnarray}\label{eq-458}
b(t)=y_{0}(t)+y(t),
\end{eqnarray}
and the sum is finite. Moreover,
\begin{eqnarray}\label{eq-414-3}
\|\beta'(t)\|_{L^{\infty}_{t}}+\|y'(t)\|_{L^{\infty}_{t}}\lesssim \|\dot{\sigma}(t)\|_{L^{1}_{t}}\leq \epsilon^{2\alpha}.
\end{eqnarray}
Thus,
\begin{eqnarray}
\big\|\big[\overline{P}_{c}'(t),\overline{P}_{c}(t)\big]\psi(t)\big\|_{L^{2}_{t}L^{6/5}_{x}}\lesssim \epsilon^{2\alpha}\|\psi(t,s)\|_{L^{2}_{t}L^{6}_{x}}\leq \frac{1}{2}\|\psi(t,s)\|_{L^{2}_{t}L^{6}_{x}}\nonumber
\end{eqnarray}
which combining (\ref{eq-414-6}) lead to desire estimates.

As far as (\ref{eq-413'}). Notice that $$\int_{0}^{t}\mathscr{U}_{2}(t,s)P_{c}(s)G(s)ds\in {\rm Ran}P_{c}(t),$$ applying \cite[Theorem 1.7]{Bec3} for equation (\ref{eq-413''})  with initial data $f=0$ and the same argument as above we obtain
 (\ref{eq-413'}).

Concerning (\ref{eq-414'}), we only need to prove it  for $U_{2}(t,s)\overline{P}_{c}(s)$.  Let  $\psi(t)=\overline{P}_{c}(t)U_{2}(t)f\in {\rm Ran}\overline{P}_{c}(t)$ be defined as before, it follows from
\cite[Theorem 1.7]{Bec3} that
\begin{eqnarray}\label{eq-414-7}
\|\psi(t)\|_{L^{1}_{t}L^{6,\infty}_{x}}
&\lesssim& \|f\|_{L^{6/5,1}}+\big\|\big[\overline{P}_{c}'(t),\overline{P}_{c}(t)\big]\psi(t)\big\|_{L^{1}_{t}L^{6/5,1}_{x}}+\|\overline{P}_{b}(t)\psi(t)\|_{L^{1}_{t}L^{6,\infty}_{x}}.\nonumber
\end{eqnarray}
The last term disappears and
\begin{eqnarray}
\big\|\big[\overline{P}_{c}'(t),\overline{P}_{c}(t)\big]\psi(t)\big\|_{L^{1}_{t}L^{6/5,1}_{x}}\lesssim \epsilon^{2\alpha}\|\psi(t)\|_{L^{1}_{t}L^{6,\infty}_{x}}\leq \frac{1}{2}\|\psi(t)\|_{L^{1}_{t}L^{6/5,1}_{x}},\nonumber
\end{eqnarray}
which imply that
\begin{eqnarray}\label{eq-414-7-1}
\big\|\overline{P}_{c}(t)U_{2}(t)f\big\|_{L^{1}_{t}L^{6,\infty}_{x}}
&\lesssim& \|f\|_{L^{6/5,1}}.
\end{eqnarray}
We now prove (\ref{eq-414'}) uniformly in $s\in \mathbb{R}$. Consider the linear equation (\ref{eq-410'}) and denote by $\psi(t,s)=\overline{P}_{c}(t)U_{2}(t,s)f$, it follow from  Duhamel's formula and the commutative property (\ref{eq-410'''}) that
\begin{eqnarray}\label{eq-414-7-2}
\big\|\psi(t,s)\big\|_{L^{1}_{t}L^{6,\infty}_{x}}&\leq& \big\|\mathscr{U}_{0}(t-s)f\big\|_{L^{1}_{t}L^{6,\infty}_{x}}+ \big\|\int^{t}_{s}\mathscr{U}_{0}(t-\tau)T(\tau)\mathcal {V}_{2}(\sigma_{T_{0}})T^{\ast}(\tau)\psi(\tau)d\tau\big\|_{L^{1}_{t}L^{6,\infty}_{x}}\nonumber\\
&&\ \ \ \ \ \ +\ \big\|\int^{t}_{s}\mathscr{U}_{0}(t-\tau)[\overline{P}_{c}'(\tau),\overline{P}_{c}(\tau)\big]\psi(\tau)d\tau\big\|_{L^{1}_{t}L^{6,\infty}_{x}}\nonumber\\
&\lesssim&\|f\|_{L^{6/5,1}}.
\end{eqnarray}
Here $\mathscr{U}_{0}(t)$ is the linear propagator of $\mathscr{H}_{0}$ and in the second inequality, we use the fact that $\big\|\mathscr{U}_{0}(t)\big\|_{L^{6/5,1}_{x}\rightarrow L^{6,\infty}_{x}}\in L^{1}$ (see \cite[Proposition 1.2]{Bec3}),
(\ref{eq-414-7-1}) and Young's inequality.

\end{proof}

\begin{lem}\label{le-42}
Let $\mathscr{U}_{1}(t)$ be the linear flow of  Schr\"{o}dinger equation
\begin{eqnarray}\label{eq-416}
&&i\psi_{t}(t)=\big(\mathscr{H}_{0}+\mathcal {V}_{1\epsilon}\big)\psi(t)+G(t)\nonumber
\end{eqnarray}
with initial data $\psi(0)=f$. Then we have for admissible pairs $(p,q)$ and $(\tilde{p},\tilde{q})$,
\begin{eqnarray}\label{eq-414}
\big\|\mathscr{U}_{1}(t)f\big\|_{L^{p}_{t}L^{q}_{x}}\lesssim \|f\|_{L^{2}}\ {\rm and}\ \Big\|\int_{0}^{t}\mathscr{U}_{1}(t,s)G(s)ds\Big\|_{L^{p}_{t}L^{q}_{x}}\lesssim \|G\|_{L^{\tilde{p}'}_{t}L^{\tilde{q}'}_{x}}.
\end{eqnarray}
Moreover,
\begin{eqnarray}\label{eq-415}
\int_{\mathbb{R}}\big\|\mathscr{U}_{1}(t,s)f\big\|_{L^{6,\infty}}dt\lesssim \|f\|_{L^{6/5,1}}
\end{eqnarray}
uniformly for all $s\in\mathbb{R}$.
\end{lem}

\begin{proof}
It suffices to prove this lemma for $s=0$ and the proof would be done by using scaling since $V_{\epsilon}(\cdot)=\epsilon^{2}V(\epsilon\cdot)$. Specifically, it is easy to see that $\psi_{\epsilon}(t)=\psi(\epsilon^{-1}x,\epsilon^{-2}t)$ is the solution of Schr\"{o}dinger equation
\begin{eqnarray}
&&i\partial_{t}\psi_{\epsilon}(t)=\big(\mathscr{H}_{0}+\mathcal {V}_{1}\big)\psi_{\epsilon}(t)+G_{\epsilon}(t),\nonumber\\
&&\psi_{\epsilon}(s)=f(\epsilon^{-1}\cdot):=f_{\epsilon}.\nonumber
\end{eqnarray}
Notice that for admissible pairs $(p,q)$, it follows from \cite[Theorem 1.3]{Bec3}
\begin{eqnarray}
\|\psi_{\epsilon}\|_{L^{p}_{t}L^{q}_{x}}\lesssim \|f_{\epsilon}\|_{L^{2}}+\|G_{\epsilon}\|_{L^{\tilde{p}'}_{t}L^{\tilde{q}'}_{x}},\nonumber
\end{eqnarray}
which imply the desire Strichartz estimates (\ref{eq-414}). Furthermore, we have (see \cite[Theorem 1.3]{Bec3})
$$\int_{\mathbb{R}}\big\|\psi_{\epsilon}\big\|_{L^{6,\infty}}dt\lesssim \|f_{\epsilon}\|_{L^{6/5,1}}+\|G_{\epsilon}\|_{L^{1}_{t}L^{\frac{6}{5},1}_{x}},$$
which combing the same argument as the one in the proof of Lemma \ref{le-41} lead to  (\ref{eq-415}).
\end{proof}

\subsection{The completion of proof }

 In order to finish the proof of Theorem \ref{thm-2}, we divide the proof into several steps as follows:
\begin{proof}
\emph{Step I. Reduce to local decay estimates.} Let $(p,q)$ and $(\tilde{p},\tilde{q})$ be admissible pairs. Assume that
\begin{eqnarray}\label{eq-303'}
\big\|\mathcal {V}^{1/2}_{1\epsilon}Z(t)\big\|^{2}_{L^{2}_{t}L^{2}_{x}}+\big\|\mathcal {V}^{1/2}_{2}(t)Z(t)\big\|_{L^{2}_{t}L^{2}_{x}}\lesssim \|Z_{0}\|_{L^{2}_{x}}+\|F\|_{L^{\tilde{p}'}_{t}L^{\tilde{q}'}_{x}}+B.
\end{eqnarray}

We write equation (\ref{eq-302}) as
\begin{eqnarray}
i\partial_{t}Z(t)=\mathscr{H}_{0}Z(t)+\mathcal{V}_{2}(t)Z(t)+\mathcal {V}_{1\epsilon}Z(t)+F(t),\nonumber
\end{eqnarray}
the Duhamel formula leads to
\begin{eqnarray}\label{eq-20}
\|Z(t)\|_{L_{t}^{p}L_{x}^{q}}&\leq& \big\|\mathscr{U}_{0}(t)Z_{0}\big\|_{L_{t}^{p}L_{x}^{q}}+\Big\|\int_{0}^{t}\mathscr{U}_{0}(t,s)F(s)ds\Big\|_{L_{t}^{p}L_{x}^{q}}\nonumber\\
&&\ \ \ \ \ \ \ \ \ \ \ \ \ \ \ \ \ +\
\Big\|\int_{0}^{t}\mathscr{U}_{0}(t,s)\big(\mathcal{V}_{1\epsilon}+\mathcal{V}_{2}(t)\big)Z(s)ds\Big\|_{L_{t}^{q}L_{x}^{q}}\nonumber\\
&\lesssim& \big\|Z_{0}\big\|_{L^{2}}+\big\|F(t)\big\|_{L^{\tilde{p}'}_{t}L^{\tilde{q}'}_{x}}+\big\|(\mathcal{V}_{1\epsilon}+\mathcal{V}_{2}(t))Z(t)\big\|_{L_{t}^{2}L_{x}^{\frac{6}{5}}}\nonumber\\
&\lesssim& \big\|Z_{0}\big\|_{L^{2}}+\big\|F(t)\big\|_{L^{\tilde{p}'}_{t}L^{\tilde{q}'}_{x}}+B,\nonumber
\end{eqnarray}
where in the last inequality we  the fact (follows from (\ref{eq-303'})) that
\begin{eqnarray}
&&\big\|(\mathcal{V}_{1}+\mathcal{V}_{2}(t))Z(t)\big\|_{L_{t}^{2}L_{x}^{\frac{6}{5}}}\nonumber\\
&\lesssim&  \big\|\mathcal {V}^{1/2}_{1\epsilon}\big\|_{L^{3}_{x}}\big\|\mathcal {V}^{1/2}_{1\epsilon}Z(t)\big\|^{2}_{L^{2}_{t}L^{2}_{x}}+\big\|\mathcal {V}^{1/2}_{2}(t)\big\|_{L^{\infty}_{t}L^{3}_{x}}\big\|\mathcal {V}^{1/2}_{2}(t)Z(t)\big\|_{L^{2}_{t}L^{2}_{x}}\nonumber\\
&\lesssim& \|Z_{0}\|_{L^{2}_{x}}+\|F\|_{L^{\tilde{p}'}_{t}L^{\tilde{q}'}_{x}}+B.\nonumber
\end{eqnarray}

\emph{Step II. Local decay estimates.}
We will prove for  arbitrary compact supported smooth function $\mathcal {V}_{1}$ and  localized function $\mathcal {V}_{2}$, the local decay estimates hold.
Let $\delta>0$ be some fixed small number and $b(t)$ be defined by (\ref{eq-458})
we introduce a partition of unity associated with the sets
\begin{eqnarray}
B_{\delta\epsilon t}(0)=\big\{x:|x|<\delta\epsilon t\big\}, \ \ \ \ \ \ \ B_{\delta\epsilon t}(b(t))=\big\{x:|x-b(t)|<\delta\epsilon t\big\}\nonumber
 \end{eqnarray}
and
\begin{eqnarray}
\mathbb{R}^{3} \setminus\big(B_{\delta\epsilon t}(0)\cup B_{\delta\epsilon t}(b(t))\big).\nonumber
\end{eqnarray}
Let $\chi_{1}(t,x)$ be a cut-off function such that
\begin{eqnarray}
\chi_{1}(t,x)=1,\ x\in B_{\delta\epsilon t}(0) \ \ \ \text{and}\ \ \ \chi_{1}(t,x)=0,\ x\in \mathbb{R}^{3} \setminus B_{2\delta\epsilon t}(0)\nonumber
\end{eqnarray}
and define
\begin{eqnarray}
\chi_{2}(t,x)=\chi_{1}(t,x-b(t)),\ \ \chi_{3}(t,x)=1-\chi_{1}(t,x)-\chi_{2}(t,x).\nonumber
\end{eqnarray}
Observe that the supports of $\chi_{2}(t,\cdot)$ and $\mathcal{V}_{1\epsilon}(\cdot)$ are disjoint and $\mathcal{V}_{2}(t,\widetilde{\sigma}(t))\chi_{1}(t,\cdot)$ is arbitrary small since the separation condition (\ref{eq-211}) holds, which will be used in the further.

It follows the decomposition of the solution $Z(t)$:
\begin{eqnarray}
Z(t)=\chi_{1}(t,\cdot)Z(t)+\chi_{2}(t,\cdot)P_{b}(t)Z(t)+\chi_{2}(t,\cdot)P_{c}(t)Z(t)+\chi_{3}(t,\cdot)Z(t).\nonumber
\end{eqnarray}
Thus it suffices to estimate the $L^{2}_{t}L^{2}_{x}$ norm of
\begin{eqnarray}\label{eq-47}
&&\mathcal {V}^{1/2}_{2}(t,\widetilde{\sigma}(t))\chi_{1}(t,\cdot)Z(t),\ \mathcal {V}^{1/2}_{2}(t,\widetilde{\sigma}(t))\chi_{2}(t,\cdot)P_{b}(t)Z(t),\nonumber\\
&& \mathcal {V}^{1/2}_{2}(t,\widetilde{\sigma}(t))\chi_{2}(t,\cdot)P_{c}(t)Z(t),\ \mathcal {V}^{1/2}_{2}(t,\widetilde{\sigma}(t))\chi_{3}(t,\cdot)Z(t)
\end{eqnarray}
and
\begin{eqnarray}\label{eq-48}
&&\mathcal {V}^{1/2}_{1\epsilon}\chi_{1}(t,\cdot)Z(t),\ \mathcal {V}^{1/2}_{1\epsilon}\chi_{2}(t,\cdot)P_{b}(t)Z(t),\nonumber\\
 &&\mathcal {V}^{1/2}_{1\epsilon}\chi_{2}(t,\cdot)P_{c}(t)Z(t),\ \mathcal {V}^{1/2}_{1\epsilon}\chi_{3}(t,\cdot)Z(t).
\end{eqnarray}
Notice that
\begin{eqnarray}\label{eq-46}
&&\big\|\mathcal {V}^{1/2}_{2}(t,\widetilde{\sigma}(t))\chi_{2}(t,\cdot)P_{b}(t)Z(t)\big\|_{L^{2}_{t}L^{2}_{x}}+\big\|\mathcal {V}^{1/2}_{1\epsilon}\chi_{2}(t,\cdot)P_{b}(t)Z(t)\big\|_{L^{2}_{t}L^{2}_{x}}\nonumber\\
&\leq&\big\|\mathcal {V}^{1/2}_{2}(t,\widetilde{\sigma}(t))\chi_{2}(t,\cdot)\big\|_{L^{\infty}_{t}L^{3}_{x}}\big\|P_{b}(t)Z(t)\big\|_{L^{2}_{t}L^{6}_{x}}+
\big\|\mathcal {V}^{1/2}_{1\epsilon}\chi_{2}(t,\cdot)\big\|_{L^{\infty}_{t}L^{3}_{x}}\big\|P_{b}(t)Z(t)\big\|_{L^{2}_{t}L^{6}_{x}}\lesssim B,\nonumber
\end{eqnarray}
we only need to estimate the rest of the terms in (\ref{eq-47}) and (\ref{eq-48}). Here and in the following, we will use $\mathcal {V}_{2}(t)$ and $\chi_{i}$ instead of $\mathcal {V}_{2}(t,\widetilde{\sigma}(t))$ and $\chi_{i}(t)$ $(i=1,2,3)$.

Consider the homogeneous equation
\begin{eqnarray}\label{eq-412}
&&i\partial_{t}Z(t)=\mathscr{H}(t,\widetilde{\sigma}(t))Z(t)+F(t)
=\mathscr{H}_{2}(t)Z(t)+\mathcal {V}_{1\epsilon}Z(t)-E(t)Z(t)+F(t)\nonumber\\
&&Z(0)=Z_{0}
\end{eqnarray}
with $\|P_{b}(t)Z(t)\|_{L^{2}_{t}L^{6}_{x}}\lesssim B$ where $P_{b}(t)$ is defined by (\ref{eq-45}). Denote by$\mathscr{U}(t)$  the propagator of the homogeneous equation (\ref{eq-412}) (with $F=0$).

By using Duhamel's formula
\begin{eqnarray}\label{eq-412'}
Z(t)=\mathscr{U}_{0}(t)Z_{0}-i\int^{t}_{0}\mathscr{U}_{0}(t,s)\big(\mathcal {V}_{2}(s)+\mathcal {V}_{1\epsilon}\big)Z(s)ds-i\int^{t}_{0}\mathscr{U}_{0}(t,s)F(s)ds,\nonumber\\
\end{eqnarray}
we obtain
\begin{eqnarray}\label{eq-418}
\big\|\mathcal {V}^{1/2}_{1\epsilon}Z(t)\big\|_{L^{2}_{x}}&\leq& \big\|\mathcal {V}^{1/2}_{1\epsilon}\mathscr{U}_{0}(t)Z_{0}\big\|_{L^{2}_{x}}+\int^{t}_{0}\big\|\mathcal {V}^{1/2}_{1\epsilon}\mathscr{U}_{0}(t,s)\big(\mathcal {V}_{2}(s)+\mathcal {V}_{1\epsilon}\big)Z(s)\big\|_{L^{2}_{x}}ds\nonumber\\
&&\ \ \ \ \ \ \ \ \ \ \ \ \ \ \ \ +\ \int^{t}_{0}\big\|\mathcal {V}^{1/2}_{1\epsilon}\mathscr{U}_{0}(t,s)F(s)\big\|_{L^{2}_{x}}ds\nonumber\\
&\leq&\big\|\mathcal {V}^{1/2}_{1\epsilon}\mathscr{U}_{0}(t)Z_{0}\big\|_{L^{2}_{x}}+\big\|\mathcal {V}^{1/2}_{1\epsilon}\big\|^{2}_{L^{3,2}_{x}}\int^{t}_{0}\big\|\mathscr{U}_{0}(t,s)\big\|_{L^{6/5,1}_{x}\rightarrow L^{6,\infty}_{x}}\big\|\mathcal {V}^{1/2}_{1\epsilon}Z(s)\big\|_{L^{2}_{x}}ds\nonumber\\
&&\ +\ \big\|\mathcal {V}^{1/2}_{1\epsilon}\big\|_{L^{3,2}_{x}}\big\|\mathcal {V}^{1/2}_{2}(t)\big\|_{L^{3,2}_{x}}\int^{t}_{0}\big\|\mathscr{U}_{0}(t,s)\big\|_{L^{6/5,1}_{x}\rightarrow L^{6,\infty}_{x}}\big\|\mathcal {V}^{1/2}_{2}(s)Z(s)\big\|_{L^{2}_{x}}ds\nonumber\\
&&\ \ \ \ \ \ +\ \big\|\mathcal {V}^{1/2}_{1\epsilon}\big\|_{L^{3}_{x}}\big\|\int^{t}_{0}\mathscr{U}_{0}(t,s)F(s)ds\big\|_{L^{6}_{x}}.
\end{eqnarray}
Denote by
\begin{eqnarray}\label{eq-418'}
g(t)=\big\|\mathcal {V}^{1/2}_{1\epsilon}Z(t)\big\|_{L^{2}_{x}}+\big\|\mathcal {V}^{1/2}_{2}(t)Z(t)\big\|_{L^{2}_{x}},
\end{eqnarray}
$$d_{0}(t)=\big\|\mathcal {V}^{1/2}_{1\epsilon}\mathscr{U}_{0}(t)Z_{0}\big\|_{L^{2}_{x}}+\big\|\mathcal {V}^{1/2}_{2}(t)\mathscr{U}_{0}(t)Z_{0}\big\|_{L^{2}_{x}}+\big\|\int^{t}_{0}\mathscr{U}_{0}(t,s)F(s)ds\big\|_{L^{6}_{x}}$$
and
$$h_{0}(t)=\big\|\mathscr{U}_{0}(t)\big\|_{L^{6/5,1}_{x}\rightarrow L^{6,\infty}_{x}},$$
it follows that $d_{0}\in L^{2}_{t}$ with
\begin{eqnarray}
\|d_{0}\|_{L^{2}_{t}}&\lesssim& \big(\big\|\mathcal {V}^{1/2}_{1\epsilon}\big\|_{L^{3}_{x}}+\big\|\mathcal {V}^{1/2}_{2}(t)\big\|_{L^{3}_{x}}\big)\big\|\mathscr{U}_{0}(t)Z_{0}\big\|_{L^{2}_{t}L^{6}_{x}}+\big\|\int^{t}_{0}\mathscr{U}_{0}(t,s)F(s)ds\big\|_{L^{2}_{t}L^{6}_{x}}\nonumber\\
&\lesssim&\|Z_{0}\|+\|F\|_{L^{\tilde{p}'}_{t}L^{\tilde{q}'}_{x}},\nonumber
\end{eqnarray}
 $h\in L^{1}_{t}$ and
\begin{eqnarray}
g(t)\leq C\Big(d_{0}(t)+\int^{t}_{0}h_{0}(t-s)g(s)ds\Big).\nonumber
\end{eqnarray}
Then for fix large enough $T$, by using Gronwall's inequality we can find a large constant $C(T)$ such that
\begin{eqnarray}\label{eq-419}
\int_{0}^{T}\big\|\mathcal {V}^{1/2}_{1\epsilon}Z(t)\big\|^{2}_{L^{2}_{x}}dt+\int_{0}^{T}\big\|\mathcal {V}^{1/2}_{2}(t)Z(t)\big\|_{L^{2}_{x}}dt\leq C^{2}(T)\big(\|Z_{0}\|^{2}_{L^{2}_{x}}+\big\|F\big\|^{2}_{L^{\tilde{p}'}_{t}L^{\tilde{q}'}_{x}}\big).
\end{eqnarray}
which combing Duhamel's formula (\ref{eq-412'}) further imply that for any localized function $\mathcal {V}$ and $c(t)=0\ {\rm or}\ b(t)$,
\begin{eqnarray}\label{eq-419'}
\int_{0}^{T}\big\|\mathcal {V}(\cdot-c(t))Z(t)\big\|^{2}_{L^{2}_{x}}dt\leq C^{2}(T)\big(\|Z_{0}\|^{2}_{L^{2}_{x}}+\big\|F\big\|^{2}_{L^{\tilde{p}'}_{t}L^{\tilde{q}'}_{x}}\big).
\end{eqnarray}
Next we will show that the constant $C(T)$ in (\ref{eq-419}) can be taken independent of $T$, this could be done by following the bootstrap argument in [RSS] which is based on the observation that it is enough to show if (\ref{eq-419}) holds for $C(T)$, it also hold for $C(T)/2$. Actually, one only need to prove (\ref{eq-419}) for $\int^{T}_{\tilde{M}}$ for some large positive constant $\tilde{M}\ll T$ which is independent of $T$.
We will do it channel by channel and begin with  the estimates for $\mathcal {V}^{1/2}_{2}(t)\chi_{2}P_{c}(t)Z(t)$.

\emph{$II_{1}$. Local decay for $\mathcal {V}^{1/2}_{2}(t)\chi_{2}P_{c}(t)Z(t)$.}
Assume that $A$ is a large constant to be fixed later and $A\ll T$, it follows from Duhamel's formula that
\begin{eqnarray}\label{eq-411}
\mathcal {V}^{1/2}_{2}(t)\chi_{2}P_{c}(t)Z(t)&=&\mathcal {V}^{1/2}_{2}(t)\chi_{2}P_{c}(t)\mathscr{U}_{2}(t)Z_{0}-i\mathcal {V}^{1/2}_{2}(t)\chi_{2}\int^{t-A}_{0}\mathscr{U}_{2}(t,s)P_{c}(s)\mathcal {V}_{1\epsilon}Z(s)ds\nonumber\\
&&\ \ \ \ \ \ \ -\ i\mathcal {V}^{1/2}_{2}(t)\chi_{2}\int^{t}_{t-A}\mathscr{U}_{2}(t,s)P_{c}(s)\mathcal {V}_{1\epsilon}Z(s)ds\nonumber\\
&&\ \ \ \ \ \ \ \ \ \ +\ i\mathcal {V}^{1/2}_{2}(t)\chi_{2}\int^{t}_{0}\mathscr{U}_{2}(t,s)P_{c}(s)E(s)Z(s)ds\nonumber\\
&&\ \ \ \ \ \ \ \ \ \ \ \ \ -\ i\mathcal {V}^{1/2}_{2}(t)\chi_{2}\int^{t}_{0}\mathscr{U}_{2}(t,s)P_{c}(s)F(s)ds
\end{eqnarray}
By using Lemma \ref{le-41}, we have
\begin{eqnarray}\label{eq-417}
&&\big\|\mathcal {V}^{1/2}_{2}(t)\chi_{2}P_{c}(t)\mathscr{U}_{2}(t,0)Z_{0}\big\|_{L^{2}_{t}L^{2}_{x}}\nonumber\\
&\leq&\big\|\mathcal {V}^{1/2}_{2}(t)\chi_{2}\big\|_{L^{\infty}_{t}L^{3}_{x}}\big\|P_{c}(t)\mathscr{U}_{2}(t,0)Z_{0}\big\|_{L^{2}_{t}L^{6}_{x}}\lesssim \|Z_{0}\|_{L^{2}}.
\end{eqnarray}
and
\begin{eqnarray}\label{eq-417'}
&&\big\|\mathcal {V}^{1/2}_{2}(t)\chi_{2}\int^{t}_{0}\mathscr{U}_{2}(t,s)P_{c}(s)F(s)ds\big\|_{L^{2}_{t}L^{2}_{x}}\nonumber\\
&\leq&\big\|\mathcal {V}^{1/2}_{2}(t)\chi_{2}\big\|_{L^{\infty}_{t}L^{3}_{x}}\big\|\int^{t}_{0}\mathscr{U}_{2}(t,s)P_{c}(s)F(s)ds\big\|_{L^{2}_{t}L^{6}_{x}}
\lesssim\|F\big\|_{L^{\tilde{p}'}_{t}L^{\tilde{q}'}_{x}}.
\end{eqnarray}
Now let us give the formula for $E(t)$,
\begin{eqnarray}\label{eq-420}
E(t)=iT_{0}(t)\big[\overline{P}_{c}'(t),\overline{P}_{c}(t)\big]T^{\ast}_{0}(t)=iT_{0}(t)\big[\overline{P}_{b}'(t),\overline{P}_{b}(t)\big]T^{\ast}_{0}(t),
\end{eqnarray}
with $\overline{P}_{b}(t)=T(t)P_{b}(\sigma_{T_{0}})T^{\ast}(t)$ and
\begin{eqnarray}\label{eq-421}
T'(t)=T(t)\big(i\beta'(t)-y'(t)\nabla\big).\nonumber
\end{eqnarray}
Then similarly to (\ref{eq-414-3}),
there exist some localized functions $\tilde{\Phi}_{j}$ and $\tilde{\Psi}_{k}$ centered at $b(t)$ such that
\begin{eqnarray}\label{eq-424}
E(t)f=\big(\beta'(t)-y'(t)\big)\sum_{j,k}\tilde{\Phi}_{j}\big\langle\tilde{\Psi}_{k},f\big\rangle,
\end{eqnarray}
where the sum is finite.
Thus it follows Lemma \ref{le-42}, (\ref{eq-414-3}), (\ref{eq-419})  and Shur's Lemma that
\begin{eqnarray}\label{eq-425}
&&\Big\|\chi_{2}\int^{t}_{0}\big\|\mathcal {V}^{1/2}_{2}(t)\mathscr{U}_{2}(t,s)P_{c}(s)E(s)Z(s)\big\|_{L^{2}_{x}}ds\Big\|_{L^{2}_{t}}\nonumber\\
&\leq&\big\|\mathcal {V}^{1/2}_{2}(t)\big\|_{L^{3,2}_{x}}\Big\|\int^{t}_{0}h_{2}(t,s)\|E(s)Z(s)\|_{L_{x}^{6/5,1}}ds\Big\|_{L^{2}_{t}}\nonumber\\
&\lesssim& \big\|\dot{\sigma}\|_{L^{1}_{t}}\Big\|\int^{t}_{0}h_{2}(t,s)\big\|\mathcal{V}Z(s)\big\|_{L^{2}_{x}}ds\Big\|_{L^{2}_{t}}\nonumber\\
&\lesssim&\epsilon^{2\alpha}C(T)\big(\|Z_{0}\|_{L^{2}_{x}}+\big\|F\big\|_{L^{\tilde{p}'}_{t}L^{\tilde{q}'}_{x}})\leq \frac{C(T)}{2}\big(\|Z_{0}\|_{L^{2}_{x}}+\big\|F\big\|_{L^{\tilde{p}'}_{t}L^{\tilde{q}'}_{x}}),
\end{eqnarray}
with  some localized function $\mathcal{V}$ and
\begin{eqnarray}\label{eq-426}
h_{2}(t,s)=\big\|\mathscr{U}_{2}(t,s)P_{c}(s)\big\|_{L^{6/5,1}_{x}\rightarrow L^{6,\infty}_{x}}.
\end{eqnarray}
For the second term in (\ref{eq-411}),
\begin{eqnarray}\label{eq-427}
&&\mathcal {V}^{1/2}_{2}(t)\chi_{2}\int^{t-A}_{0}\mathscr{U}_{2}(t,s)P_{c}(s)\mathcal {V}_{1\epsilon}Z(s)ds\nonumber\\
&=&\mathcal {V}^{1/2}_{2}(t)\chi_{2}\Big(\int^{A}_{0}\mathscr{U}_{2}(t,s)P_{c}(s)\mathcal {V}_{1\epsilon}Z(s)ds+\int^{t-A}_{A}\mathscr{U}_{2}(t,s)P_{c}(s)\mathcal {V}_{1\epsilon}Z(s)ds\Big).
\end{eqnarray}
It follows from Schur's Lemma and the bootstrap assumption (\ref{eq-419}) that
\begin{eqnarray}\label{eq-428}
&&\int_{0}^{T}\big\|\mathcal {V}^{1/2}_{2}(t)\chi_{2}\int^{A}_{0}\mathscr{U}_{2}(t,s)P_{c}(s)\mathcal {V}_{1\epsilon}Z(s)ds\big\|^{2}_{L^{2}_{x}}dt\nonumber\\
&\leq&\int_{0}^{2A}\big\|\mathcal {V}^{1/2}_{2}(t)\chi_{2}\int^{A}_{0}\mathscr{U}_{2}(t,s)P_{c}(s)\mathcal {V}_{1\epsilon}Z(s)ds\big\|^{2}_{L^{2}_{x}}dt\nonumber\\
&&\ \ \ \ \ \ \ +\ \int_{2A}^{T}\big\|\mathcal {V}^{1/2}_{2}(t)\chi_{2}\int^{A}_{0}\mathscr{U}_{2}(t,s)P_{c}(s)\mathcal {V}_{1\epsilon}Z(s)ds\big\|^{2}_{L^{2}_{x}}dt\nonumber\\
&\lesssim& \big\|\mathcal {V}^{1/2}_{2}(t)\chi_{2}\big\|^{2}_{L^{\infty}_{t,x}}\big\|\mathcal {V}^{1/2}_{1\epsilon}\big\|^{2}_{L^{\infty}_{x}}
\int_{0}^{2A}\Big(\int^{A}_{0}\big\|\mathscr{U}_{2}(t,s)P_{c}(s)\big\|_{L^{2}\rightarrow L^{2}}\big\|\mathcal {V}^{1/2}_{1\epsilon}Z(s)\big\|_{L^{2}_{x}}ds\Big)^{2}dt\nonumber\\
&&\ \ \ \ \ \ \ \ +\ \big\|\mathcal {V}^{1/2}_{2}(t)\chi_{2}\big\|^{2}_{L^{\infty}_{t}L^{3}_{x}}\big\|\mathcal {V}^{1/2}_{1\epsilon}\big\|^{2}_{L^{3}_{x}}
\int_{2A}^{T}\Big(\int^{A}_{0}h_{2}(t,s)\big\|\mathcal {V}^{1/2}_{1\epsilon}Z(s)\big\|_{L^{2}_{x}}ds\Big)^{2}dt\nonumber\\
&\lesssim& \big(C_{A}+k_{1}(A)C^{2}(T)\big)\int_{0}^{A}\big\|\mathcal {V}^{1/2}_{1\epsilon}Z(s)\big\|^{2}_{L^{2}_{x}}ds\nonumber\\
&\leq& \frac{C^{2}(T)}{2}\big(\|Z_{0}\|^{2}_{L^{2}_{x}}+\big\|F\big\|^{2}_{L^{\tilde{p}'}_{t}L^{\tilde{q}'}_{x}}\big),\nonumber
\end{eqnarray}
where $C_{A}$ is independent of $T$ and it follows from the proof of (\ref{eq-414-7-2}) that for sufficient large $A$,
\begin{eqnarray}\label{eq-429}
k_{1}(A)=\Big(\sup_{2A<t<T}\int^{A}_{0}h_{2}(t,s)ds\Big)\Big(\sup_{0<s<A}\int^{T}_{2A}h_{2}(t,s)dt\Big)
\end{eqnarray}
is small. Similarly,
\begin{eqnarray}\label{eq-430}
&&\int_{0}^{T}\big\|\mathcal {V}^{1/2}_{2}(t)\chi_{2}\int^{t-A}_{A}\mathscr{U}_{2}(t,s)P_{c}(s)\mathcal {V}_{1\epsilon}Z(s)ds\big\|^{2}_{L^{2}_{x}}dt\nonumber\\
&\lesssim&\big\|\mathcal {V}^{1/2}_{2}(t)\chi_{2}\big\|^{2}_{L^{\infty}_{t}L^{3}_{x}}\big\|\mathcal {V}^{1/2}_{1\epsilon}\big\|^{2}_{L^{3}_{x}}
\int_{0}^{T}\Big(\int^{t-A}_{A}h_{2}(t,s)\big\|\mathcal {V}^{1/2}_{1\epsilon}Z(s)\big\|_{L^{2}_{x}}ds\Big)^{2}dt\nonumber\\
&\lesssim& k_{2}(A)C^{2}(T)\int_{0}^{A}\big\|\mathcal {V}^{1/2}_{1\epsilon}Z(s)\big\|^{2}_{L^{2}_{x}}ds
\leq \frac{1}{2}C^{2}(T)\big(\|Z_{0}\|^{2}_{L^{2}_{x}}+\big\|F\big\|^{2}_{L^{\tilde{p}'}_{t}L^{\tilde{q}'}_{x}}\big)
\end{eqnarray}
where
\begin{eqnarray}\label{eq-430'}
k_{2}(A)=\Big(\sup_{2A<t<T}\int^{t-A}_{A}h_{2}(t,s)ds\Big)\Big(\sup_{A<s<T-A}\int^{T}_{s+A}h_{2}(t,s)dt\Big)
\end{eqnarray}
is also a small constant for large $A$.

For the third term in (\ref{eq-411}),
\begin{eqnarray}\label{eq-431}
&&\mathcal {V}^{1/2}_{2}(t)\chi_{2}\int^{t}_{t-A}\mathscr{U}_{2}(t,s)P_{c}(s)\mathcal {V}_{1\epsilon}Z(s)ds=\mathcal {V}^{1/2}_{2}(t)\chi_{2}\int^{t}_{t-A}\mathscr{U}_{2}(t,s)P_{c}(s)\mathcal {V}_{1\epsilon}\mathscr{U}_{1}(s)Z_{0}ds\nonumber\\
&&\ \ \ \ \ \ \ \ \ \ \ \ +\ \mathcal {V}^{1/2}_{2}(t)\chi_{2}\int^{t}_{t-A}\mathscr{U}_{2}(t,s)P_{c}(s)\mathcal {V}_{1\epsilon}\int^{s}_{0}\mathscr{U}_{1}(s,\tau)\mathcal{V}_{2}(\tau)Z(\tau)d\tau ds\nonumber\\
&&\ \ \ \ \ \ \ \ \ \ \ \ \ \ +\ \mathcal {V}^{1/2}_{2}(t)\chi_{2}\int^{t}_{t-A}\mathscr{U}_{2}(t,s)P_{c}(s)\mathcal {V}_{1\epsilon}\int^{s}_{0}\mathscr{U}_{1}(s,\tau)F(\tau)d\tau ds.
\end{eqnarray}
By using Lemma \ref{le-41} and Young's inequality (Schur's Lemma), we have
\begin{eqnarray}\label{eq-432}
&&\Big\|\mathcal {V}^{1/2}_{2}(t)\chi_{2}\int^{t}_{t-A}\mathscr{U}_{2}(t,s)P_{c}(s)\mathcal {V}_{1\epsilon}\mathscr{U}_{1}(s)Z_{0}ds\Big\|_{L^{2}_{t}L^{2}_{x}}\nonumber\\
&\leq&\big\|\mathcal {V}^{1/2}_{2}(t)\chi_{2}\big\|_{L^{\infty}_{t}L^{3}_{x}}\big\|\mathcal {V}^{1/2}_{1\epsilon}\big\|_{L^{3}_{x}}\Big\|\int^{t}_{t-A}h_{2}(t,s)\big\|\mathcal {V}_{1\epsilon}^{1/2}\mathscr{U}_{1}(s)Z_{0}\big\|_{L^{2}_{x}}ds\Big\|_{L^{2}_{t}}\nonumber\\
&\lesssim& \big\|\mathcal {V}_{1\epsilon}^{1/2}\mathscr{U}_{1}(t)Z_{0}\big\|_{L^{2}_{t}L^{2}_{x}}\lesssim
\big\|\mathscr{U}_{1}(t)Z_{0}\big\|_{L^{2}_{t}L^{6}_{x}}\lesssim \big\|Z_{0}\big\|_{L^{2}_{x}}
\end{eqnarray}
and
\begin{eqnarray}\label{eq-433}
&&\Big\|\mathcal {V}^{1/2}_{2}(t)\chi_{2}\int^{t}_{t-A}\mathscr{U}_{2}(t,s)P_{c}(s)\mathcal {V}_{1\epsilon}\int^{s}_{0}\mathscr{U}_{1}(s,\tau)F(\tau)d\tau ds\Big\|_{L^{2}_{t}L^{2}_{x}}\nonumber\\
&\leq&\big\|\mathcal {V}^{1/2}_{2}(t)\chi_{2}\big\|_{L^{\infty}_{t}L^{3}_{x}}\big\|\mathcal {V}^{1/2}_{1\epsilon}\big\|_{L^{3}_{x}}\Big\|\int^{t}_{t-A}h_{2}(t,s)\big\|\int^{s}_{0}\mathcal {V}_{1\epsilon}^{1/2}\mathscr{U}_{1}(s,\tau)F(\tau)d\tau\big\|_{L^{2}_{x}}ds\Big\|_{L^{2}_{t}}\nonumber\\
&\lesssim&\Big\|\int^{s}_{0}\mathcal {V}_{1\epsilon}^{1/2}\mathscr{U}_{1}(s,\tau)F(\tau)d\tau\Big\|_{L^{2}_{s}L^{2}_{x}}\lesssim
\Big\|\int^{s}_{0}\mathscr{U}_{1}(s,\tau)F(\tau)d\tau\Big\|_{L^{2}_{s}L^{6}_{x}}\lesssim
\big\|F\big\|^{2}_{L^{\tilde{p}'}_{t}L^{\tilde{q}'}_{x}}.
\end{eqnarray}
It remains to estimate
\begin{eqnarray}\label{eq-434}
&&\mathcal {V}^{1/2}_{2}(t)\chi_{2}\int^{t}_{t-A}\mathscr{U}_{2}(t,s)P_{c}(s)\mathcal {V}_{1\epsilon}\int^{s}_{0}\mathscr{U}_{1}(s,\tau)\mathcal{V}_{2}(\tau)Z(\tau)d\tau ds\nonumber\\
&=&\mathcal {V}^{1/2}_{2}(t)\chi_{2}\int^{t}_{t-A}\mathscr{U}_{2}(t,s)P_{c}(s)\mathcal {V}_{1\epsilon}\int^{s-B}_{0}\mathscr{U}_{1}(s,\tau)\mathcal{V}_{2}(\tau)Z(\tau)d\tau ds\nonumber\\
&&\ \ \ \ \ \ \ \ \ \ +\ \mathcal {V}^{1/2}_{2}(t)\chi_{2}\int^{t}_{t-A}\mathscr{U}_{2}(t,s)P_{c}(s)\mathcal {V}_{1\epsilon}\int^{s}_{s-B}\mathscr{U}_{1}(s,\tau)\mathcal{V}_{2}(\tau)Z(\tau)d\tau ds\nonumber\\
&:=&J_{1}+J_{2}.
\end{eqnarray}
Here $B$  which is a large constant independent of $C(T)$ so that $T\gg B\gg A$ will be chosen later. For $J_{1}$, it follows from Young's inequality that
\begin{eqnarray}\label{eq-435}
&&\big\|J_{1}\big\|_{L^{2}_{t}L^{2}_{x}}\nonumber\\
&\leq& \big\|\mathcal {V}^{1/2}_{2}(t)\chi_{2}\big\|_{L^{\infty}_{t}L^{3}_{x}}\big\|\mathcal {V}^{1/2}_{1\epsilon}\big\|_{L^{3}_{x}}\Big\|\int^{t}_{t-A}h_{2}(t,s)\int^{s-B}_{0}\big\|\mathscr{U}_{1}(s,\tau)\mathcal{V}_{2}(\tau)Z(\tau)\big\|_{L^{2}_{x}}d\tau ds\Big\|_{L^{2}_{t}}\nonumber\\
&\lesssim&\Big\|\int^{s-B}_{0}\big\|\mathcal{V}^{1/2}_{1\epsilon}\mathscr{U}_{1}(s,\tau)\mathcal{V}_{2}(\tau)Z(\tau)\big\|_{L^{2}_{x}}d\tau\Big\|_{L^{2}_{t}}\nonumber\\
&\leq&\big\|\mathcal {V}^{1/2}_{2}(t)\chi_{2}\big\|_{L^{\infty}_{t}L^{3}_{x}}\big\|\mathcal {V}^{1/2}_{1\epsilon}\big\|_{L^{3}_{x}}
\Big\|\int^{s-B}_{0}h_{1}(s,\tau)\big\|\mathcal{V}_{2}^{1/2}(\tau)Z(\tau)\big\|_{L^{2}_{x}}d\tau\Big\|_{L^{2}_{t}}\nonumber
\end{eqnarray}
where
\begin{eqnarray}\label{eq-436}
h_{1}(t,s)=\big\|\mathscr{U}_{1}(t,s)\big\|_{L^{6/5,1}_{x}\rightarrow L^{6,\infty}_{x}}.
\end{eqnarray}
And then by using Lemma \ref{le-42} and the same argument as in (\ref{eq-428})-(\ref{eq-430'}), we obtain
\begin{eqnarray}\label{eq-437}
\big\|J_{1}\big\|_{L^{2}_{t}L^{2}_{x}}\leq \frac{1}{2}C(T)\big(\|Z_{0}\|_{L^{2}_{x}}+\big\|F\big\|_{L^{\tilde{p}'}_{t}L^{\tilde{q}'}_{x}}\big).
\end{eqnarray}
As far as $J_{2}$, let $M$ be large positive constant and $M\ll T$,
\begin{eqnarray}\label{eq-438}
J_{2}&=&\mathcal {V}^{1/2}_{2}(t)\chi_{2}\int^{t}_{t-A}\mathscr{U}_{2}(t,s)P_{c}(s)K_{\leq M}\mathcal {V}_{1\epsilon}\int^{s}_{s-B}\mathscr{U}_{1}(s,\tau)\mathcal{V}_{2}(\tau)Z(\tau)d\tau ds\nonumber\\
&&\ \ \ +\ \mathcal {V}^{1/2}_{2}(t)\chi_{2}\int^{t}_{t-A}\mathscr{U}_{2}(t,s)P_{c}(s)K_{> M}\mathcal {V}_{1\epsilon}\int^{s}_{s-B}\mathscr{U}_{1}(s,\tau)\mathcal{V}_{2}(\tau)Z(\tau)d\tau ds\nonumber\\
&=&J_{2}^{L}+J_{2}^{H},
\end{eqnarray}
where $K_{\leq M}:=K(|i\nabla|\leq M)$ and $K_{>M}:=K(|i\nabla|\leq M)$. The estimates for $J_{2}^{L}$ will be accomplished by using the following inequality,
\begin{eqnarray}\label{eq-439}
\sup_{|t-s|\leq A}\big\|\mathcal {V}^{1/2}_{2}(t)\chi_{2}\mathscr{U}_{2}(t,s)P_{c}(s)K_{\leq M}\mathcal {V}_{1\epsilon}^{1/2}\big\|_{L^{2}_{x}\rightarrow L^{2}_{x}}\leq \frac{AM}{\delta t}.
\end{eqnarray}
It is proved by [RSS] for $\epsilon=1$, we will prove that it holds uniformly for $\epsilon$ small later.
Then it follows from  the bootstrap assumption (\ref{eq-419}) and H\"{o}lder's inequality that
\begin{eqnarray}\label{eq-441}
\big\|J_{2}^{L}\big\|_{L^{2}_{t}L^{2}_{x}}&\leq& \Big\|\sup_{|t-s|\leq A}\big\|\mathcal {V}^{1/2}_{2}(t)\chi_{2}\mathscr{U}_{2}(t,s)P_{c}(s)K_{\leq M}\mathcal {V}_{1\epsilon}^{1/2}\big\|_{L^{2}_{x}\rightarrow L^{2}_{x}}\nonumber\\
&&\ \ \ \ \ \ \times \ \int^{t}_{t-A}\big\|\mathcal {V}_{1\epsilon}^{1/2}\int^{s}_{s-B}\mathscr{U}_{1}(s,\tau)\mathcal{V}_{2}(\tau)Z(\tau)d\tau\big\|_{L^{2}_{x}} ds\Big\|_{L^{2}_{t}}\nonumber\\
&\leq&\big\|AM\langle t\rangle^{-1}\big\|_{L^{2}_{t}}\Big\|A^{1/2}\int^{s}_{s-B}\mathcal {V}_{1\epsilon}^{1/2}\mathscr{U}_{1}(s,\tau)\mathcal{V}_{2}(\tau)Z(\tau)d\tau\Big\|_{L^{2}_{s}}\nonumber\\
&\leq& \frac{1}{2}C(T)\big(\|Z_{0}\|_{L^{2}_{x}}+\big\|F\big\|_{L^{\tilde{p}'}_{t}L^{\tilde{q}'}_{x}}\big).
\end{eqnarray}
Here as mentioned before that one only need to  prove (\ref{eq-419}) for $\int^{T}_{\tilde{M}}$ for some large positive constant $\tilde{M}\ll T$, therefore we can choose $t$ large enough such that $A^{3/2}Mt^{-\varepsilon}$ are small for any $\varepsilon>0$, as well as $t^{-(1-\varepsilon)}\in L^{2}_{t}$. On the other hand,
\begin{eqnarray}\label{eq-443}
J_{2}^{H}&=&\mathcal {V}^{1/2}_{2}(t)\chi_{2}\int^{t}_{t-A}\mathscr{U}_{2}(t,s)P_{c}(s)\mathcal {V}_{1\epsilon}^{1/2}K_{> M}\mathcal {V}_{1\epsilon}^{1/2}\int^{s}_{s-B}\mathscr{U}_{1}(s,\tau)\mathcal{V}_{2}(\tau)Z(\tau)d\tau ds\nonumber\\
&&\ +\ \mathcal {V}^{1/2}_{2}(t)\chi_{2}\int^{t}_{t-A}\mathscr{U}_{2}(t,s)P_{c}(s)\big[K_{> M},\mathcal {V}_{1\epsilon}^{1/2}\big]\mathcal {V}_{1\epsilon}^{1/2}\int^{s}_{s-B}\mathscr{U}_{1}(s,\tau)\mathcal{V}_{2}(\tau)Z(\tau)d\tau ds\nonumber\\
&:=&J_{2}^{H,1}+J_{2}^{H,2}.\nonumber
\end{eqnarray}
It follows from the Fubini theorem and H\"{o}lder's inequality that
\begin{eqnarray}\label{eq-444'}
\big\|J_{2}^{H,2}\big\|_{L^{2}_{t}L^{2}_{x}}&\leq&\Big\|\int^{t}_{t-A}\big\|\chi_{2}\mathcal {V}^{1/2}_{2}(t)\big\|_{L^{\infty}}\big\|\mathscr{U}_{2}(t,s)P_{c}(s)\big\|_{L^{2}_{x}\rightarrow L^{2}_{x}}\big\|\big[K_{> M},\mathcal {V}_{1\epsilon}^{1/2}\big]\big\|_{L^{2}_{x}\rightarrow L^{2}_{x}}\nonumber\\
&&\ \ \ \ \ \times\ \big\|\int^{s}_{s-B}\mathcal {V}_{1\epsilon}^{1/2}\mathscr{U}_{1}(s,\tau)\mathcal{V}_{2}(\tau)Z(\tau)d\tau\big\|_{L^{2}_{x}} ds\Big\|_{L^{2}_{t}}\nonumber\\
&\leq&A^{1/2}\epsilon M^{-1}\Big\|\Big(\int^{t}_{t-A}\big\|\int^{s}_{s-B}\mathcal {V}_{1\epsilon}^{1/2}\mathscr{U}_{1}(s,\tau)\mathcal{V}_{2}(\tau)Z(\tau)d\tau\big\|^{2}_{L^{2}_{x}} ds\Big)^{1/2}\Big\|_{L^{2}_{t}}\nonumber\\
&\leq&A\epsilon M^{-1}\Big\|\int^{s}_{s-B}\big\|\mathcal {V}_{1\epsilon}^{1/2}\mathscr{U}_{1}(s,\tau)\mathcal{V}_{2}(\tau)Z(\tau)\big\|_{L^{2}_{x}}d\tau\Big\|_{L^{2}_{s}}\nonumber\\
&\leq&AB\epsilon^{2}M^{-1}\big\|\mathcal{V}_{2}^{1/2}(\tau)Z(\tau)\big\|_{L^{2}_{\tau}L^{2}_{x}}\nonumber\\
&\leq& AB\epsilon^{2}M^{-1}C(T)\big(\|Z_{0}\|_{L^{2}_{x}}+\big\|F\big\|_{L^{\tilde{p}'}_{t}L^{\tilde{q}'}_{x}}\big),
\end{eqnarray}

To deal with $J_{2}^{H,1}$, we claim that for all $T>0$ and $\alpha>0$,
\begin{eqnarray}\label{eq-442}
\int_{\tau}^{T}\int_{\mathbb{R}^{3}}\big|\mathcal {V}_{1\epsilon}^{1/2}\nabla\langle\nabla\rangle^{-1/2}\mathscr{U}_{1}(s,\tau)f\big|^{2}dxds\leq \epsilon^{3} (T-\tau)\big(1+\|\mathcal{V}_{1}\|_{L^{\infty}}\big)\|f\|^{2}_{L^{2}_{x}}.
\end{eqnarray}
Similarly to (\ref{eq-444'}),
\begin{eqnarray}\label{eq-446}
&&\big\|J_{2}^{H,1}\big\|_{L^{2}_{t}L^{2}_{x}}\nonumber\\
&\leq& \Big\|\int^{t}_{t-A}h_{2}(t,s)\int^{s}_{s-B}\big\|K_{> M}\mathcal {V}_{1\epsilon}^{1/2}\mathscr{U}_{1}(s,\tau)\mathcal{V}_{2}(\tau)Z(\tau)\big\|_{L^{2}_{x}}d\tau ds\Big\|_{L^{2}_{t}}\nonumber\\
&\leq&\Big\|\int^{s}_{s-B}\big\|K_{> M}\mathcal {V}_{1\epsilon}^{1/2}\mathscr{U}_{1}(s,\tau)\mathcal{V}_{2}(\tau)Z(\tau)\big\|_{L^{2}_{x}}d\tau\Big\|_{L^{2}_{s}}\nonumber\\
&\leq&B^{1/2}\sup_{\tau}\Big(\int_{\tau}^{\tau+B}\big\|K_{> M}\mathcal {V}_{1\epsilon}^{1/2}\mathscr{U}_{1}(s,\tau)\mathcal{V}^{1/2}_{2}(\tau)\big\|^{2}_{L^{2}_{x}\rightarrow L^{2}_{x}}ds\Big)^{1/2}\big\|\mathcal{V}_{2}^{1/2}(\tau)Z(\tau)\big\|_{L^{2}_{\tau}L^{2}_{x}},\nonumber
\end{eqnarray}
Notice that by (\ref{eq-442}),
\begin{eqnarray}\label{eq-447}
&&\int_{\tau}^{\tau+B}\big\|K_{> M}\mathcal {V}_{1\epsilon}^{1/2}\mathscr{U}_{1}(s,\tau)\mathcal{V}^{1/2}_{2}(\tau)\big\|^{2}_{L^{2}_{x}\rightarrow L^{2}_{x}}ds\nonumber\\
&\leq&M^{-1}\int_{\tau}^{\tau+B}\big\|\nabla\langle\nabla\rangle^{-1/2}\mathcal {V}_{1\epsilon}^{1/2}\mathscr{U}_{1}(s,\tau)\mathcal{V}^{1/2}_{2}(\tau)\big\|^{2}_{L^{2}_{x}\rightarrow L^{2}_{x}}ds\nonumber\\
&\leq&M^{-1}\int_{\tau}^{\tau+B}\big\|\big[\nabla\langle\nabla\rangle^{-1/2},\mathcal {V}_{1\epsilon}^{1/2}\big]\mathscr{U}_{1}(s,\tau)\mathcal{V}^{1/2}_{2}(\tau)\big\|^{2}_{L^{2}_{x}\rightarrow L^{2}_{x}}ds\nonumber\\
&&\ \ \ \ \ \ +\ M^{-1}\int_{\tau}^{\tau+B}\big\|\mathcal {V}_{1\epsilon}^{1/2}\nabla\langle\nabla\rangle^{-1/2}\mathscr{U}_{1}(s,\tau)\mathcal{V}^{1/2}_{2}(\tau)\big\|^{2}_{L^{2}_{x}\rightarrow L^{2}_{x}}ds\nonumber\\
&\lesssim&\epsilon^{3} M^{-1}B.\nonumber
\end{eqnarray}
Thus,
\begin{eqnarray}\label{eq-448}
\big\|J_{2}^{H,1}\big\|_{L^{2}_{t}L^{2}_{x}}&\leq& \epsilon^{3/2}BM^{-1/2}\big\|\mathcal{V}_{2}^{1/2}(\tau)Z(\tau)\big\|_{L^{2}_{\tau}L^{2}_{x}}\nonumber\\
&\leq&\epsilon^{3/2}BM^{-1/2}C(T)\big(\|Z_{0}\|_{L^{2}_{x}}+\big\|F\big\|_{L^{\tilde{p}'}_{t}L^{\tilde{q}'}_{x}}\big),\nonumber
\end{eqnarray}
which combining (\ref{eq-444'}) lead to
\begin{eqnarray}\label{eq-449}
\big\|J_{2}^{H}\big\|_{L^{2}_{t}L^{2}_{x}}&\lesssim& \big(AB\epsilon^{2}M^{-1}+\epsilon^{3/2}BM^{-1/2}\big)C(T)\big(\|Z_{0}\|_{L^{2}_{x}}+\big\|F\big\|_{L^{\tilde{p}'}_{t}L^{\tilde{q}'}_{x}}\big)\nonumber\\
&\leq& \frac{1}{2}C(T)\big(\|Z_{0}\|_{L^{2}_{x}}+\big\|F\big\|_{L^{\tilde{p}'}_{t}L^{\tilde{q}'}_{x}}\big),
\end{eqnarray}
here we choose $M\gg A,B$.

It remains to prove the claims (\ref{eq-439}) and (\ref{eq-442}). For (\ref{eq-442}),
we only need to prove it for $\tau=0$. Denote by $\Psi(t)=\mathscr{U}_{1}(t)f$ and then
$\Psi(\epsilon^{-1}x,\epsilon^{-2}t)$ satisfies the homogeneous Schr\"{o}dinger equation
\begin{eqnarray}\label{eq-450}
&&i\partial_{t}\Psi(t)=\big(\mathscr{H}_{0}+\mathcal {V}_{1}\big)\Psi(t),\nonumber\\
&&\Psi(t)=f(\epsilon^{-1}\cdot).
\end{eqnarray}
Therefore, by using (3.29) in the proof of \cite[Lemma 3.1]{RSS1}, we have for any constant $T>0$ and $\alpha>0$,
\begin{eqnarray}\label{eq-451}
\int_{0}^{\epsilon^{2}T}\int_{\mathbb{R}^{3}}(1+|x|^{\alpha})^{-(1/\alpha+1)}\big|\nabla\langle\nabla\rangle^{-1/2}\Psi(\epsilon^{-1}x,\epsilon^{-2}t)\big|^{2}dxdt
\leq  \epsilon^{2}T\big(1+\|\mathcal{V}_{1}\|_{L^{\infty}}\big)\|f(\epsilon^{-1}\cdot)\|^{2}_{L^{2}_{x}}.\nonumber
\end{eqnarray}
which leads to
\begin{eqnarray}\label{eq-452}
\int_{0}^{T}\int_{\mathbb{R}^{3}}(1+|\epsilon x|^{\alpha})^{-(1/\alpha+1)}\big|\nabla\langle\nabla\rangle^{-1/2}\Psi(x,t)\big|^{2}dxdt
\leq \epsilon T\big(1+\|\mathcal{V}_{1}\|_{L^{\infty}}\big)\|f\|^{2}_{L^{2}_{x}}.
\end{eqnarray}
Thus
\begin{eqnarray}\label{eq-453}
&&\int_{0}^{T}\int_{\mathbb{R}^{3}}\big|\mathcal {V}_{1\epsilon}^{1/2}\nabla\langle\nabla\rangle^{-1/2}\Psi(x,t)\big|^{2}dxdt\nonumber\\
&=&\int_{0}^{T}\int_{\mathbb{R}^{3}}\big|\mathcal {V}_{1\epsilon}^{1/2}(1+|\epsilon x|^{\alpha})^{(1/\alpha+1)}\big|(1+|\epsilon x|^{\alpha})^{-(1/\alpha+1)}\big|\nabla\langle\nabla\rangle^{-1/2}\Psi(x,t)\big|^{2}dxdt\nonumber\\
&\leq& \epsilon^{3} T\big(1+\|\mathcal{V}_{1}\|_{L^{\infty}}\big)\|f\|^{2}_{L^{2}_{x}}.
\end{eqnarray}
Let us turn to (\ref{eq-439}), the proof follows from a commutator argument. Assume that $\epsilon t\leq1$,
\begin{eqnarray}\label{eq-456}
&&\sup_{|t-s|\leq A}\big\|\mathcal {V}^{1/2}_{2}(t)\chi_{2}\mathscr{U}_{2}(t,s)P_{c}(s)K_{\leq M}\mathcal {V}_{1\epsilon}^{1/2}\big\|_{L^{2}_{x}\rightarrow L^{2}_{x}}\nonumber\\
&\lesssim& \big\|K_{\leq M}\big\|_{L^{2}_{x}\rightarrow L_{x}^{2}}\big\|\mathcal {V}_{1\epsilon}^{1/2}\big\|_{L^{\infty}}\lesssim \epsilon\lesssim t^{-1},
\end{eqnarray}
where $K_{\leq M}f=\big(\hat{\eta}(\xi/M)\hat{f}(\xi)\big)^{\vee}$ with some smooth bump function $\eta$.
If $\epsilon t>1$, we use $P^{2}_{c}(s)=P_{c}(s)$ and the intertwining identity (\ref{eq-43}), consider
\begin{eqnarray}\label{eq-454}
\big\|\big[\chi_{2}(t),P_{c}(t)\big]f\big\|_{L^{2}_{x}}&=&\big\|\big[\chi_{2}(t),P_{b}(t)\big]f\big\|_{L^{2}_{x}}\nonumber\\
&=&\Big\|\sum_{j=1}^{8}\widetilde{\xi}_{j}(x,t;\widetilde{\sigma}(t))\big\langle f, (\chi_{2}(t)-1)\widetilde{\xi}_{j}(x,t;\widetilde{\sigma}(t))\big\rangle\nonumber\\
&&\ \ \ -\ (\chi_{2}(t)-1)\widetilde{\xi}_{j}(x,t;\widetilde{\sigma}(t))\big\langle f, \widetilde{\xi}_{j}(x,t;\widetilde{\sigma}(t))\big\rangle\Big\|_{L^{2}_{x}}\nonumber\\
&\lesssim& e^{-c\epsilon t}\|f\|_{L^{2}_{x}}
\end{eqnarray}
with some constant $c>0$
and then
\begin{eqnarray}\label{eq-455}
\big\|\mathcal {V}^{1/2}_{2}(t)\big[\chi_{2}(t),P_{c}(t)\big]\mathscr{U}_{2}(t,s)P_{c}(s)K_{\leq M}\mathcal {V}_{1\epsilon}^{1/2}f\big\|_{L^{2}_{x}}&\lesssim& e^{-c\epsilon t}\big\|K_{\leq M}\big\|_{L^{2}_{x}\rightarrow L_{x}^{2}}\big\|\mathcal {V}_{1\epsilon}^{1/2}\big\|_{L^{\infty}}\|f\|_{L^{2}_{x}}\nonumber\\
&\lesssim& \epsilon e^{-c\epsilon t} \|f\|_{L^{2}_{x}}\lesssim t^{-1}\|f\|_{L^{2}_{x}}.
\end{eqnarray}
 Notice that
\begin{eqnarray}\label{eq-456}
\big[\chi_{2}(t),\mathscr{U}_{2}(t,s)\big]&=&\mathscr{U}_{2}(t,s)\big(\mathscr{U}_{2}(s,t)\chi_{2}(t)\mathscr{U}_{2}(t,s)-\chi_{2}(t)\big)\nonumber\\
&=&i\mathscr{U}_{2}(t,s)\int_{s}^{t}\mathscr{U}_{2}(s,\tau)\big[\mathscr{H}_{2}(\tau),\chi_{2}(t)\big]\mathscr{U}_{2}(\tau,s)d\tau\nonumber\\
&=&i\mathscr{U}_{2}(t,s)\int_{s}^{t}\mathscr{U}_{2}(s,\tau)\big(\big[\mathscr{H}_{0},\chi_{2}(t)\big]+\big[E(\tau),\chi_{2}(t)\big]\big)\mathscr{U}_{2}(\tau,s)d\tau,
\end{eqnarray}
where
$\mathscr{H}_{2}(t)=\mathscr{H}_{0}+\mathcal {V}_{2}(t,\widetilde{\sigma}(t))+E(t)$ with
$E(t)=iT_{0}(t)\big[\overline{P}_{c}'(t),\overline{P}_{c}(t)\big]T^{\ast}_{0}(t)$ is defined as in (\ref{eq-41}).  Then it follows
\begin{eqnarray}\label{eq-459}
&&\big\|\mathcal {V}^{1/2}_{2}(t)P_{c}(t)\big[\chi_{2}(t),\mathscr{U}_{2}(t,s)\big]P_{c}(s)K_{\leq M}\mathcal {V}_{1\epsilon}^{1/2}\big\|_{L^{2}_{x}\rightarrow L^{2}_{x}}\nonumber\\
&\lesssim& \epsilon A\sup_{\tau,s}\big\|[\mathscr{H}_{0},\chi_{2}(t)\big]\mathscr{U}_{2}(\tau,s)P_{c}(s)K_{\leq M}\big\|_{L^{2}_{x}\rightarrow L^{2}_{x}}\nonumber\\
&&\ \ \ \ \ \ \ \ \ \ \ \ \ \ \ \ +\ \epsilon A\sup_{\tau,s}\big\|\big[E(\tau),\chi_{2}(t)\big]\mathscr{U}_{2}(\tau,s)P_{c}(s)K_{\leq M}\big\|_{L^{2}_{x}\rightarrow L^{2}_{x}}.
\end{eqnarray}
Observe now that
\begin{eqnarray}\label{eq-457}
\big|\nabla \chi_{2}(t)\big|\lesssim \frac{1}{\delta\epsilon t},\ \ \big|\Delta \chi_{2}(t)\big|\lesssim \frac{1}{(\delta\epsilon t)^{2}}\nonumber
\end{eqnarray}
and then
\begin{eqnarray}\label{eq-460}
\sup_{\tau,s}\big\|[\mathscr{H}_{0},\chi_{2}(t)]\mathscr{U}_{2}(\tau,s)P_{c}(s)K_{\leq M}\big\|_{L^{2}_{x}\rightarrow L^{2}_{x}}
&\lesssim&
\sup_{\tau,s}\big\|\nabla \chi_{2}(t)\nabla\mathscr{U}_{2}(\tau,s)P_{c}(s)K_{\leq M}\big\|_{L^{2}_{x}\rightarrow L^{2}_{x}}\nonumber\\
&&\ \ +\ \sup_{\tau,s}\big\|\Delta \chi_{2}(t)\mathscr{U}_{2}(\tau,s)P_{c}(s)K_{\leq M}\big\|_{L^{2}_{x}\rightarrow L^{2}_{x}}\nonumber\\
&\lesssim& M\big(\frac{1}{\delta\epsilon t}+\frac{1}{(\delta\epsilon t)^{2}}\big),
\end{eqnarray}
where we use the fact that
\begin{eqnarray}
\sup_{\tau,s}\big\|\nabla\mathscr{U}_{2}(\tau,s)P_{c}(s)K_{\leq M}\big\|_{L^{2}_{x}\rightarrow L^{2}_{x}}\lesssim M,\nonumber
\end{eqnarray}
and it is obtained by interpolation between
\begin{eqnarray}
\sup_{\tau,s}\big\|\mathscr{U}_{2}(\tau,s)P_{c}(s)K_{\leq M}\big\|_{L^{2}_{x}\rightarrow L^{2}_{x}}\lesssim 1\nonumber
\end{eqnarray}
and
\begin{eqnarray}\label{eq-459'}
&&\sup_{\tau,s}\big\|\Delta\mathscr{U}_{2}(\tau,s)P_{c}(s)K_{\leq M}f\big\|_{L^{2}_{x}}\nonumber\\
&\lesssim& \sup_{\tau,s}\big\|\mathscr{H}_{2}(\tau)\mathscr{U}_{2}(\tau,s)P_{c}(s)K_{\leq M}f\big\|_{L^{2}_{x}}
+\sup_{\tau,s}\big\|\mathcal{V}_{2}(\tau)\mathscr{U}_{2}(\tau,s)P_{c}(s)K_{\leq M}f\big\|_{L^{2}_{x}}\nonumber\\
&&\ \ \ \ \ \ \ \ +\ \sup_{\tau,s}\big\|E(\tau)\mathscr{U}_{2}(\tau,s)P_{c}(s)K_{\leq M}f\big\|_{L^{2}_{x}}\nonumber\\
&\lesssim& \big\|\Delta K_{\leq M}f\big\|_{L^{2}_{x}}+\big\| K_{\leq M}f\big\|_{L^{2}_{x}}\lesssim M^{2}\|f\|_{L^{2}_{x}}.
\end{eqnarray}
Here for the first term in the first inequality of (\ref{eq-459'}), we use the following identity which is obtained  by differentiation both side of (\ref{eq-43}) with respect to $s$ at $s=t$,
\begin{eqnarray}
i\dot{P_{c}}(t)=\mathscr{H}_{2}(t)P_{c}(t)-P_{c}(t)\mathscr{H}_{2}(t).\nonumber
\end{eqnarray}
On the other hand, notice that the formula (\ref{eq-424}) for $E(\tau)$ and
\begin{eqnarray}\label{eq-458}
\big[E(\tau),\chi_{2}(t)\big]f&=&\big(\beta'(t)-y'(t)\big)\sum_{j,k}\big(\tilde{\Phi}_{j}\big\langle\tilde{\Psi}_{k},\chi_{2}(t)f\big\rangle
-\chi_{2}(t)\tilde{\Phi}_{j}\big\langle\tilde{\Psi}_{k},f\big\rangle\big)\nonumber\\
&=&\big(\beta'(t)-y'(t)\big)\sum_{j,k}\big(\tilde{\Phi}_{j}\big\langle(\chi_{2}(t)-1)\tilde{\Psi}_{k},f\big\rangle
-(\chi_{2}(t)-1)\tilde{\Phi}_{j}\big\langle\tilde{\Psi}_{k},f\big\rangle\big)\nonumber
\end{eqnarray}
with $\tilde{\Phi}_{j}$ and $\tilde{\Psi}_{k}$ centered at $b(\tau)$, similarly to (\ref{eq-454}),
\begin{eqnarray}\label{eq-461}
\big\|\big[E(\tau),\chi_{2}(t)\big]f\big\|_{L^{2}_{x}}\lesssim e^{-c\epsilon t}\|f\|_{L^{2}_{x}},\nonumber
\end{eqnarray}
where we use the fact that $\tau\in [s,t]$ and $s\in[t-A,t]$. Then we obtain
\begin{eqnarray}\label{eq-462}
\sup_{\tau,s}\big\|\big[E(\tau),\chi_{2}(t)\big]\mathscr{U}_{2}(\tau,s)P_{c}(s)K_{\leq M}\big\|_{L^{2}_{x}\rightarrow L^{2}_{x}}\lesssim e^{-c\epsilon t}\|f\|_{L^{2}_{x}}.
\end{eqnarray}
Therefore, it follows from  (\ref{eq-459}), (\ref{eq-460}) and (\ref{eq-463}) that
\begin{eqnarray}
\label{eq-463}
&&\big\|\mathcal {V}^{1/2}_{2}(t)P_{c}(t)\big[\chi_{2}(t),\mathscr{U}_{2}(t,s)\big]P_{c}(s)K_{\leq M}\mathcal {V}_{1\epsilon}^{1/2}\big\|_{L^{2}_{x}\rightarrow L^{2}_{x}}\nonumber\\
&\lesssim& \epsilon A\big(\frac{M}{\delta\epsilon t}+\frac{M}{(\delta\epsilon t)^{2}}+e^{-c\epsilon t}\big)\lesssim \frac{AM}{\delta t}.
\end{eqnarray}
Now we proceed to $\mathcal {V}^{1/2}_{2}(t)P_{c}(s)\mathscr{U}_{2}(t,s)\big[\chi_{2}(t),P_{c}(s)\big]K_{\leq M}\mathcal {V}_{1\epsilon}^{1/2}$, since $s\in [t-A,t]$, similarly to (\ref{eq-455}),
\begin{eqnarray}\label{eq-464}
\big\|\mathcal {V}^{1/2}_{2}(t)P_{c}(t)\mathscr{U}_{2}(t,s)\big[\chi_{2}(t),P_{c}(s)\big]K_{\leq M}\mathcal {V}_{1\epsilon}^{1/2}f\big\|_{L^{2}_{x}}
\lesssim \epsilon e^{-c\epsilon t} \|f\|_{L^{2}_{x}}\lesssim \frac{\|f\|_{L^{2}_{x}}}{t}.
\end{eqnarray}
Finally,
\begin{eqnarray}\label{eq-465}
\big\|[\chi_{2}(t),K_{\leq M}]\big\|_{L^{2}_{x}\rightarrow L^{2}_{x}}\lesssim  M^{-1}\|\nabla \chi_{2}\|_{L^{\infty}_{x}}\lesssim M^{-1}\frac{1}{\epsilon\delta t},\nonumber
\end{eqnarray}
which leads to
\begin{eqnarray}\label{eq-466}
&&\big\|\mathcal {V}^{1/2}_{2}(t)P_{c}(t)\mathscr{U}_{2}(t,s)P_{c}(s)[\chi_{2}(t),K_{\leq M}]\mathcal {V}_{1\epsilon}^{1/2}f\big\|_{L^{2}_{x}}\nonumber\\
&\lesssim& \epsilon M^{-1}\frac{1}{\epsilon\delta t} \|f\|_{L^{2}_{x}}\lesssim (M\delta t)^{-1}\|f\|_{L^{2}_{x}}.
\end{eqnarray}
One concludes from  (\ref{eq-455}), (\ref{eq-463}), (\ref{eq-464}) and (\ref{eq-466}) that for $\epsilon t>1$,
\begin{eqnarray}\label{eq-467}
\big\|\mathcal {V}^{1/2}_{2}(t)\chi_{2}\mathscr{U}_{2}(t,s)P_{c}(s)K_{\leq M}\mathcal {V}_{1\epsilon}^{1/2}-\mathcal {V}^{1/2}_{2}(t)\mathscr{U}_{2}(t,s)P_{c}(s)K_{\leq M}\chi_{2}\mathcal {V}_{1\epsilon}^{1/2}\big\|_{L^{2}_{x}\rightarrow L^{2}_{x}}\lesssim AM(\delta t)^{-1},\nonumber
\end{eqnarray}
which combining the observation $\chi_{2}\mathcal {V}_{1\epsilon}=0$
finish the proof of claim (\ref{eq-439}).

\emph{$II_{2}$. Local decay for $\mathcal {V}^{1/2}_{1\epsilon}\chi_{1}(t,\cdot)Z(t)$.} This will be done by using similar argument as the one in $II_{1}$, we will not write in detail and only sketch the proof. For large $A\ll T$ to be fixed later, it follows from Duhamel's formula that
\begin{eqnarray}\label{eq-469}
\mathcal {V}^{1/2}_{1\epsilon}\chi_{1}Z(t)&=&\mathcal {V}^{1/2}_{1\epsilon}\chi_{1}\mathscr{U}_{1}(t)Z_{0}-i\mathcal {V}^{1/2}_{1\epsilon}\chi_{1}\int^{t-A}_{0}\mathscr{U}_{1}(t,s)\mathcal {V}_{2}(s)Z(s)ds\nonumber\\
&&-\ i\mathcal {V}^{1/2}_{1\epsilon}\chi_{1}\int^{t}_{t-A}\mathscr{U}_{1}(t,s)\mathcal {V}_{2}(s)Z(s)ds+i\mathcal {V}^{1/2}_{1\epsilon}\chi_{1}\int^{t}_{t-A}\mathscr{U}_{1}(t,s)F(s)ds.
\end{eqnarray}
By using the endpoint Strichartz estimates for $\mathscr{U}_{1}(t)$ (see Lemma \ref{le-42}), we have
\begin{eqnarray}\label{eq-470}
\big\|\mathcal {V}^{1/2}_{1\epsilon}\chi_{1}\mathscr{U}_{1}(t)Z_{0}\big\|_{L^{2}_{t}L^{2}_{x}}
&\leq&\big\|\mathcal {V}^{1/2}_{1\epsilon}\chi_{1}\big\|_{L^{\infty}_{t}L^{3}_{x}}\big\|\mathscr{U}_{1}(t)Z_{0}\big\|_{L^{2}_{t}L^{6}_{x}}\nonumber\\
&\lesssim& \|Z_{0}\|_{L^{2}}.
\end{eqnarray}
and
\begin{eqnarray}\label{eq-471}
\Big\|\mathcal {V}^{1/2}_{1\epsilon}\chi_{1}\int^{t}_{t-A}\mathscr{U}_{1}(t,s)F(s)ds\Big\|_{L^{2}_{t}L^{2}_{x}}
&\leq&\big\|\mathcal {V}^{1/2}_{1\epsilon}\chi_{1}\big\|_{L^{\infty}_{t}L^{3}_{x}}\Big\|\int^{t}_{t-A}\mathscr{U}_{1}(t,s)F(s)ds\Big\|_{L^{2}_{t}L^{6}_{x}}\nonumber\\
&\lesssim& \big\|F\big\|_{L^{\tilde{p}'}_{t}L^{\tilde{q}'}_{x}}.
\end{eqnarray}
The proof for the second term in (\ref{eq-469}) is essentially similar to the one in (\ref{eq-411}). More precisely, one only need to estimate it with  $$h_{2}(t,s)=\big\|\mathscr{U}_{2}(t,s)P_{c}(s)\big\|_{L^{6/5,1}_{x}\rightarrow L^{6,\infty}_{x}}$$
replaced by
$$h_{1}(t,s)=\big\|\mathscr{U}_{1}(t,s)\big\|_{L^{6/5,1}_{x}\rightarrow L^{6,\infty}_{x}}$$
in (\ref{eq-428})-(\ref{eq-430'}). We omit these repeated procedures. As for the third term in (\ref{eq-469}),
\begin{eqnarray}\label{eq-472}
\mathcal {V}^{1/2}_{1\epsilon}\chi_{1}\int^{t}_{t-A}\mathscr{U}_{1}(t,s)\mathcal {V}_{2}(s)Z(s)ds&=&\mathcal {V}^{1/2}_{1\epsilon}\chi_{1}\int^{t}_{t-A}\mathscr{U}_{1}(t,s)\mathcal {V}_{2}(s)P_{b}(s)Z(s)ds\nonumber\\
&&\ \ \ \ +\ \mathcal {V}^{1/2}_{1\epsilon}\chi_{1}\int^{t}_{t-A}\mathscr{U}_{1}(t,s)\mathcal {V}_{2}(s)P_{c}(s)Z(s)ds\nonumber\\
&:=&J_{b}+J_{c}.
\end{eqnarray}
It follows from Shur's Lemma and (\ref{eq-302}) that
\begin{eqnarray}\label{eq-473}
\|J_{b}\|_{L^{2}_{t}L^{2}_{x}}&\leq& \big\|\mathcal {V}^{1/2}_{1\epsilon}\chi_{1}\big\|_{L^{\infty}_{t}L^{3,2}_{x}}\big\|\mathcal {V}^{1/2}_{2}\big\|_{L^{\infty}_{t}L^{3,2}_{x}}
\big\|\mathcal {V}^{1/2}_{2}\big\|_{L^{\infty}_{t}L^{3}_{x}}\Big\|\int^{t}_{t-A}h_{1}(t,s)\|P_{b}(s)Z(s)\|_{L^{6}_{x}}ds\Big\|_{L^{2}_{t}}\nonumber\\
&\lesssim&\|P_{b}(s)Z(s)\|_{L^{2}_{s}L^{6}_{x}}\lesssim B.
\end{eqnarray}
On the other hand,
\begin{eqnarray}\label{eq-474}
J_{c}&=&\mathcal {V}^{1/2}_{1\epsilon}\chi_{1}\int^{t}_{t-A}\mathscr{U}_{1}(t,s)\mathcal {V}_{2}(s)P_{c}(s)\mathscr{U}_{2}(s)Z_{0}ds\nonumber\\
&&\ \ \ -\ i\mathcal {V}^{1/2}_{1\epsilon}\chi_{1}\int^{t}_{t-A}\mathscr{U}_{1}(t,s)\mathcal {V}_{2}(s)\int^{s-B}_{0}\mathscr{U}_{2}(s,\tau)P_{c}(\tau)\mathcal {V}_{1\epsilon}Z(\tau)d\tau ds\nonumber\\
&&\ \ \ \ -\ i\mathcal {V}^{1/2}_{1\epsilon}\chi_{1}\int^{t}_{t-A}\mathscr{U}_{1}(t,s)\mathcal {V}_{2}(s)\int^{s}_{s-B}\mathscr{U}_{2}(s,\tau)P_{c}(\tau)\mathcal {V}_{1\epsilon}Z(\tau)d\tau ds\nonumber\\
&&\ \ \ \ \ \ +\ i\mathcal {V}^{1/2}_{1\epsilon}\chi_{1}\int^{t}_{t-A}\mathscr{U}_{1}(t,s)\mathcal {V}_{2}(s)\int^{s}_{0}\mathscr{U}_{2}(s,\tau)P_{c}(\tau)E(\tau)Z(\tau)d\tau ds\nonumber\\
&&\ \ \ \ \ \ \ \ \ -i\mathcal {V}^{1/2}_{1\epsilon}\chi_{1}\int^{t}_{t-A}\mathscr{U}_{1}(t,s)\mathcal {V}_{2}(s)\int^{s}_{0}\mathscr{U}_{2}(s,\tau)P_{c}(\tau)F(\tau)d\tau ds\nonumber\\
&:=&J^{1}_{c}+J^{2}_{c}+J^{3}_{c}+J^{4}_{c}+J^{5}_{c}.\nonumber
\end{eqnarray}
Then by using Shur's Lemma and Lemma \ref{le-41}, we obtain
\begin{eqnarray}\label{eq-475}
&& \|J^{1}_{c}\|_{L^{2}_{t}L^{2}_{x}}\nonumber\\
 &\leq& \big\|\mathcal {V}^{1/2}_{1\epsilon}\chi_{1}\big\|_{L^{\infty}_{t}L^{3,2}_{x}}\big\|\mathcal {V}^{1/2}_{2}\big\|_{L^{\infty}_{t}L^{3,2}_{x}}
\big\|\mathcal {V}^{1/2}_{2}\big\|_{L^{\infty}_{t}L^{3}_{x}}\Big\|\int^{t}_{t-A}h_{1}(t,s)\|\mathscr{U}_{2}(s)P_{c}(s)Z_{0}\|_{L^{6}_{x}}ds\Big\|_{L^{2}_{t}}\nonumber\\
&\lesssim& \|\mathscr{U}_{2}(s)P_{c}(s)Z_{0}\|_{L^{2}_{s}L^{6}_{x}}\lesssim \|Z_{0}\|_{L^{2}_{x}}.
\end{eqnarray}
Similarly,
\begin{eqnarray}\label{eq-476}
 \|J^{2}_{c}\|_{L^{2}_{t}L^{2}_{x}}+ \|J^{4}_{c}\|_{L^{2}_{t}L^{2}_{x}}+ \|J^{5}_{c}\|_{L^{2}_{t}L^{2}_{x}}&\lesssim& \Big\|\mathcal {V}_{2}^{1/2}(s)\int^{s-B}_{0}\mathscr{U}_{2}(s,\tau)P_{c}(\tau)\mathcal {V}_{1\epsilon}Z(\tau)d\tau\Big\|_{L^{2}_{s}L^{2}_{x}}\nonumber\\
 &&\  \ \ \ \ +\ \Big\|\mathcal {V}_{2}^{1/2}(s)\int^{s}_{0}\mathscr{U}_{2}(s,\tau)P_{c}(\tau)E(\tau)Z(\tau)d\tau\Big\|_{L^{2}_{s}L^{2}_{x}}\nonumber\\
 &&\  \ \ \ \ \ \ +\ \Big\|\mathcal {V}_{2}^{1/2}(s)\int^{s}_{0}\mathscr{U}_{2}(s,\tau)P_{c}(\tau)F(\tau)d\tau\Big\|_{L^{2}_{s}L^{2}_{x}}\nonumber\\
 &\leq&\frac{1}{2}C(T)\big(\|Z_{0}\|_{L^{2}_{x}}+\big\|F\big\|_{L^{\tilde{p}'}_{t}L^{\tilde{q}'}_{x}}\big),
\end{eqnarray}
where we use (\ref{eq-427})-(\ref{eq-430'}), (\ref{eq-425}) and (\ref{eq-417'}) in the last inequality.
\begin{eqnarray}\label{eq-477}
iJ^{3}_{c}&=&\mathcal {V}^{1/2}_{1\epsilon}\chi_{1}\int^{t}_{t-A}\mathscr{U}_{1}(t,s)K_{\leq M}\mathcal {V}_{2}(s)\int^{s}_{s-B}\mathscr{U}_{2}(s,\tau)P_{c}(\tau)\mathcal {V}_{1\epsilon}Z(\tau)d\tau ds\nonumber\\
&&\ \ \ +\ \mathcal {V}^{1/2}_{1\epsilon}\chi_{1}\int^{t}_{t-A}\mathscr{U}_{1}(t,s)K_{\geq M}\mathcal {V}_{2}(s)\int^{s}_{s-B}\mathscr{U}_{2}(s,\tau)P_{c}(\tau)\mathcal {V}_{1\epsilon}Z(\tau)d\tau ds\nonumber\\
&:=&J^{3,L}_{c}+J^{3,H}_{c}.
\end{eqnarray}
Now we follow the same arguments as the ones treating $J^{L}_{2}$ and $J_{2}^{H}$ in (\ref{eq-438}) and reduce the proof to the following two claims:
\begin{eqnarray}\label{eq-478}
\sup_{|t-s|\leq A}\big\|\mathcal {V}^{1/2}_{1\epsilon}\chi_{1}\mathscr{U}_{1}(t,s)K_{\leq M}\mathcal {V}^{1/2}_{2}(s)\big\|_{L^{2}_{x}\rightarrow L^{2}_{x}}\leq \frac{AM}{\delta t}
\end{eqnarray}
and
for all $T>0$ and $\alpha>0$,
\begin{eqnarray}\label{eq-479}
&&\int_{\tau}^{T}\int_{\mathbb{R}^{3}}\big|\mathcal {V}^{1/2}_{2}(s)\nabla\langle\nabla\rangle^{-1/2}\mathscr{U}_{2}(s,\tau)P_{c}(\tau)f\big|^{2}dxds\nonumber\\
&\leq&  (T-\tau)\big(1+\|\mathcal{V}_{2}\|_{L^{\infty}}\big)\|f\|^{2}_{L^{2}_{x}}.
\end{eqnarray}
The proof for (\ref{eq-478}) shares exactly the same method as the one used in the proof of (\ref{eq-439}). As for (\ref{eq-479}), we here will sketch the proof by using the idea in \cite[Lemma 3.4]{RSS1}. By using the trick in (\ref{eq-453}), it is enough to prove that for all $T>0$ and $\alpha>0$,
\begin{eqnarray}\label{eq-480}
&&\int_{\tau}^{T}\int_{\mathbb{R}^{3}}(1+|x-b(s)|^{\alpha})^{-(1/\alpha+1)}\big|\nabla\langle\nabla\rangle^{-1/2}\mathscr{U}_{2}(s,\tau)P_{c}(\tau)f\big|^{2}dxds\nonumber\\
&\leq& (T-\tau)\big(1+\|\mathcal{V}_{2}\|_{L^{\infty}}\big)\|f\|^{2}_{L^{2}_{x}}.
\end{eqnarray}
To this end, denote  $m:=w(x-b(s))(x-b(s))\frac{\nabla}{\langle\nabla\rangle}$ and $\psi(s,\tau)=\mathscr{U}_{2}(s,\tau)P_{c}(\tau)f$,
where
$$w(x)=(1+|x|^{\alpha})^{-1/\alpha},\ \ \ \alpha>0.$$
One has
\begin{eqnarray}\label{eq-481}
&&\frac{d}{ds}\langle m\psi(s,\tau), \psi(s,\tau)\rangle\nonumber\\
&=&\dot{b}(s)\langle(\nabla w\cdot x+w)(\cdot-b(s))\frac{\nabla}{\langle\nabla\rangle}\psi(s,\tau),\psi(s,\tau)\rangle-i\langle[m,H(s)]\psi(s,\tau),\psi(s,\tau)\rangle\nonumber\\
&=&\dot{b}(s)\langle(\nabla w\cdot x+w)(\cdot-b(s))\frac{\nabla}{\langle\nabla\rangle}\psi(s,\tau),\psi(s,\tau)\rangle\nonumber\\
&&\ +\ i\int_{\mathbb{R}^{3}}w(\cdot-b(s))|\nabla\langle\nabla\rangle^{-1/2}\psi(s,\tau)|^{2}dx\nonumber\\
&&\ \  +\ i\int_{\mathbb{R}^{3}}\nabla w\cdot x(\cdot-b(s)))|\nabla\langle\nabla\rangle^{-1/2}\psi(s,\tau)|^{2}dx\nonumber\\
&&\ \ \  -\ i\big\langle [\langle\nabla\rangle^{1/2},w(\cdot-b(s))]\frac{\nabla}{\langle\nabla\rangle}\psi(s,\tau), \nabla\langle\nabla\rangle^{-1/2}\psi(s,\tau)\big\rangle\nonumber\\
&&\ \ \ \  -\ i\big\langle [\langle\nabla\rangle^{1/2},\nabla w\cdot x(\cdot-b(s)))]\frac{\nabla}{\langle\nabla\rangle}\psi(s,\tau), \nabla\langle\nabla\rangle^{-1/2}\psi(s,\tau)\big\rangle\nonumber\\
&&\ \ \ \ +\ \frac{i}{2}\big\langle \Delta(w\cdot x(\cdot-b(s)))\frac{\nabla}{\langle\nabla\rangle}\psi(s,\tau),\psi(s,\tau)\big\rangle
 - i\langle [m,\mathcal {V}_{2}(s)]\psi(s,\tau),\psi(s,\tau)\rangle\nonumber\\
&:=&I_{1}+I_{2}+I_{3}+I_{4}+I_{5}+I_{6}+I_{7}.
\end{eqnarray}
By using the estimates in the proof of \cite[Lemma 3.4]{RSS1}, we have
\begin{eqnarray}\label{eq-482}
I_{2}+I_{3}\geq \int_{\mathbb{R}^{3}}(1+|x-b(s)|^{\alpha})^{-(1/\alpha+1)}\big|\nabla\langle\nabla\rangle^{-1/2}\psi(s,\tau)\big|^{2}dx
\end{eqnarray}
and
\begin{eqnarray}\label{eq-483}
&&I_{4}+I_{5}\nonumber\\
&\leq& C\|\psi(s,\tau)\|_{L^{2}_{x}}+\frac{1}{4}\int_{\mathbb{R}^{3}}(1+|x-b(s)|^{\alpha})^{-(1/\alpha+1)}\big|\nabla\langle\nabla\rangle^{-1/2}\psi(s,\tau)\big|^{2}dx.
\end{eqnarray}
Moreover, since the multipliers $m$, $\dot{b}(s)(\nabla w\cdot x+w)(\cdot-b(s))\frac{\nabla}{\langle\nabla\rangle}$ and $\Delta(w\cdot x(\cdot-b(s)))\frac{\nabla}{\langle\nabla\rangle}$ are bounded in $L^{2}$ uniformly in $s$, it follows
\begin{eqnarray}\label{eq-484}
I_{1}+I_{6}+I_{7}\leq C(1+\|\mathcal {V}_{2}(s)\|_{L^{\infty}_{x}})\|\psi(s,\tau)\|_{L^{2}_{x}}.
\end{eqnarray}
Now integrating both sides of (\ref{eq-481}) in time and using estimates (\ref{eq-482})-(\ref{eq-484}) and the fact
$$|\langle m\psi(\tau,\tau), \psi(\tau,\tau)+\langle m\psi(T,\tau), \psi(T,\tau)\rangle\rangle|+\|\psi(s,\tau)\|_{L^{2}_{x}}\leq C\|f\|_{L^{2}},$$
we obtain (\ref{eq-480}).

\emph{$II_{3}$. Local decay for $\mathcal {V}^{1/2}_{2}(t,\widetilde{\sigma}(t))\chi_{1}(t,\cdot)Z(t)$ and $\mathcal {V}^{1/2}_{1\epsilon}\chi_{2}(t,\cdot)P_{c}(t)Z(t)$.}
The local decay for $\mathcal {V}^{1/2}_{1\epsilon}\chi_{2}(t,\cdot)P_{c}(t)Z(t)$ is easy since $\mathcal {V}^{1/2}_{1\epsilon}(\cdot)\chi_{2}(t,\cdot)=0$ and
\begin{eqnarray}\label{eq-468}
\big\|\mathcal {V}^{1/2}_{2}(t,\widetilde{\sigma}(t))\chi_{1}(t,\cdot)Z(t)\big\|_{L^{2}_{t}L^{2}_{x}}&\leq&\big\|\mathcal {V}^{1/4}_{2}(t,\widetilde{\sigma}(t))\chi_{1}(t,\cdot)\big\|_{L^{\infty}_{t}L^{\infty}_{x}}\big\|\mathcal {V}^{1/4}_{2}(t,\widetilde{\sigma}(t))Z(t)\big\|_{L^{2}_{t}L^{2}_{x}}\nonumber\\
&\leq&\frac{1}{2}C(T)\big(\|Z_{0}\|_{L^{2}_{x}}+\big\|F\big\|_{L^{\tilde{p}'}_{t}L^{\tilde{q}'}_{x}}\big),
\end{eqnarray}
where we use (\ref{eq-419'}) and the fact that $\big\|\mathcal {V}^{1/4}_{2}(t,\widetilde{\sigma}(t))\chi_{1}(t,\cdot)\big\|_{L^{\infty}_{t}L^{\infty}_{x}}$ is arbitrary small.

\emph{$II_{4}$. Local decay for $\mathcal {V}^{1/2}_{2}(t,\widetilde{\sigma}(t))\chi_{3}(t,\cdot)Z(t)$ and $\mathcal {V}^{1/2}_{1\epsilon}\chi_{3}(t,\cdot)Z(t)$.} It is enough to consider
\begin{eqnarray}\label{eq-485}
\mathcal {V}^{1/2}_{2}(t)\chi_{3}Z(t)&=&\mathcal {V}^{1/2}_{2}(t)\chi_{3}\mathscr{U}_{0}(t)Z_{0}-i\mathcal {V}^{1/2}_{2}(t)\chi_{3}\int^{t-A}_{0}\mathscr{U}_{0}(t,s)(\mathcal {V}_{2}(s)+\mathcal {V}_{1\epsilon})Z(s)ds\nonumber\\
&&\ \ -\ i\mathcal {V}^{1/2}_{2}(t)\chi_{3}\int^{t}_{t-A}\mathscr{U}_{0}(t,s)(\mathcal {V}_{2}(s)+\mathcal {V}_{1\epsilon})Z(s)ds\nonumber\\
&&\ \ \ \ +\ i\mathcal {V}^{1/2}_{2}(t)\chi_{3}\int^{t}_{0}\mathscr{U}_{0}(t,s)F(s)ds.\nonumber
\end{eqnarray}
We mainly focus on the third term of the above identity since the rest can be dealt just as the corresponding ones in $II_{1}$ and $II_{2}$. Notice that
\begin{eqnarray}\label{eq-486}
&&\mathcal {V}^{1/2}_{2}(t)\chi_{3}\int^{t}_{t-A}\mathscr{U}_{0}(t,s)\mathcal {V}_{2}(s)Z(s)ds
=\mathcal {V}^{1/2}_{2}(t)\chi_{3}\int^{t}_{t-A}\mathscr{U}_{0}(t,s)\mathcal {V}_{2}(s)P_{b}(s)Z(s)ds\nonumber\\
&&\ \ \ \ \ \ \ +\ \mathcal {V}^{1/2}_{2}(t)\chi_{3}\int^{t}_{t-A}\mathscr{U}_{0}(t,s)\mathcal {V}_{2}(s)P_{c}(s)\mathscr{U}_{2}(s)Z_{0}ds\nonumber\\
&&\ \ \ \ \ \ \ \ \ \ \ -\ i\mathcal {V}^{1/2}_{2}(t)\chi_{3}\int^{t}_{t-A}\mathscr{U}_{0}(t,s)\mathcal {V}_{2}(s)\int^{s-B}_{0}\mathscr{U}_{2}(s,\tau)P_{c}(\tau)\mathcal {V}_{1\epsilon}Z(\tau)d\tau ds\nonumber\\
&&\ \ \ \ \ \ \ \ \ \ \ \ \ \ -\ i\mathcal {V}^{1/2}_{2}(t)\chi_{3}\int^{t}_{t-A}\mathscr{U}_{0}(t,s)\mathcal {V}_{2}(s)\int^{s}_{s-B}\mathscr{U}_{2}(s,\tau)P_{c}(\tau)\mathcal {V}_{1\epsilon}Z(\tau)d\tau ds\nonumber\\
&&\ \ \ \ \ \ \ \ \ \ \ \ \ \ \ \ \ +\ i\mathcal {V}^{1/2}_{2}(t)\chi_{3}\int^{t}_{t-A}\mathscr{U}_{0}(t,s)\mathcal {V}_{2}(s)\int^{s}_{0}\mathscr{U}_{2}(s,\tau)P_{c}(\tau)E(\tau)Z(\tau)d\tau ds\nonumber\\
&&\ \ \ \ \ \ \ \ \ \ \ \ \ \ \ \ \ \ \ \ +i\mathcal {V}^{1/2}_{2}(t)\chi_{3}\int^{t}_{t-A}\mathscr{U}_{0}(t,s)\mathcal {V}_{2}(s)\int^{s}_{0}\mathscr{U}_{2}(s,\tau)P_{c}(\tau)F(\tau)d\tau ds
\end{eqnarray}
and
\begin{eqnarray}\label{eq-487}
&&\mathcal {V}^{1/2}_{2}(t)\chi_{3}\int^{t}_{t-A}\mathscr{U}_{0}(t,s)\mathcal {V}_{1\epsilon}Z(s)ds=\mathcal {V}^{1/2}_{2}(t)\chi_{3}\int^{t}_{t-A}\mathscr{U}_{0}(t,s)P_{c}(s)\mathcal {V}_{1\epsilon}\mathscr{U}_{1}(s)Z_{0}ds\nonumber\\
&&\ \ \ \ \ \ \ \ \ \ \ \ \ +\ \mathcal {V}^{1/2}_{2}(t)\chi_{3}\int^{t}_{t-A}\mathscr{U}_{0}(t,s)\mathcal {V}_{1\epsilon}\int^{s-B}_{0}\mathscr{U}_{1}(s,\tau)\mathcal{V}_{2}(\tau)Z(\tau)d\tau ds\nonumber\\
&&\ \ \ \ \ \ \ \ \ \ \ \ \ \ \ \ \ +\ \mathcal {V}^{1/2}_{2}(t)\chi_{3}\int^{t}_{t-A}\mathscr{U}_{0}(t,s)\mathcal {V}_{1\epsilon}\int^{s}_{s-B}\mathscr{U}_{1}(s,\tau)\mathcal{V}_{2}(\tau)Z(\tau)d\tau ds\nonumber\\
&&\ \ \ \ \ \ \ \ \ \ \ \ \ \ \ \ \ \ \ \ \ +\ \mathcal {V}^{1/2}_{2}(t)\chi_{3}\int^{t}_{t-A}\mathscr{U}_{0}(t,s)\mathcal {V}_{1\epsilon}\int^{s}_{0}\mathscr{U}_{1}(s,\tau)F(\tau)d\tau ds.
\end{eqnarray}
for some large $B\ll T$. As before, for (\ref{eq-486}) we only need to consider
\begin{eqnarray}\label{eq-488}
&&\mathcal {V}^{1/2}_{2}(t)\chi_{3}\int^{t}_{t-A}\mathscr{U}_{0}(t,s)\mathcal {V}_{2}(s)\int^{s}_{s-B}\mathscr{U}_{2}(s,\tau)P_{c}(\tau)\mathcal {V}_{1\epsilon}Z(\tau)d\tau ds\nonumber\\
&=&\mathcal {V}^{1/2}_{2}(t)\chi_{3}\int^{t}_{t-A}\mathscr{U}_{0}(t,s)K_{\leq M}\mathcal {V}_{2}(s)\int^{s}_{s-B}\mathscr{U}_{2}(s,\tau)P_{c}(\tau)\mathcal {V}_{1\epsilon}Z(\tau)d\tau ds\nonumber\\
&&\ \ +\ \mathcal {V}^{1/2}_{2}(t)\chi_{3}\int^{t}_{t-A}\mathscr{U}_{0}(t,s)K_{\geq M}\mathcal {V}_{2}(s)\int^{s}_{s-B}\mathscr{U}_{2}(s,\tau)P_{c}(\tau)\mathcal {V}_{1\epsilon}Z(\tau)d\tau ds\nonumber\\
&:=&J_{L}+J_{H}.
\end{eqnarray}
The local decay estimate for $J_{H}$ can be conclude by using (\ref{eq-479}) and the same argument as the one for $J_{2}^{H}$ in $II_{1}$ (or $J^{3,H}_{c}$ in $II_{2}$). Moreover, it follows from bootstrap assumption (\ref{eq-419}) and H\"{o}lder's inequality that
\begin{eqnarray}\label{eq-489}
\|J_{L}\|_{L^{2}_{t}L^{2}_{x}}&\leq& \Big\|\sup_{|t-s|\leq A}\epsilon A^{1/2}B^{1/2}\big\|\mathcal {V}^{1/2}_{2}(t)\chi_{3}\mathscr{U}_{0}(t,s)K_{\leq M}\mathcal {V}_{2}^{1/2}(s)\big\|_{L^{2}_{x}\rightarrow L^{2}_{x}}\Big\|_{L^{2}_{t}}\big\|\mathcal{V}_{2}^{1/2}(\tau)Z(\tau)\big\|_{L^{2}_{\tau}L^{2}_{x}}\nonumber\\
&\lesssim&\big\|A^{1/2}B^{1/2}\langle t\rangle^{-1}\big\|_{L^{2}_{t}}\big\|\mathcal{V}_{2}^{1/2}(\tau)Z(\tau)\big\|_{L^{2}_{\tau}L^{2}_{x}}\nonumber\\
&\leq& \frac{1}{2}C(T)\big(\|Z_{0}\|_{L^{2}_{x}}+\big\|F\big\|_{L^{\tilde{p}'}_{t}L^{\tilde{q}'}_{x}}\big),
\end{eqnarray}
where we use the estimate
\begin{eqnarray}
\sup_{|t-s|\leq A}\epsilon\big\|\mathcal {V}^{1/2}_{2}(t)\chi_{3}\mathscr{U}_{0}(t,s)K_{\leq M}\mathcal {V}_{2}^{1/2}(s)\big\|_{L^{2}_{x}\rightarrow L^{2}_{x}}\lesssim \langle t\rangle^{-1}\nonumber
\end{eqnarray}
and the argument mentioned in (\ref{eq-441}) that one only need to prove (\ref{eq-419}) for $\int^{T}_{\tilde{M}}$ for some large positive constant $\tilde{M}\ll T$, therefore we can choose $t$ large enough such that $A^{1/2}B^{1/2}t^{-\varepsilon}$ are small for any $\varepsilon>0$, as well as $t^{-(1-\varepsilon)}\in L^{2}_{t}$. The corresponding term
$$\mathcal {V}^{1/2}_{2}(t)\chi_{3}\int^{t}_{t-A}\mathscr{U}_{0}(t,s)\mathcal {V}_{1\epsilon}\int^{s}_{s-B}\mathscr{U}_{1}(s,\tau)\mathcal{V}_{2}(\tau)Z(\tau)d\tau ds$$
in (\ref{eq-487}) can be dealt similarly.
Hence we finish the proof.

\end{proof}

{\bf Acknowledgements:} The first author is supported by NSFC (No. 11661061 and No. 11671163). The second author is partially supported by a grant from the Simons
Foundation (395767 to Avraham Soffer). A. Soffer is partially supported by NSF grant
DMS01600749 and NSF DMS-1201394. The third author is supported by NSFC (No. 11371158) and the program for Changjiang Scholars and Innovative Research Team in University (No. IRT13066). Part of this work was done while the second author was Visiting Professor at Central China Normal Univ. (CCNU), China.  Finally, we would like to thank Professor G.S. Perelman for her interests and helpful discussions in this work.


\begin{thebibliography}{99}
\bibitem{ASY}J. E. Avron, R. Seiler  L. G. Yaffe, \emph{Adiabatic theorems and applications to
the quantum hall effect,} Comm. Math. Phys. 110 (1987) 33-49.

\bibitem{S-S} W. K. Abou Salem and C. Sulem, \emph{Resonant tunneling of fast solitons through
large potential barriers,} Canad. J. Math. 63 (2011), 1201-1219.

\bibitem{AS}W. K. Abou Salem, \emph{Solitary wave dynamics in time-dependent potentials,} J. Math. Phys. 49 (2008), 032101.




\bibitem{AFS}W. K. Abou Salem, J. Fr\"{o}hlich and I. Sigal, \emph{Colliding solitons for the non-linear Schr\"{o}dinger equation,} Comm. Math. Phys. 291 (2009), 151-176.

\bibitem{Bam1}D. Bambusi and A. Maspero, \emph{Freezing of energy of a soliton in an external potential},  Comm. Math. Phys. 344 (2016), 155-191.



\bibitem{Bec3}M. Beceanu, \emph{New estimates for a time-dependent Schr\"{o}dinger equation,} Duke Math. J. 159 (2011), 417-477.

\bibitem{Be-G}M. Beceanu and M. Goldberg, \emph{Schr\"{o}dinger dispersive estimates for a scaling-critical class of potentials,} Comm. Math. Phys. 314 (2012), 471-481.





\bibitem{BJ}J. Bronski and R. Jerrard, \emph{Soliton dynamics in a potential,} Math. Res. Lett. 7 (2000), 329-342.

 \bibitem{Cai}K. Cai, \emph{Fine properties of charge transfer models}, arXiv: math/0311048v1.

\bibitem{Chen}G. Chen, \emph{Strichartz Estimates for Charge Transfer Models},  arXiv:1507.07644.


\bibitem{CL}C. Cote and S. Le Coz,  \emph{High-speed excited multi-solitons in nonlinear
Schr¡§odinger equations,} J. Math. Pures Appl. 96 (2011), 135-166.


\bibitem{CucMa1}S. Cuccagna and M. Maeda, \emph{On weak interaction between a ground state and a non-trapping potential}, J. Differ. Equ. 256 (2014), 1395-1466.



\bibitem{DH}K. Datchev and J. Holmer, \emph{Fast soliton scattering by attractive delta impurities,}
Comm. Partial Diff. Equ. 34 (2009), 1074-1113.

\bibitem{DSY}Q. Deng, A. Soffer and X. Yao, \emph{Endpoint Strichartz estimates for charge transfer Hamiltonians,} arXiv:1507.03870v2.


\bibitem{FJGS}J. Fr\"{o}hlich, S. Gustafson, B. Jonsson and I. Sigal, \emph{Solitary wave dynamics in an external potential}, Comm. Math. Phys. 250 (2004), 613-642.

\bibitem{FJGS-1}J. Fr\"{o}hlich, S. Gustafson, B. Jonsson and I. Sigal, \emph{Long time motion of NLS
solitary waves in a confining potential,} Ann. H. Poinc. 7 (2006), 621-660.

\bibitem{GB} M. Goldberg, \emph{Dispersive bounds for the three-dimensional Schr¡§odinger equation with almost
critical potentials,} Geom. and Funct. Anal., 16 (2006), 517-536.

\bibitem{GHW}R. Goodman, P. Holmes and  M. I. Weinstein, \emph{Strong NLS soliton-defect interactions,}
Phys. D 192 (2004), 215-248.

\bibitem{Gr}J. Graf, \emph{Phase Space Analysis of the Charge transfer Model}, Helv. Physica Acta 63 (1990), 107-138.


\bibitem{GNP}S. Gustafson, K. Nakanishi and T. Tsai, \emph{Asymptotic stability and completeness in the energy space for nonlinear Schr\"{o}dinger equations with small solitary waves,} Int. Math. Res. Not. 66 (2004), 3559-3584.






 \bibitem{HM1}J. Holmer and M. Zworski, \emph{Slow soliton interaction with delta impurities,} J.
Mod. Dyn. 1 (2007), 689-718.

\bibitem{HM}J. Holmer and  M. Zworski, \emph{Soliton interaction with slowly varying potentials,} Internat. Math. Res. Notices  (2008), Art. ID runn026, 36 pp.

\bibitem{HMZ1}J. Holmer, J. Marzuola and M. Zworski, \emph{Soliton splitting by external delta
potentials,} J. Nonlinear Sci. 17 (2007), 349-367.

\bibitem{HMZ2} J. Holmer, J. Marzuola and M. Zworski, \emph{Fast soliton scattering by delta impurities,}
Comm. Math. Phys. 274 (2007), 187-216.









\bibitem{JSS1}J. Journ\'{e},  A. Soffer and C. Sogge, \emph{Decay estimates for Schr\"{o}dinger operators,} Comm. Pure Appl. Math. 44 (1991), 573-604.

\bibitem{Kato1}T. Kato, \emph{On the Adiabatic Theorem of Quantum Mechanics,} J. Phys. Soc. Japan (1950), 435-439.




\bibitem{MM1} Y. Martel and F. Merle, \emph{Multi solitary waves for nonlinear Schr\"{o}dinger equations,} Ann. Inst. H. Poincar\'{e}. Anal. Non Lin. 23 (2006), 849-864.

\bibitem{MMT} Y. Martel, F. Merle and T. Tsai, \emph{Stability in $H^{1}$ of the sum of K solitary waves for some nonlinear Schro\"{o}dinger equations,} Duke Math. J. 133 (2006), 405-466.


\bibitem{Per2}G. Perelman, \emph{Some results on the scattering of weakly interacting solitons for nonlinear Schr\"{o}dinger equation.} In: Demuth et al., M., eds. Spectral Theory, Microlocal Analysis, Singular Manifolds. Math. Top. 14. Berlin: Akademie Verlag, (1997), 78-137.


\bibitem{Per1}G. Perelman, \emph{Asymptotic stability of multi-soliton solutions for nonlinear Schr\"{o}dinger equations}, Comm. Partial Diff. Equ. 29 (2004), 1051-1095.

\bibitem{Per5}G. Perelman, \emph{A remark on soliton-potential interactions for nonlinear Schr\"{o}dinger equations,}  Math. Res. Lett. 16 (2009), 477-486.

\bibitem{Per4} G. Perelman, \emph{Two soliton collision for nonlinear Schr\"{o}dinger equations in dimension 1}, Ann. Inst. H. Poincar\'{e}. Anal. Non Lin. 28 (2011), 357-384.



\bibitem{RSS1}I. Rodnianski, W. Schlag and A. Soffer, \emph{Dispersive analysis of charge transfer models,} Comm. Pure Appl. Math. 58 (2005), 149-216.

\bibitem{RSS2}I. Rodnianski, W. Schlag and A. Soffer, \emph{Asymptotic stability of N-soliton states of NLS,}  arXiv:math/03091114v1.





\bibitem{Wu}U. W\"{u}ller, \emph{Geometric Methods in Scattering Theory of the Charge Transfer Model}, Duke Math J.  62 (1991), 273-313.

\bibitem{Ya1}K, Yajima, \emph{A multichannel scattering theory for some time dependent Hamiltonians, charge transfer problem,} Comm. Math. Phys. 75 (1980),153-178.

\bibitem{Zi}L. Zielinski, \emph{Asymptotic completeness for multiparticle dispersive charge transfer models}, J. Funct. Anal. 150 (1997), 453-470.

\bibitem{SZ}G. Zhou and I. Sigal, \emph{Relaxation of solitons in Nonlinear Schr\"{o}dinger equations with potential,} Adv Math.  216 (2007), 443-490.




\end{thebibliography}
\end{document}